\documentclass[12pt,a4paper]{article}
\usepackage[english]{babel}

\usepackage{mathptmx} 
\usepackage[scaled=0.90]{helvet} 
\usepackage{courier} 
\normalfont
\usepackage[T1]{fontenc}
\usepackage{a4wide}
\usepackage{dsfont}
\usepackage{latexsym}
\usepackage{enumerate}
\usepackage{amssymb}
\usepackage{amsmath,amscd}
\usepackage{amsthm}
\usepackage{amsfonts}
\usepackage{srcltx}

\newtheorem{thm}{Theorem}[section]
\newtheorem{prop}[thm]{Proposition}
\newtheorem{lemma}[thm]{Lemma}
\newtheorem{corollary}[thm]{Corollary}

\newtheorem{defi}[thm]{Definition}

\theoremstyle{remark}
\newtheorem{remark}[thm]{Remark}

\newtheoremstyle{dotless}{}{}{\itshape}{}{\bfseries}{}{ }{}
\theoremstyle{dotless}
\newtheorem*{mmtheorem}{Main Theorem}

\allowdisplaybreaks

\def\cP{{\mathcal P}}
\def\cA{{\mathcal A}}
\def\gs{\sigma}
\def\Cartan{\theta}
\def\gS{{\Sigma}}
\def\faq{{\mathfrak a}_{\rm q}}
\def\fah{{\mathfrak a}_{\rm h}}
\def\faqd{{\mathfrak a}_{\rm q}^*}
\def\fahd{{\mathfrak a}_{\rm h}^*}
\def\ih{{\rm h}}
\def\iq{{\rm q}}
\def\ga{\alpha}
\def\gb{\beta}
\def\Aq{A_\iq}
\def\fa{{\mathfrak a}}
\def\faq{\fa_{\rm q}}
\def\fg{{\mathfrak g}}
\def\fq{{\mathfrak q}}
\def\fh{{\mathfrak h}}
\def\fH{{\mathfrak H}}
\def\fp{\mathfrak{p}}
\def\fk{\mathfrak{k}}
\def\ft{\mathfrak{t}}
\def\HRPq{\mathfrak{H}^R_{P,{\rm q}}}
\def\after{\circ}
\def\faqreg{\fa_\iq^{\rm reg}}
\def\Aqreg{A_{\iq}^{\rm reg}}
\def\itema{\item[{\rm (a)}]}
\def\itemb{\item[{\rm (b)}]}
\def\itemc{\item[{\rm (c)}]}
\def\itemd{\item[{\rm (d)}]}
\def\Phext{Q_{\rm h}}
\def\Qh{Q_{\rm h}}
\def\fl{{\mathfrak l}}
\def\fn{{\mathfrak n}}
\def\HP{\mathfrak{H}_P}
\def\prq{\textrm{pr}_{{\rm q}}}
\def\HPq{\mathfrak{H}_{P,{\rm q}}}
\def\R{\mathbb{R}}
\def\fm{{\mathfrak m}}
\def\starfaR{{}^*\fa_R}
\def\pr{{\rm pr}}
\def\Ad{{\rm Ad}}
\def\C{\mathbb{C}}
\def\d{{\rm d}}
\def\nb{{\mathfrak b}}

\newcommand{\inp}[2]{\langle #1,#2\rangle}
\def\dotvar{\cdot}
\def\Ad{{\rm Ad}}
\def\ad{{\rm ad}}
\def\cC{\mathcal{C}}
\def\WKH{W_{K\cap H}}
\def\rmH{{\rm H}}
\def\Aut{{\rm Aut}}
\def\cO{{\mathcal O}}
\def\fad{\fa^*}
\def\fv{{\mathfrak v}}
\def\WKaq{W(\faq)}
\def\WKa{W(\fa)}
\def\WKHXaq{W_{K\cap H_X}}

\def\conv{{\rm conv}}
\def\gl{\lambda}
\def\rk{{\rm rk}}

\def\gL{\Lambda}
\def\fc{{\mathfrak c}}

\def\bpP{{\bp \! P}}
\def\bpG{{\bp G}}
\def\bpA{{\bp \! A}}
\def\reg{{\rm reg}}
\def\bpAqreg{{\bp \! A_\iq^\reg}}
\def\codim{{\rm codim}}
\def\conv{{\rm conv}}
 \def\dotvar{\, \cdot\,}
 \def\inp#1#2{\langle #1 , #2 \rangle}
\def\setmid{\,:\;\;}
\def\cM{{\mathcal M}}
 \def\barHPq{{\overline{\mathfrak H}}_{P,{\rm q}}}
 \def\bp{{}^\backprime}
 \def\WKHaq{W_{K\cap H}}
 
 \def\half{{\textstyle\frac12}}
 \def\End{{\rm End}}
 \def\GammaP{\Gamma(P)}
 \def\GammaPX{\Gamma(P_X)}
 \def\cN{{\mathcal N}}

\begin{document}
\title{Convexity theorems for semisimple symmetric spaces}
\author{Dana B\u alibanu and Erik van den Ban}
\date{March 9, 2015}
\maketitle
\begin{abstract}
We prove a convexity theorem for semisimple symmetric spaces $G/H$ which generalizes
an earlier theorem of the second named author to a setting without restrictions
on the minimal parabolic subgroup involved. The new more general result specializes to Kostant's 
non-linear convexity theorem for a real semisimple Lie group $G$ in two ways, firstly by 
taking $H$ maximal compact and secondly 
by viewing $G$ as a symmetric space for $G \times G.$
\end{abstract}
\section{Introduction}\label{s: introduction}
            In this paper we prove a generalization of the convexity theorem in \cite{ban1986} for a symmetric space $G/H.$
             Here     $G$     is a connected semisimple Lie group with finite center,  $\sigma$     an involution of     $G$ and $H$      an open subgroup of the group    $G^{\sigma}$ of fixed points     for $\sigma$. The generalization involves Iwasawa decompositions related to minimal parabolic subgroups of $G$ of arbitrary type instead of the particular type of parabolic subgroup considered in \cite{ban1986}.

            From now on we assume     more generally that $G$ is a real reductive group of the Harish-Chandra class; this will allow an inductive argument relative to the real rank of $G.$
            Let $\theta:G\to G$ be a Cartan involution of      $G$ that commutes with $\sigma$;
            for its existence, see      \cite[Thm. 6.16]{knapp2002}. The associated group $K:= G^\Cartan$ of fixed points is a maximal compact subgroup of $G$.
            For the infinitesimal involutions determined by $\sigma$ and $\theta$ we use the same      symbols: $\theta,\sigma:\mathfrak{g}\to\mathfrak{g}$;
            here $\fg$ denotes the Lie algebra of $G.$ With respect to      the infinitesimal involutions, $\mathfrak{g}$ decomposes as
            \[\mathfrak{g}=\mathfrak{k}\oplus\mathfrak{p}=\mathfrak{h}\oplus\mathfrak{q},\]
            where $\mathfrak{k}$ and $\mathfrak{p}$ are the +1 and -1 eigenspaces for $\theta$ and likewise, $\mathfrak{h}$ and $\mathfrak{q}$ are the +1 and -1 eigenspaces     for  $\sigma$. Note that $\mathfrak{k}$ is the Lie algebra of $K$ and $\mathfrak{h}$ is the Lie algebra of $H$.

            Since $\sigma$ and $\theta$ commute, their composition $\sigma\theta$ is again an involution  of  $\mathfrak{g}$. With respect to      the latter involution, $\mathfrak{g}$ decomposes      into eigenspaces
            \[\mathfrak{g}=\mathfrak{g}_+\oplus\mathfrak{g}_-.\]
            Observe that the +1 eigenspace $\mathfrak{g}_+$ equals $\mathfrak{k}\cap\mathfrak{h}\oplus\mathfrak{p}\cap\mathfrak{q}$, while
             the -1 eigenspace $\mathfrak{g}_-$ equals $\mathfrak{k}\cap\mathfrak{q}\oplus\mathfrak{p}\cap\mathfrak{h}$.

            We fix a maximal abelian subspace $\fa_\iq$ of $\mathfrak{p}\cap\mathfrak{q}$, and  a
            a maximal abelian subspace      $\mathfrak{a}$ of $\mathfrak{p}$ that contains $\fa_\iq$. Then $\fa$ is $\gs$-stable
            and decomposes as:
            \begin{equation}\label{e: a-decomp}
            \mathfrak{a}=\fa_\ih\oplus\fa_\iq,
            \end{equation}
            where      $\fa_\ih:= \mathfrak{a}\cap\mathfrak{h}$.      The associated projection onto $\faq$ will be denoted by  $\prq:\mathfrak{a}\to\fa_\iq.$

            Let $\Sigma(\mathfrak{g},\mathfrak{a})$ be the set of roots of $\mathfrak{a}$ in $\mathfrak{g}$ and $\Sigma(\mathfrak{g},\fa_\iq)$ the set of roots of $\fa_\iq$ in $\mathfrak{g}$. Both      of these sets form root systems, possibly non-reduced, see e.g. \cite{ross1979} or \cite{knapp2002}. The associated Weyl groups are given by
            \begin{equation}
            \label{e: Weyl groups}
            \WKa = N_K(\fa)/Z_K(\fa)\qquad {\rm and} \quad \WKaq = N_K(\faq)/Z_K(\faq).
            \end{equation}

            Let $A=\exp\mathfrak{a}$ and      let $\mathcal{P}(A)$ be the set of minimal parabolic subgroups of $G$ containing $A$. If $P\in\mathcal{P}(A)$, then $P$ has a unique Langlands decomposition given by
            \begin{equation}
            \label{e: Langlands deco}
            P=MAN_P,
            \end{equation}
            where $M:=Z_K(\mathfrak{a})$, $N_P=\exp\mathfrak{n}_P$ and $\mathfrak{n}_P$ is the sum of the root spaces corresponding to a      uniquely determined positive system of $\Sigma(\mathfrak{g},\mathfrak{a})$. We denote this positive system by $\Sigma(P)$. The map given by
            \[P\mapsto\Sigma(P),\quad P\in\mathcal{P}(A),\]
            defines a bijection between $\mathcal{P}(A)$ and the set of positive systems of $\Sigma(\mathfrak{g},\mathfrak{a})$.

            Let
   \[
    \Sigma(P,\sigma\theta):=\{\alpha\in\Sigma(P) \setmid \sigma\theta\alpha\in\Sigma(P)\}
    \]
     and
    \begin{equation}
    \label{e: defi of Sigma P minus}
    \Sigma(P)_-:=\{\alpha\in\Sigma(P,\sigma\theta)\setmid\sigma\theta\alpha=
    \alpha\;\Longrightarrow\;\sigma\theta|_{\mathfrak{g}_{\alpha}}\neq {\rm id}_{\mathfrak{g}_{\alpha}}\}.
    \end{equation}

            Any parabolic subgroup $P\in\mathcal{P}(A)$ induces an Iwasawa decomposition
            \[
            G\simeq K\times A\times N_P
            \]
            and      the associated infinitesimal      decomposition
            $  \mathfrak{g}=\mathfrak{k}\oplus\mathfrak{a}\oplus\mathfrak{n}_P.$
            The \emph{Iwasawa projection} corresponding to $P$ is defined as the real analytic map
            \begin{equation}
            \label{e: Iwasawa projection}
            \mathfrak{H}_P:G\longrightarrow\mathfrak{a},\quad\textrm{determined by}\quad g\in K\exp\mathfrak{H}_P(g)N_P,\quad(g\in G).
            \end{equation}

             The main result of \cite{kostant1973},  known as `Kostant's (nonlinear) convexity theorem' characterizes the image under $\fH_P$ of the set $aK$, for $a\in A,$ as follows:
            \[\fH_P(aK)={\rm conv}(\WKa\cdot\log a).\]
           Here '$\conv$' indicates that the convex hull in $\fa$ is taken.

                 In our setting, it is natural to study the more general question of convexity of the set $\fH_P(aH)$, for $a\in A.$
             The first answer to this question was provided in      \cite[Thm. 1.1]{ban1986} under the assumption that      $P \in \cP(A)$ satisfies $\Sigma(P,\sigma\theta)=\Sigma(P)\setminus\fa_\ih^*$;
              here $\fahd$ and $\faqd$ are viewed as subspaces of $\fad$ in accordance with the decomposition
             (\ref{e: a-decomp}).

             In the present paper we generalize  \cite[Thm 1.1]{ban1986} to any parabolic subgroup $P \in \cP(A).$
             To prepare for our main result, we need a few remarks and definitions as well as new notation.
            \begin{remark}
            Since $\exp: \fa \to A$ is a diffeomorphism, it follows that  $A\simeq A_\iq\times A_\ih$ where $\Aq:= \exp (\faq)$ and $A_\ih:=\exp(\fah) = A\cap H$.
            Thus, we just need to consider $a\in A_\iq$.
            \end{remark}
            \begin{remark}
            \label{r: defi HPq}
                Since $\fH_P(aH)= \fH_{P}(a H A_\ih) = \fH_{P}(aH)+\fah$, it suffices to consider the image of $aH$,      for $a\in A_\iq$, under the map \[\fH_{P,{\rm q}}:=\prq\circ\fH_P:G\to\faq.\]
            \end{remark}

We recall from \cite[Eqn.\ (1.2)]{ban1986} that the subgroup $H$ is said to be {\em essentially connected}
if
\begin{equation}
\label{e: essentially connected}
H=Z_{K\cap H}(\fa_\iq)H^{\circ},
\end{equation}
where $H^\circ$ denotes the identity component of $H.$

            Let $B$ be  an extension of the Killing form from  $[\fg,\fg]$ to a      bilinear form on the entire algebra $\fg$, such that
            $B$ is  $\Ad(G)$-invariant, invariant under both $\Cartan$ and $\sigma,$ negative definite on $\fk$ and positive definite on $\fp.$ Then $B$ is non-degenerate.

            We define a positive definite inner product on $\mathfrak{g}$ by
            \begin{equation}
            \label{e: inner product from B}
            \inp{U}{V}:=-B(U,\Cartan V),\qquad (U,V\in\fg).
            \end{equation}
            Note that the root space  decomposition and the eigenspace decompositions (with respect to $\theta$ and $\sigma$) are orthogonal with respect to this inner product. Moreover, the extended Killing form and the inner product coincide if either $U$ or $V$ belongs to $\fp$.
            \begin{defi}\label{d: WKHaq}
                    The Weyl group $\WKHaq$ is defined as
                     \[\WKHaq:=N_{K\cap H}(\fa_\iq)/Z_{K\cap H}(\fa_\iq).\]
            \end{defi}
              Note that $\WKHaq$ may be viewed as a subgroup of $\WKaq.$
            If     $\alpha$ is a root in $\Sigma(\mathfrak{g},\mathfrak{a})$ we denote by $H_{\alpha}$ the element of $\fa$      perpendicular to $\ker\alpha$ with respect to $\langle\cdot,\cdot\rangle,$ and normalized by 
             $\alpha(H_{\alpha})=2.$
             \begin{defi}
            \label{d: Gamma P}
            Let $P$ be a minimal parabolic subgroup of $G$ containing $A.$ Then we define the finitely generated polyhedral cone $\Gamma(P)$ in $\faq$ by
            \begin{equation}
            \label{e: defi Gamma P}
            \Gamma(P):=  \sum_{\alpha\in\Sigma(P)_-}\R_{\geq 0}\;\pr_\iq(H_{\alpha}).
            \end{equation}
            \end{defi}

            \begin{mmtheorem}[Theorem \ref{t: conv thm new}]
                         Let   $H$ be an essentially connected open subgroup of $G^{\sigma},$ see (\ref{e: essentially connected}).
                          Let $P$ be any minimal parabolic subgroup     of $G$
                     containing $A$ and let $a\in A_\iq$. Then
                    \[
                    \mathfrak{H}_{P,{\rm q}}(aH)=\conv(\WKHaq \cdot\log a)+\Gamma(P).
                    \]
            \end{mmtheorem}

            If     the two involutions $\sigma$ and $\theta$ are equal, then $K=H$ and $\Sigma(P,\sigma\theta)=\Sigma(P)$. This implies that $\WKa=\WKHaq$ and that $\Sigma(P)_-=\emptyset$. Thus, we obtain that $\Gamma(P)=0$ and hence,      in this case our main theorem coincides with the original non-linear
            convexity theorem of Kostant \cite{kostant1973}.

            For $P$ satisfying $\gS(P, \gs \Cartan) = \gS(P)\setminus \fahd$ the above result
            coincides with  \cite[Thm 1.1]{ban1986}. This will be explained in detail in Section
            \ref{ss: relation with vdbconv}.

            The proof of the main theorem follows the   line of argument described  below,
            which is a considerable extension of the argumentation of \cite{ban1986}, which in turn was inspired by  \cite{heckman1980}.

            We first prove the theorem for a regular element $a\in A_\iq$.
            Since the map $\mathfrak{H}_P:G\to\mathfrak{a}$ is right $H\cap P$-invariant, see Lemma \ref{l: HP-inv}, the map
            $$
            F_a:H\to\fa_\iq, \;\;h \mapsto \mathfrak{H}_{P,{\rm q}}(ah)
            $$
            factors through a map $\bar{F_a}:H/H\cap P\to\fa_\iq$. In order for the idea of the proof in \cite{ban1986} to work in the present situation, one needs to establish properness of the map $\bar{F_a}$. This is done in Section \ref{s: prop-sec} by reducing the problem to the case of a   suitable
             $\sigma$-stable parabolic subgroup $R$ combined with application of results of
              \cite{ban1986}. The      established properness implies that the image $F_a(H)$ is closed in $\fa_\iq$.

             The considerations of Section \ref{s: prop-sec} also lead to the constraint on
             the image $F_a(H)$ that it does not contain any line of $\faq,$ see Corollary \ref{c: half space result}.

            In Section \ref{s: the set critical points} we introduce the functions $F_{a,X}:H\to\R$,      for $X \in \faq,$
            defined by
            $$
            F_{a,X}(h)=\langle X,F_{a}(h)\rangle=B(X,F_{a}(h)).
            $$
            Geometrically,      these functions test the Iwasawa projection by linear forms on $\fa_\iq$, and  give us  constraints on  the image of $H$ under $F_a$. For a more detailed exposition on $F_{a,X}$ we refer the reader to \cite{duiskolkvar1983}. Our own study of this function follows ideas in \cite{ban1986} and \cite{duiskolkvar1983}.

            In Section \ref{s: properties of the set of critical points} we calculate the critical set $\cC_{a,X}$ of the function $F_{a,X}$ explicitly, for $a \in \Aqreg$ and $X\in\fa_\iq.$ In particular, we show that this set  is the     union of a     finite collection $\cM_{a,X}$
            of      injectively immersed connected submanifolds of $H$. If $\cC_{a, X}\subsetneq H,$ then all submanifolds in $\cM_{a,X}$ are lower dimensional, so that $\cC_{a,X}$ is thin in the sense of the Baire theorem, i.e. its closure has empty
            interior.            These considerations allow us to show that in case $\gS(\fg, \faq)$ spans $\faqd,$ the set  $\cC_a$ of points in $H$ where $F_a$ is not submersive, is closed and thin, see  Proposition \ref{p: submersiveness Fa}. In particular, we then have  that
           \begin{equation}
           \label{e: inclusion of Fa Cca}
           F_a(\cC_a)\subsetneq F_a( H).
           \end{equation}

In Sections  \ref{s: computation of hessians} and \ref{s: transversal signature of the hessian} we calculate the Hessians of $F_{a,X}$ and their transversal signatures along all manifolds from $\cM_{a,X}.$ These calculations, which are extensive, in particular
allow us to determine all points where the transversal signatures are definite. This  in turn gives us all points where $F_{a,X}$ attains local maxima and minima. A main result of Section 8 is Lemma \ref{l: local-min-global-min} which asserts that for every local minimum $m$ of the function $F_{a,X}$ we have that
$\inp{X}{\dotvar}\geq m$ on the set
 \begin{equation}
 \label{e: introduction set Omega}
              \Omega:=\conv(\WKHaq \cdot\log a)+\Gamma(P).
 \end{equation}

 In Section \ref{s: limit argument} we prepare for the proof of the main theorem
 by using a limit argument to reduce  to the case of  a regular element $a \in \Aq.$

The proof of the    main theorem is   finally given in Section \ref{s: section-induction}.
It proceeds by induction over the rank of the root system $\Sigma(\mathfrak{g},\fa_\iq)$.     More precisely, for $a \in \Aqreg$ the set $\cC_{a,X}$
            depends on $X \in \faq$ through the centralizer $\fg_X$ of $X$ in $\fg.$ It is shown that $\cC_{a,X}\subsetneq  H$ implies that
            $\rk\, \Sigma(\mathfrak{g}_X,\fa_\iq) < \rk \, \Sigma(\mathfrak{g},\fa_\iq)$ so that            the induction hypothesis holds for the centralizer $G_X$ of $X$ in $G.$
            This allows us to determine the image $F_a(C_{a,X})$ for such $X.$ In particular,
            this leads to a precise description of the image $F_a(\cC_a)$     from which it is seen that the latter image contains the boundary of the set $\Omega.$

              In the proof  we use this observation, together with the earlier obtained      constraint that the image $F_a(H)$ does not contain a line, to conclude that $F_a(H)$ is contained in $\Omega$. In particular, this implies that,  for each $X \in \faq,$  every local minimum of $F_{a,X}$ is global.

              For the converse inclusion, we first show that the image of $H\setminus\cC_a$ under the map $F_a$ is a union of connected components of
              $\Omega\setminus F_a(\cC_a)$. The established fact that every local minimum of
              $F_{a,X}$ is global then allows us to show that all connected components appear in the image, thereby completing the proof.

              We conclude the paper with two appendices, \ref{a: proof of KDV} and \ref{a: group case}.
              In Appendix \ref{a: proof of KDV} we give the proof of Lemma \ref{l: KDV} concerning the decomposition of nilpotent groups in terms of subgroups generated by roots, and in \ref{a: group case} we discuss the convexity theorem for the case of the group viewed as a symmetric space.

                       Both the linear and the nonlinear convexity theorems of Kostant, see \cite{kostant1973}, have been extensively studied. Heckman proved the linear theorem in \cite{heckman1980} by means of techniques as above and  obtained the non-linear theorem from the linear one by       a homotopy argument.  Inspired by this, Duistermaat \cite{duistermaat1984} obtained a remarkable universal homotopy containing Heckman's homotopy for all $a \in A$ at once.

             Both convexity theorems of Kostant have been explained in the framework of symplectic geometry: see \cite{atiyah1982}, \cite{guilleminsternberg1982}, \cite{duistermaat1983-conv}, \cite{guilleminsternberg1984}, \cite{Kirwan1984} for the linear convexity theorem and \cite{luratiu1991}, \cite{hilgertneeb1998} for the nonlinear one.

             The convexity theorem of \cite{ban1986}, which generalizes  Kostant's nonlinear convexity theorem, has been given a symplectic interpretation in \cite{fothotto2006}. This leads us to suspect that such an interpretation should be possible in the present case as well; we intend to investigate this in the future.

            Finally, we wish to mention that many of our calculations have been inspired by \cite{heckman1980} and \cite{duiskolkvar1983}.\\
            \\
            \emph{Acknowledgements.\ } The authors would like to thank Job Kuit for the proof of Proposition \ref{p: N=N_+(N_P cap H) decomposition} and his useful comments. One of the authors, D.B., thanks Ioan M\v arcu\c t for his interest in this work and all his good suggestions on how to improve it.

    \section{Some structure theory for parabolic subgroups}\label{s: parabolic subgroups}
                In this section we will construct a (minimal) parabolic subgroup in $\cP(A),$      see the text preceding (\ref{e: Langlands deco}),
            which has a special position relative to the involution $\gs;$ it will
            play an important role in Section \ref{s: prop-sec}. We will also discuss some structure theory of parabolic subgroups from $\cP(A)$ and      derive a useful decomposition for their unipotent radicals.

            We recall that every parabolic subgroup $P$ from $\cP(A)$ has a Langlands decomposition of the form
            (\ref{e: Langlands deco}). Thus, its ($\Cartan$-stable) Levi component $L_P$ is given by
            $$
            L_P = L = MA
            $$
            and the multiplication map $L\times N_P \to P$ is a diffeomorphism. The opposite parabolic subgroup $\bar P$ is defined to be the unique
            parabolic subgroup from $\cP(A)$ with $\gS(\bar P)= -\gS(P).$ It equals $\Cartan (P).$

          \subsection{Extremal minimal parabolic subgroups}\label{H-mod-HP-decomposition}

                If $\tau$ is any involution      of $G$ which leaves $A$ invariant, then its infinitesimal version $\tau: \fg \to \fg$ leaves $\fa$ invariant,
                and we put
                \begin{equation}
                \label{e: defi Sigma P tau}
                \Sigma(P,\tau):=\{\alpha\in\Sigma(P)\setmid\tau\alpha\in\Sigma(P)\}.
                \end{equation}
                Observe that $\Sigma(P,\tau)=\Sigma(P)\cap\tau\Sigma(P).$

                \begin{defi}
                \label{d: h extreme parabolic}
                    \rm A minimal parabolic subgroup $Q\in \cP(A)$ is said to be $\fh$-extreme if
                    \begin{equation}\label{e: h extreme}
                        \gS(Q,\sigma) = \gS(Q)\setminus \faqd.
                    \end{equation}
                \end{defi}

                Starting with any minimal parabolic subgroup $P\in\cP(A)$, we can obtain an $\fh$-extreme minimal parabolic subgroup      by changing one simple root at a time. This process is described in Lemma \ref{l: Q-existence} below.

                \begin{lemma}\label{l: gS(P) as disjoint union}
                    Let $P \in \cP(A).$ Then
                    $$\gS(P) = \gS(P, \gs) \sqcup \gS(P, \gs \Cartan)\qquad \mbox{\rm (disjoint union)}.$$
                \end{lemma}

                \begin{proof}
                    Let $\ga \in \gS(P).$ From the fact that $\sigma\theta\alpha=-\sigma\alpha$, the result follows easily.
                \end{proof}

                \begin{lemma}\label{c: if-not-sigma-stable-then-one-simple-root}
                    Let $P \in \cP(A)$ and assume that
                    \begin{equation}
                    \label{e: gS P gs properly contained}
                    \Sigma(P,\sigma)\varsubsetneq\Sigma(P)\setminus\faqd.
                    \end{equation}
                    Then there exists a $P$-simple root $\alpha\in\Sigma(P,\sigma\theta)$ with $\ga \notin \faqd$.
                \end{lemma}
                \begin{remark}
                    A root $\alpha\in\Sigma(\fg,\fa)$ is said to be $P$-simple if it is simple in the positive system $\Sigma(P)$.
                \end{remark}
                \begin{proof} Assume the contrary.
                    Then each $P$-simple root $\beta \in \gS(P, \gs \Cartan)$
                    satisfies $\gs \Cartan \beta = \beta.$ In view of Lemma \ref{l: gS(P) as disjoint union}
                    it follows that for every simple root $\gb \in \gS(P)$ we
                  have either $\gs \beta \in \gS(P)$ or $\gs \beta = \Cartan \beta = - \beta.$

                    The set $\gS(P)$ is a positive system for the root system $\gS(\fg, \fa).$ Hence, there exists an element $X \in \fa$ such that $\ga(X) > 0$ for all
                    $\ga \in \gS(P).$ Put $X_{\rm h}: = \frac12(X + \gs(X)).$ Then for every simple root $\gb$ in $\gS(P)$ we have either  $\gs \gb = - \gb,$ in which case $\gb(X_\ih) = 0,$ or $\gs\gb \in \gS(P),$ in which case $\gb(X_\ih) > 0.$ In any case, for each simple $\gb \in \gS(P),$ the value $\gb(X_\ih)$ is a nonnegative real number. Moreover, the number is zero if and only if $\gs \gb = - \gb.$ It follows that for all $\ga \in \gS(P)$ the number  $\ga (X_\ih)$ is nonnegative
                    and furthermore, that it is zero if and only if $\ga \in \faqd.$
                    Since $\gs \Cartan (X_\ih) = - X_\ih$ we now
                    infer that $\gS(P)\setminus \faqd \cap \gS(P,\gs\Cartan) = \emptyset,$
                    hence $\gS(P)\setminus \faqd \subseteq\gS(P,\gs),$
                    contradicting (\ref{e: gS P gs properly contained}).
                \end{proof}

                For a root $\ga \in \gS(\fg,\fa),$ the associated reflection is denoted by $s_\ga: \fa \to \fa.$

                \begin{corollary}\label{c: the-psgp-s-alpha-P}
                    If $P$ and $\alpha$ are as in Lemma \ref{c: if-not-sigma-stable-then-one-simple-root},
                    then $P':=s_{\alpha}(P)$ has the following properties:
                    \begin{enumerate}[{\ }\rm (1) \sl]
                        \itema  $\Sigma(P)\cap\faqd=\Sigma(P')\cap\faqd$\,,
                        \itemb $\Sigma(P,\sigma)\subsetneq\Sigma(P',\sigma)$.
                    \end{enumerate}
                                 \end{corollary}

                In the proof of the above corollary, we will  follow  the convention to write
                $$
                R_\circ := \{ \ga \in R :  \half\, \ga \notin R\}
                $$
                for any possibly non-reduced root system $R.$ The elements of $R_\circ$ are called the indivisible roots in $R.$ Furthermore, if $S \subseteq R$ is any subset, we will write $S_\circ: =  S \cap R_\circ.$ Finally, we agree to write $\Sigma_\circ(P)$ for $\Sigma(P)_\circ.$

                \begin{proof}
                    It suffices to prove (a) and (b) with everywhere $\Sigma$ replaced by $\Sigma_\circ.$ Since $P':=  s_{\alpha}(P)$ with $\ga$ simple in $\gS(P)$, we have
                    \[\Sigma_\circ(P')=(\Sigma_\circ(P)\setminus\{\alpha\})\cup\{-\alpha\},\]
                    which implies (a).

                    Let $\beta\in\Sigma_\circ(P)\cap\sigma\Sigma_\circ(P)$. Then $\beta\neq\alpha$ and $\sigma\beta\neq\sigma\alpha$ and we infer that $\beta$ and $\sigma\beta$ both belong to $\Sigma_\circ(P')$. It follows that $\beta\in\Sigma_\circ(P')\cap\sigma\Sigma_\circ(P')$. This proves the inclusion in (b). We still need to show that equality cannot hold. This follows from the fact that
                    $\theta\alpha=- \alpha\in\Sigma(P',\sigma)\setminus\gS(P).$
                \end{proof}

                \begin{lemma}\label{l: Q-existence} Let $P \in \cP(A).$ Then there exists a minimal parabolic subgroup $\Phext \in\mathcal{P}(A)$
                     such that the following conditions hold:
                    \begin{enumerate}[{\ }\rm (1)\sl]
                        \itema $\Sigma(\Phext)\cap\fa_\iq^*=\Sigma(P)\cap\fa_\iq^*$,
                        \itemb $\Sigma(\Phext)\cap\fa_\ih^*=\Sigma(P)\cap\fa_\ih^*$,
                        \itemc  $\gS(P, \gs) \subseteq \gS(\Qh, \gs)$,
                        \itemd $Q_\ih$ is $\fh$-extreme, see {\rm (\ref{e: h extreme})}.
                    \end{enumerate}
                \end{lemma}

                \begin{proof}
                    If $\ga \in \Sigma(P)\cap \faqd,$ then $\gs \ga  = -\ga \notin \Sigma(P).$ Hence
                    \begin{equation}\label{e: intersection gS and gs gS}
                            \Sigma(P,\sigma)=\Sigma(P)\cap\sigma\Sigma(P)\subseteq \Sigma(P)\setminus\fa_\iq^*.
                    \end{equation}
                    If the above inclusion is an equality, the result holds with $\Phext:= P.$ If not, then the inclusion in (\ref{e: intersection gS and gs gS}) is proper and Lemma \ref{c: if-not-sigma-stable-then-one-simple-root} guarantees the existence of a simple root $\alpha\in\Sigma(P)\setminus\fa_\iq^*$ such that $\sigma\theta\alpha\in\Sigma(P)$. By applying Corollary \ref{c: the-psgp-s-alpha-P} we see that the  minimal parabolic subgroup $P' :=s_{\alpha}(P)$ satisfies the above conditions (a) and (b), and
                    \begin{equation}\label{e: P and P' towards extreme}
                        \Sigma(P,\sigma)\subsetneq\Sigma(P',\sigma).
                    \end{equation}
                    Put $P_0 = P$ and $P_1 = P'.$ By applying the above process repeatedly, we obtain a sequence of parabolic subgroups $P=P_0,P_1,\ldots ,P_k$ satisfying
                    \begin{enumerate}
                        \itema $\Sigma(P_i)\cap\faqd=\Sigma(P_{i+1})\cap\faqd$,
                        \itemb    $\Sigma(P_i)\cap\fahd=\Sigma(P_{i+1})\cap\fahd$,
                        \itemc    $\Sigma(P_i,\sigma)\subsetneq\Sigma(P_{i+1},\sigma)$,
                    \end{enumerate}
                    for $0 \leq i < k.$ The process ends when for some $k>0$ the condition $\Sigma(P_k)\cap\sigma\Sigma(P_k)=\Sigma(P_k)\setminus\faqd$ is satisfied. The parabolic subgroup $\Phext = P_k$ satisfies all assertions of the lemma.
                \end{proof}

                \begin{remark}\label{r: fq extreme parabolic}
                In analogy with Definition \ref{d: h extreme parabolic}, a parabolic subgroup $Q \in \cP(A)$ is said to be $\fq$-extreme if $\gS(Q, \gs \Cartan ) = \gS(Q) \setminus \fahd.$ With obvious modifications in the proof, Lemma \ref{l: Q-existence} is valid with everywhere $\gs\Cartan$ in place of $\gs$ and with
                $\fq$-extreme in place of $\fh$-extreme. However, we will not need this result in the present paper.
                \end{remark}

            \subsection{The convexity theorem for a q-extreme parabolic subgroup}
            \label{ss: relation with vdbconv}
           We shall now explain why the result of      \cite{ban1986} is a special case of the Main Theorem. We keep the notation as above and impose that $P\in\mathcal{P}(A)$ is $\fq$-extreme,     see Remark \ref{r: fq extreme parabolic}. Then
           $
           \Sigma(P,\sigma\theta)=\Sigma(P)\setminus\fa_\ih^*,
           $
              so that
             \[
             \Delta^+:=\Sigma(P,\sigma\theta)|_{\faq}
             \]
             is a positive system for $\gS(\fg, \faq).$  For $\ga\in \Sigma(\fg, \faq),$     the root space $\fg_\ga$ is $\gs\Cartan$-invariant;  we write
              $\fg_{\ga,\pm}$ for the $\pm 1$ eigenspaces of $\gs\Cartan|_{\fg_\ga}.$ Put
             \[
             \Delta^+_- = \{ \ga \in \Delta^+\setmid \fg_{\ga, -} \neq 0\}.
             \]
             Then \cite[Thm 1.1]{ban1986} asserts that
            $$
            \mathfrak{H}_{P,{\rm q}}(aH)=\conv(\WKHaq \cdot \log a)+\Upsilon(P),
            $$
            where $\Upsilon(P)$ is the finitely generated polyhedral cone in $\faq$ defined    by
            \[
            \Upsilon(P)=\sum_{\alpha\in\Delta^+_-}\mathbb{R}_{\geq 0}H_{\alpha};
            \]
             here $H_\ga$ denotes the element of $\faq$
             with $H_\ga \perp \ker \ga$ and $\ga(H_\ga) = 2.$

            Thus, our main theorem coincides with \cite[Thm. 1.1]{ban1986} provided that $\Gamma(P) = \Upsilon(P).$ The latter is asserted by the following lemma.

            \begin{lemma}
            Let $P \in \cP(A)$ be  $\fq$-extreme. Then $\Upsilon(P) = \Gamma(P).$
            \end{lemma}

            \begin{proof} For a root $\ga \in \gS(\fg, \fa)$ we denote by $H_\ga^\vee \in \fa$
            the element determined by
            \begin{equation}
            \label{e: defi H ga vee}
            \inp{H_\ga^\vee}{X} = \ga(X)
            \end{equation}
            for all $X \in \fa.$ Then it is readily verified that
            \begin{equation}
            \label{e: relations H ga}
            H_\ga^\vee = 2 H_\ga/\inp{H_\ga}{H_\ga}.
            \end{equation}
            Similarly, for $\ga \in \gS(\fg, \faq)$ we define $H_\ga^\vee$ to be the element
            of $\faq$ determined by (\ref{e: defi H ga vee}) for all $X\in \faq.$ For this element we
            also have (\ref{e: relations H ga}), but now as an identity of elements of $\faq.$
            If $\ga \in \gS(\fg, \fa)$ has non-zero restriction to $\faq,$ then $\ga|_{\faq} \in \gS(\fg, \faq)$
            and for natural reasons we have
            $$ \prq(H_\ga^\vee) = H_{\ga|_{\faq}}^\vee.$$
           From this we conclude that
           \begin{equation}
            \label{e: proj of H ga}
            \pr_\iq (H_\ga) = c_\ga H_{\ga|_{\faq}},
            \end{equation}
            with $c_\ga = \|H_\ga\|^2 \|H_{\ga|_{\faq}}\|^{-2} > 0.$

            After these preliminary remarks we will now complete the proof.
            Let $\ga \in \gS(P,\gs\Cartan).$ Then  $\ga|_{\faq}$ is non-zero, hence  a root in $\gS(\fg, \faq),$ and (\ref{e: proj of H ga}) is valid.
                     As $\gs\Cartan$ restricts to the identity on $\faq,$ the $\fa$-roots  $ \ga$ and $\gs \Cartan\ga$ have the same restriction to $\faq$ giving the root $\ga|_{\faq}$ of  $\Delta^+.$
            If the given $\fa$-roots are different, then the  sum $\fg_\ga + \gs\Cartan(\fg_\ga) $ is direct and contained in $\fg_{\ga|_{\faq}}$ and we see that
            $\fg_{\ga|_{\faq}, -} \neq 0,$  so that $\ga \in \gS(P)_-$  and      $\ga|_{\faq} \in \Delta^+_-.$
            On the other hand, if $\ga = \gs\Cartan \ga,$ then $\fg_{\ga} = \fg_{\ga|_{\faq}}$
            and we see that $\ga \in \gS(P)_-$ if and only if
            $\ga|_{\faq} \in \Delta^+_-.$ It follows from this argument that $\gS(P)_-|_{\faq} = \Delta^+_-.$ Using (\ref{e: proj of H ga}) we now see that
            $$
            \Gamma(P) = \sum_{\ga \in \gS(P)_-}  \R_{\geq 0} H_{\ga|_{\faq}} = \sum_{\ga \in \Delta^+_-} \R_{\geq 0} H_{\ga} = \Upsilon(P).
            $$
            \end{proof}

            \subsection{Decompositions of nilpotent Lie groups}

            In this section we give a brief survey of a number of useful results on  decompositions of nilpotent Lie groups that will be needed in this paper.

We start by recalling the following standard result.

\begin{lemma}\label{l: basics nilpotent groups}
Let $N$ be a connected and simply connected Lie group with nilpotent Lie algebra $\fn.$
If $\fn_0$ is a subalgebra of $\fn,$ then the exponential map maps $\fn_0$ diffeomorphically onto
a closed subgroup of $N.$
\end{lemma}



                 \begin{lemma}[{\cite[ Lemma IV.6.8]{helgason1984}}]\label{Helgason}
                    Let $N$ be a connected, simply connected nilpotent Lie group with Lie algebra $\mathfrak{n}$. Let $(\mathfrak{n}_i)_{0\leq i \leq k}$
                     be a strictly decreasing sequence of ideals      of $\fn$ such that     $\fn_0 = \fn,$ $\fn_k = 0$ and
                     $$
                     [\mathfrak{n},\mathfrak{n}_i]\subseteq \mathfrak{n}_{i+1} \quad\quad \mbox{\rm for all} \quad 0 \leq i < k.
                     $$
                     Let $\mathfrak{b}_1$ and $\mathfrak{b}_2$ be two mutually complementary subspaces of $\mathfrak{n}$ such that $\mathfrak{n}_i=\mathfrak{b}_1\cap\mathfrak{n}_i+\mathfrak{b}_2\cap\mathfrak{n}_i$, for all $0 \leq i \leq k$. Then the mapping \[\varphi:(X,Y)\to\exp X\exp Y\] is an analytic diffeomorphism
                     from $\mathfrak{b}_1\times\mathfrak{b}_2$ onto $N$.
                \end{lemma}

                \begin{lemma}\label{l: KDV}
                Let $N_P$ be the nilpotent radical of a minimal parabolic subgroup
                $P \in \cP(A),$ let $ \mathfrak{n}_P$ be its Lie algebra and let $\mathfrak{n}_1,\ldots,\mathfrak{n}_k\subset\mathfrak{n}_P$ be linearly independent subalgebras of $\mathfrak{n}_P$ that are direct sums of $\mathfrak{a}$-root spaces. Assume that $\mathfrak{n}=\mathfrak{n}_1\oplus\ldots\oplus\mathfrak{n}_k$ is a subalgebra of $\mathfrak{n}_P$. Denote by $N:=\exp\mathfrak{n}$ and by $N_i:=\exp\mathfrak{n}_i,\,i\in\{1,\ldots,k\},$ the corresponding closed subgroups of $N_P$. Then the multiplication map
                    \[\mu:N_1\times\ldots\times N_k\to N\]
                    is a diffeomorphism.
                \end{lemma}

                This    result is stated in     \cite[Lemma 2.3]{duiskolkvar1983} for $\fn = \fn_P,$ with reference to  \cite{chevalley1958}.
                 We need the      present slightly more general version with
                $\fn$ a subalgebra of $\fn_P.$ A proof of this      result  can be found in Appendix \ref{a: proof of KDV}.

\subsection{Fixed points for the involution in minimal parabolic subgroups}\label{ss: langlands-decomp-subsec}
                    Let $P\in\cP(A)$. The decomposition
                $P=LN_P$ induces a similar decomposition for the intersection $P\cap H$. In the present subsection we present a proof for this fact, see the lemma below.

               \begin{lemma}\label{l: P-intersect-H-decomposition}
                    $P\cap H\simeq(L\cap H)\times(N_P\cap H)$
                \end{lemma}

                \begin{proof}
                    Let $p$ be an element in $P\cap H$. According to the decomposition $P=L N_P$, we write $p=ln$.
                    Then,
                    $ln = \sigma(ln)=\sigma(l)\sigma(n)$ and we obtain that $\sigma(n)n^{-1}=\sigma(l)^{-1}l\in L$.
                    Since $\sigma(N_P)$ is the nilpotent radical of the parabolic subgroup $\gs(P) \in \cP(A),$ it follows from Lemma \ref{l: KDV} with $k=2$
that the multiplication map
$$
(\sigma(N_P)\cap\bar{N}_P)\times(\sigma(N_P)\cap N_P) \to \gs(N_P)
$$
induces a diffeomorphism. We thus see that $\sigma(n)n^{-1}\in\bar{N}_P N_P$. Now, by \cite[Lemma 7.64]{knapp2002} it follows that $\bar{N}_P N_P\cap L ={e}$ and thus $\sigma(n)=n$ and $\sigma(l)=l$.
                \end{proof}

\subsection{Decomposition of nilpotent radicals induced by the involution}
\label{ss: decomposition-for-NP}
In this subsection, we assume that $P \in \cP(A).$
We will show that the unipotent radical $N_P$ decomposes as the product of $N_P\cap H$ and a suitable closed subgroup $N_{P,+}$
of $N_P.$
To describe this subgroup, we need the existence of suitable elements of $\faq.$
As     usual, an element $X \in \faq$ is said to be regular for the root system $\gS(\fg, \faq)$ if no root of this system vanishes on it. The set of such regular elements is denoted by $\faqreg.$ We observe that in terms of the system $\gS(\fg, \fa)$ this set may be described  as
\begin{equation}\label{e: aq regular}
                \faqreg =  \{X\in\fa_\iq\setmid\;\;
                \forall\alpha\in\Sigma(\mathfrak{g},\mathfrak{a}):\;\;
                \alpha(X)                =0\Rightarrow\alpha|_{\fa_\iq}=0
              \}.
 \end{equation}

\begin{lemma}\label{l: existence Zq}{\ }
\begin{enumerate}
\itema
There exists an element $Z_\iq \in \faqreg $ such that
 $\ga(Z_\iq) > 0$ for all  $\ga \in \gS(P, \gs\Cartan).$
 \itemb
 There exists an element $Z_\ih \in \fah$ such that
 $\ga(Z_\ih) >0$ for all $\ga \in \gS(P,\gs).$
 \end{enumerate}
\end{lemma}

\begin{proof}
The set
                \[
                \mathfrak{a}' :=  \{X\in\mathfrak{a}\setmid\;\;
                \forall\alpha,\beta\in\Sigma(\mathfrak{g},\mathfrak{a}) :\;\; \alpha(X)=\beta(X)\Rightarrow\alpha=
                \beta\}
                \]
is the complement of finitely many hyperplanes in $\fa,$ hence open and dense.
Let $\mathfrak{a}^+(P)$      denote the positive chamber associated      with the positive system  $\gS(P)$
for $\gS(\fg,\fa).$ Fix $Z_P\in\mathfrak{a}^+(P)\cap\mathfrak{a}'$.
Then it is readily verified that
$
Z_\iq:=Z_P+\sigma\theta(Z_P)
$
 satisfies the requirements of (a). Likewise, the element
 $Z_\ih = Z_P + \sigma(Z_P)$ satisfies the requirements of (b).
 \end{proof}

 Given    {} $Z_\iq \in \faqreg$ we put
 $\Sigma(P, +) :=\{\alpha\in\Sigma(P):\alpha(Z_\iq)>0\}.$
      Then
                \[
                \fn_{P,+}:=\bigoplus_{\alpha\in\Sigma(P,+)}
                \fg_{\alpha}
                \]
                is a subalgebra of $\fn_P.$
                Let      $N_{P,+}:=\exp\fn_+$ be the corresponding closed subgroup of $N_P$,
                see Lemma \ref{l: basics nilpotent groups}. Define
                \[\fn_{P,\sigma}:=\bigoplus_{\alpha\in\Sigma(P,\sigma)}\fg_{\alpha}\]
                and $N_{P,\sigma}$ as the corresponding closed subgroup.
                 \begin{prop}\label{p: N=N_+(N_P cap H) decomposition}
                 Let $Z_\iq \in \faqreg$ be as in
                 Lemma \ref{l: existence Zq} (a) and let $N_{P,+}$ be defined as above. Then
                 the multiplication map
                 $$
                 N_{P,+} \times(N_P\cap H)\to N_P
                 $$
                 is a diffeomorphism.
                \end{prop}
                The proof of this result relies on the following lemma.

                \begin{lemma}\label{l: N_P_sigma_decomposition}
                    Let $P\in\mathcal{P}(A)$ and let     $Z_\iq\in\faqreg$      be as in
                    Lemma \ref{l: existence Zq} (a).
                    Put     \[\Sigma(P,\sigma,+):=\{\alpha\in\Sigma(P,\sigma)\setmid\alpha(Z_\iq)>0\}.\]
                    Then the following statements hold:
                     \begin{enumerate}
\itema
$\mathfrak{n}_{P,\sigma,+}:=\oplus_{\alpha\in\Sigma(P,\sigma,+)}\fg_{\alpha}$
is a subalgebra of $\,\mathfrak{n}_{P,\sigma}$,
\itemb
$N_{P,\sigma,+}:=\exp\mathfrak{n}_{P,\sigma,+}$ is a closed subgroup of
$\,N_{P,\sigma}$,
\itemc $\mathfrak{n}_{P,\sigma}=
\mathfrak{n}_{P,\sigma,+}\oplus(\mathfrak{n}_{P}\cap\mathfrak{h})$,
\itemd
the multiplication map  \[\mu:N_{P,\sigma,+}\times(N_{P}\cap H)\to N_{P,\sigma}\] is a diffeomorphism.
                    \end{enumerate}
                \end{lemma}

                \begin{proof}
                          (a): Assume that     $\ga, \gb \in \gS(P, \sigma,+)$ and $\ga + \gb \in \gS(\fg, \fa).$ Then
                     $\ga + \gb \in \gS(P, \sigma)$ and $(\ga + \gb) (Z_\iq) > 0$ so that     $\ga + \gb \in \gS(P, \gs, +).$ This implies (a).

                    Assertion (b) follows from (a) by application of Lemma \ref{l: basics nilpotent groups}.

                    Next, we prove (c). If $\alpha\in\Sigma(P,\sigma,+)$ then $\sigma\alpha(Z_\iq)<0$, which implies $\sigma\ga \notin \gS(P,\gs,+).$ Hence,  $\mathfrak{n}_{P,\sigma,+}\cap\mathfrak{h}=\{0\}$. It follows that
                    \[
                    \mathfrak{n}_{P,\sigma,+}\cap(\mathfrak{n}_{P}\cap\mathfrak{h})=\{0\}.
                    \]
                    It remains      to be shown that any $X\in\mathfrak{n}_{P,\sigma}$ can be written as
                    \[X=X_++X_\ih,\]
                    with $X_+\in\mathfrak{n}_{P,\sigma,+}$ and      $X_\ih\in \mathfrak{n}_{P}\cap\mathfrak{h}$.
                    It suffices to prove this for $X\in\fg_{\alpha}\subset\mathfrak{n}_{P,\sigma}$. If $\alpha(Z_\iq)>0$, then $X\in\fn_{P,\sigma,+}$ by definition. On the other hand if $\alpha(Z_\iq)=0$, then by regularity of $Z_\iq$ we have that $\alpha\in\fah^*$ and thus $\fg_\alpha\in\fh$, which implies that $X\in\mathfrak{n}_{P}\cap\mathfrak{h}$. Finally, if $\alpha(Z_\iq)<0$, then \[X=(X+\sigma(X))-\sigma(X)\]
                    with      $X + \gs(X) \in \fn_P \cap \fh$ and $-\gs(X) \in \fn_{P,\gs, +},$ and we are done.

                    For (d) fix $Z_{\mathfrak{h}}$ as in Lemma \ref{l: existence Zq} (b). Then for all $\alpha \in \Sigma(P,\sigma)$ we have that $v_\ga:= \alpha(Z_{\mathfrak{h}})>0$.
                         Let the set of positive real numbers thus obtained be ordered by $v_{\ga_1} < v_{\ga_2} < \cdots < v_{\ga_m}.$
                    We define      $\fn_0 = \fn_{P,\gs}$, $\fn_m = 0,$ and for $1 \leq i <m,$
                    \[
                    \mathfrak{n}_i:=\bigoplus_{\substack{\alpha\in\Sigma(P,\sigma)\\ \alpha(Z_\ih)> v_{\ga_i}}}\fg_{\alpha}.
                    \]
                     Then $\fn_1, \ldots, \fn_m$ is a strictly decreasing sequence of ideals in $\fn_{P,\gs}$ with $[\fn, \fn_i] \subseteq \fn_{i+1}$ for $0 \leq i <m.$
                          We note that each $\fn_i$ is invariant under $\gs.$ Hence by the same argument as used in the proof of (c) above it follows that
                      \[
                      \mathfrak{n}_i=(\mathfrak{n}_i\cap\mathfrak{n}_{P,\sigma,+})\oplus(\mathfrak{n}_i\cap(\mathfrak{n}_{P}\cap\mathfrak{h}))
                      \]
                      for all $0 \leq i \leq m.$
                      Thus, we may apply Lemma \ref{Helgason} to conclude that
                      \[
                      N_{P,\sigma}\simeq N_{P,\sigma,+}\times\exp(\mathfrak{n}_{P}\cap\mathfrak{h}).
                      \]
                    It remains to show that $\exp(\mathfrak{n}_{P}\cap\mathfrak{h})=N_P\cap H.$ This follows from
                    \[N_P\cap H\subseteq\{n\in N_P\,:\,\sigma(n)=n\}=\{\exp X\,:\,X\in\fn_P\cap\fh\}\subseteq N_P\cap H.\]
                    The proof is complete.
                \end{proof}
                \medbreak
                {\em Proof of Prop.\ \ref{p: N=N_+(N_P cap H) decomposition}.\ }
                  Let \[\mathfrak{n}_{P,\sigma\theta}:=\sum_{\alpha\in\Sigma(P,\sigma\theta)}\mathfrak{g}_{\alpha}\] and let $N_{P,\gs\theta}$ be the corresponding closed subgroup of $N_P$. Then
                $
                \mathfrak{n}_P=\mathfrak{n}_{P,\sigma\theta}\oplus\mathfrak{n}_{P,\sigma}
                $
                and by Lemma \ref{l: KDV} we obtain that
                \begin{equation}\label{e: N_P_decomp_sigma_sigma_theta}
                    N_P \simeq N_{P,\sigma\theta}\times N_{P,\sigma}.
                \end{equation}
                Applying
                Lemma \ref{l: N_P_sigma_decomposition} to the second component we obtain that
                     \[
                     N_P\simeq N_{P,\sigma\theta} \times N_{P,\sigma,+}\times (N_{P}\cap H).
                     \]
                     On the other hand,
                          $\mathfrak{n}_{P,+} = \mathfrak{n}_{P,\sigma,+}\oplus \mathfrak{n}_{P,\sigma\theta}$.
                     From this we infer by application of Lemma \ref{l: KDV} that
                     \[
                     N_{P,\sigma\theta}\times N_{P,\sigma,+}\simeq N_{P,+}
                     \]
                     The result follows.
                     \qed
                \begin{remark}
                    For the case of an $\fh$-extreme parabolic subgroup, Proposition \ref{p: N=N_+(N_P cap H) decomposition} is due to \cite{AndFJSchlicht2012}, where, for this special case, a different proof of the result is given.
                \end{remark}
    \section{Auxiliary results in convex linear algebra}\label{s: auxiliary results in convex linear algebra}

              In this section we present a few results in     convex linear algebra     which will be used in Section \ref{s: prop-sec}.

            \begin{lemma}\label{l: conv_vs_aux}
                Let $V$ be a finite dimensional real linear space and $B\subseteq  V$ a closed subset, star-shaped about the origin. If $B$ is non-compact, then there exists a $v\in V\setminus\{0\}$ such that $\mathbb{R}_{\geq 0}v\subseteq B$.
            \end{lemma}

            \def\R{{\mathbb R}}
            \begin{proof}
                Since $B$ is star-shaped, we have $sB = t(s/t)B \subseteq t B$ for all $0< s < t.$ Fix a positive definite inner product on $V$ and let $S$ be the associated unit sphere centered at the origin. For $s >0$ we define the compact set $C_s: = s^{-1} B \cap S.$ Then $s < t \implies C_s \supseteq C_t.$ As $B$ is unbounded and starshaped, each of the sets $C_s$ is non-empty. It follows that the intersection
$$
C := \cap_{s > 0} \;C_s
$$
                is non-empty. Let $v$ be a point in this intersection. Then $v \neq 0$ and for all $s > 0$ we have $sv \in s C_s \subseteq B.$ Hence, $\R_{\geq 0} v \subseteq B.$
            \end{proof}

            \begin{lemma}\label{l: conv vs}
                Let $V$ and $W$ be two finite dimensional real linear spaces,      $p:V\to W$ a linear map and  $\Gamma \subseteq V$  a closed convex cone. Then the following assertions are equivalent.
                \begin{enumerate}
                    \itema $p|_{\Gamma}$ is a proper map.
                    \itemb $\ker p\cap \Gamma =\{0\}$.
                \end{enumerate}
            \end{lemma}

            \begin{proof}
                First we prove that (a) implies (b). Assume (b) doesn't hold, i.e. there exists $v\in\ker p\cap \Gamma$, $v\neq 0$. Then $\mathbb{R}_{\geq 0}v\subseteq \ker p\cap \Gamma=(p|_{\Gamma})^{-1}(0)$ and we obtain that $(p|_{\Gamma})^{-1}(0)$ is not compact and hence $p|_{\Gamma}$ is not a proper map.

                For the converse implication, assume that (a) does not hold.      Then there exists a compact set $K\subseteq W$, such that      the set $p^{-1}(K)\cap \Gamma$ is not compact. As the latter set is closed, it is unbounded in $V.$
                Let $\bar{K}$  be the convex hull of $K\cup\{0\}$. Then $\bar{K}$ is compact     and $p^{-1}(\bar{K})\cap \Gamma$ is convex, contains 0 and is     unbounded in $V,$ hence not compact. We apply Lemma \ref{l: conv_vs_aux} and obtain that there exists $v\neq 0$ such that $\forall  t\geq 0:$ $tv\in p^{-1}(\bar{K})\cap \Gamma$. Hence, $t\cdot p(v)\in\bar{K}$ for every $t\geq 0$. Since $\bar{K}$ is compact, it follows that $p(v)=0$ and $v\in\ker p\cap \Gamma$, which implies that (b) cannot hold.
            \end{proof}

            \def\cK{{\mathcal K}}
            \begin{lemma}\label{l: properness addition on cone}
                Let $V$ be a finite dimensional real linear space, and $\Gamma$ a closed convex cone in $V$ such that there exists a linear functional $\xi \in V^*$ with $\xi > 0$ on $\Gamma\setminus \{0\}.$ Then the following holds.
                \begin{enumerate}
                    \itema For every $R > 0$ the set $\{x \in \Gamma\setmid \xi(x) \leq R\}$ is compact.
                    \itemb The addition map $\;\,a: (x,y) \mapsto x + y,\;\; \Gamma \times \Gamma \to V, $ is proper.
                \end{enumerate}
            \end{lemma}

            \begin{proof}
                Let $R > 0.$ The set $\Gamma_R:= \{x \in \Gamma\setmid \xi(x) \leq R \}$ is closed and convex and it contains the origin. If $v \in \Gamma_R\setminus \{0\}$ then the half line $\R_{\geq 0} v$ is not contained in $\Gamma_R.$ By Lemma \ref{l: conv_vs_aux} we infer that $\Gamma_R$ is compact, hence (a).

                We turn to (b). Assume $\cK \subseteq V$ is compact. Then there exist an $R > 0$ such that $\xi \leq R$ on $\cK.$ Let $(x,y)\in a^{-1}(\cK).$ Then it follows that $\xi(x + y) \leq R,$ hence $\xi(x)\leq R$ and $\xi(y) \leq R,$ so that $(x,y)$ belongs to the compact set $\Gamma_R \times \Gamma_R.$ We conclude that $a^{-1}(\cK)$ is a closed subset of $\Gamma_R \times \Gamma_R,$ hence compact.
            \end{proof}

            If $S$ is a subset of $\gS(\fg, \fa)$ then      the convex cone
            $$
            \Gamma_\fa(S): = \sum_{\ga \in S} \R_{\ge 0} H_\ga.
            $$
            is finitely generated, hence closed in $\fa.$
            Likewise,
            $$
            \Gamma_{\faq}(S): = \prq \Gamma_\fa(S) = \sum_{\ga \in S} \R_{\ge 0} \,\prq(H_\ga)
            $$
            is a  closed and convex cone in $\faq.$

            \begin{corollary}\label{c: cones and properness}
                Let $P \in \cP(A).$ Then the following assertions are valid.
                \begin{enumerate}
                    \itema The map $\pr_\iq: \Gamma_\fa(\gS(P, \gs\Cartan)) \to \faq$ is proper.
                    \itemb The addition map $\;\;a: \Gamma_{\faq}(\gS(P, \gs\Cartan)) \times \Gamma_{\faq}(\gS(P, \gs \Cartan)) \to \faq$ is proper.
                \end{enumerate}
            \end{corollary}

            \begin{proof}
                We start with (a). In view of Lemma \ref{l: conv vs} it suffices to establish the claim that  $\Gamma_\fa (\gS(P, \gs))\cap \fah = 0.$ This can be done as follows. There exists a $Y \in \fa$ such that $\ga(Y) > 0$ for all $\ga \in \gS(P).$ Put $X := Y + \gs \Cartan Y = Y - \gs(Y),$ then
                 $X \in \faq$ and $\inp{X}{H_\ga} = \inp{H_\ga}{H_\ga} \ga(X)/2 = \inp{H_\ga}{H_\ga} (\ga + \gs\Cartan \ga)(Y)/2 > 0$ for all $\ga \in \gS(P, \gs \Cartan).$
                  It follows that the linear functional $\xi = \inp{X}{\dotvar}\in \fa^*$ has strictly positive values on $\Gamma_\fa(\gS(P, \gs \Cartan))\setminus \{0\}.$ Now $\xi = 0$ on $\fah$ and we see that the claim is valid. Hence, (a).

                For (b) we proceed as follows. Let $\xi$ be as above, then $\ker \prq \subseteq \ker \xi$ and we see that $\xi > 0$ on      $\Gamma_{\faq}(\gS(P, \gs \Cartan))\setminus \{0\}. $ Now use Lemma \ref{l: properness addition on cone}.
            \end{proof}

    \section{Properness of the Iwasawa projection}\label{s: prop-sec}

            Let $P \in \cP(A)$      and let $\HP: G \to \fa$ be the Iwasawa projection defined by (\ref{e: Iwasawa projection}).
           Let $\HPq: G \to \faq$ be defined as in Remark \ref{r: defi HPq}.
            The purpose of this section is to prove that the restriction of $\HPq$ to $H$ factors through a proper map $H/H\cap P \to \faq.$

            We start with a simple lemma.

            \begin{lemma}\label{l: HP-inv}
                The map  $\HPq|_H: H\to \faq$ is      left $K\cap H$- and right $(P\cap H)$-invariant.
            \end{lemma}

            \begin{proof}
                Let $h \in H,$ $k_H \in K \cap H$ and $p \in P\cap H.$ By the Iwasawa decomposition,      the element $h$ may be decomposed as $h = k a n$, with $k\in K$, $a\in A$ and $n\in N_P$. In view of Lemma \ref{l: P-intersect-H-decomposition} we may decompose $p = m b n',$ with $m \in M\cap H,$ $b \in A \cap H$ and $n' \in N_P \cap H.$ Since $MA$ normalizes $N_P$ and centralizes $A$ we find
                $$
                k_H h p = k_H k a n m b n' = (k_H k m) a b ((mb)^{-1} n (mb)) n'\in K ab N_P.
                $$
                From this we deduce that
                $$\HPq(k_Hhp) = \prq \,(\log a + \log b) = \prq\, \log a = \HPq(h).
                $$
                \vspace{-42pt}

            \end{proof}

            \vspace{12pt}\medbreak
            It follows from the above lemma      that  the restriction of $\HPq$ to $H$ induces a smooth map

            \begin{equation}\label{e: induced smooth map on H mod P}
                \barHPq: H/H\cap P \to \faq.
            \end{equation}

            The following proposition is the main result of this section.

            \begin{prop}\label{p: properness result}
                The   induced map (\ref{e: induced smooth map on H mod P}) is proper.
            \end{prop}

                In order to prove the proposition, we will reduce to another result, Prop.\ \ref{p: reduced properness result},
            establishing some useful lemmas along the way.

            We fix    $\Qh$ in $\fh$-extreme position   and related to $P$ as in Lemma \ref{l: Q-existence}.
            Let $Z_G(\fah)$ denote the centralizer of $\fah$ in $G$ and define the parabolic subgroup
            \begin{equation}\label{e: def of R}
                R: = Z_G(\fah) N_{\Qh}.
            \end{equation}
            Let $\fn_R$ be the sum of the root spaces $\fg_\ga$ for $\ga \in \gS(\Qh,\gs)=   \gS(\Qh) \setminus \faqd $
            and put $N_R:= \exp(\fn_R).$
              Then $N_R$ is $\gs$-stable.
            It is readily seen that $R$ has the Levi decomposition $R = L_R N_R$ where $L_R = Z_G(\fah)$ is $\gs$-stable.  Hence, $R$ is $\gs$-stable.
                  Let $\gS(R)$ denote the set of $\fa$-roots that appear in $\fn_R.$
             \begin{lemma}\label{l: roots in gS P gs Cartan}
                $\gS(P) \cap \gS(\bar R) \subseteq \gS(P, \gs\Cartan).$
            \end{lemma}

            \begin{proof}
                Let $\ga \in \gS(P) \cap \gS(\bar R).$ Then $\ga\in \gS(\bar Q_{\rm h})\setminus \faqd =
                -\gS(\Qh, \gs),$ hence $\ga \notin \gS(P, \gs),$ see Lemma \ref{l: Q-existence} (b). This implies that $\ga \in \gS(P, \gs\Cartan).$
            \end{proof}

Let $R = M_R A_R N_R$ be the Langlands decomposition of $R.$ Then  $L_R = M_R A_R.$

            \begin{lemma}\label{l: deco H mod H cap P}
                The multiplication map
                $$\mu:(K\cap H)\times(M_R\cap H)\times(N_R\cap H)/(N_R \cap H \cap P)\longrightarrow H/H\cap P,$$
                given by $(k, m, [n]) \mapsto km [n]$ is surjective.
            \end{lemma}
            \begin{proof}
             The map       $K \times (\fl_R \cap \fp) \times N_R \to G$ given by $(k, X, n) \mapsto k \exp X n$ is a diffeomorphism. Since $K,$ $\fl_R \cap \fp$ and $N_R$ are $\gs$-stable, whereas
             $N_R^\gs = N_R\cap H,$ it follows that
            \begin{equation}\label{e: H decomposition}
                H = (K\cap H)(L_R\cap H)(N_R\cap H).
            \end{equation}
            Now $L_R=M_RA_R$ with  $M_R$ and $A_R$ both $\gs$-stable. Since $A_R\cap H$ normalizes $N_R\cap H$, we have that
            \begin{eqnarray*}
                H &=&(K\cap H)(M_R\cap H)(A_R\cap H)(N_R\cap H)\\
                  &=&(K\cap H)(M_R\cap H)(N_R\cap H)(A_R\cap H).
            \end{eqnarray*}
            This implies the result.
            \end{proof}

            We equip $M_R \cap H$ with the natural right-action of the closed subgroup $M_R \cap H \cap P.$ The latter group acts on $N_R\cap H$ by conjugation. Moreover, since $M_R$ normalizes $N_R$ and $P$ normalizes $N_P,$ the conjugation action leaves the closed subgroup $N_R \cap H \cap P$ invariant. Accordingly, we have an induced right-action of $M_R \cap H \cap P$ on $(N_R \cap H)/(N_R \cap H \cap P)$ given by
            $$[n]\cdot m := [ m^{-1} n m], \qquad (m \in M_R \cap H \cap P, \,n \in N_R\cap H).$$
            We equip
            $(M_R \cap H )$ with the usual right-action by  $M_R \cap H \cap P,$
            and $(N_R \cap H)/(N_R \cap H \cap P)$ with the product action.
             The latter action is proper and free, so that the associated quotient space  $(M_R \cap H)\times_{M_R \cap H \cap P} (N_R\cap H)/(N_R \cap H \cap P)$ is a smooth manifold.

            \begin{lemma}
                The multiplication map of Lemma \ref{l: deco H mod H cap P} induces a surjective smooth map
                $$\bar \mu: (K\cap H) \times (M_R \cap H)\times_{M_R \cap H \cap P} (N_R\cap H)/(N_R \cap H \cap P) \to H/H\cap P.$$
            \end{lemma}

            \begin{proof}
                Let $k \in K\cap H,$ $m \in M_R \cap H$ and $n \in N_R \cap H.$ Then for $p \in M_R \cap H \cap P$ we have
                $$\mu(k, (m,[n])\cdot p) = \mu(k,mp,[ p^{-1} n p]) = k mp (p^{-1}np ) [e] = k m n [e] = \mu(k, m,[n]).$$
                This implies that $\mu$ induces a smooth map $\bar \mu$ as described. The surjectivity of $\bar \mu$ follows from
                the surjectivity of $\mu. $
            \end{proof}

            Proposition \ref{p: properness result} will follow from the result that the composition $\HPq\circ \bar \mu$ is proper. The latter map is left-invariant under the left action of $K \cap H$ on the first component. Thus, Proposition \ref{p: properness result} will already follow from the following result.

            \def\gf{\varphi}

            \begin{lemma}\label{l: properness gf}
                The map $(m,n) \mapsto \HPq(mn)$ induces a smooth map
                $$\gf :  \;(M_R \cap H)\times_{M_R \cap H \cap P}(N_R\cap H)/(N_R \cap H \cap P)\to \faq$$
                which is proper.
            \end{lemma}

            The inclusion map $N_R \cap H \to N_R$ induces an embedding of $(N_R\cap H )/(N_R \cap H \cap P)$ onto a closed submanifold of $N_R / N_R \cap P.$  This embedding is equivariant for the conjugation action of $M_R \cap H \cap P.$ Accordingly, we may view
            $$(M_R \cap H)\times_{M_R \cap H \cap P} (N_R\cap H)/(N_R \cap H \cap P)$$
            as a closed submanifold of
            $$(M_R \cap H)\times_{M_R \cap H \cap P} N_R/(N_R \cap P).$$
            Thus, for the proof of Lemma \ref{l: properness gf}   it suffices to establish the following result.

            \begin{prop}\label{p: reduced properness result}
                The map $\psi: (m,n) \mapsto \HPq(mn)$ induces a smooth map
                \begin{equation}\label{e: bar psi}
                    \bar \psi: (M_R\cap H)\times_{M_R \cap H \cap P} N_R / (N_R \cap P) \to \faq.
                \end{equation}
                This map is proper.
            \end{prop}

                 Before we proceed with the proof of Proposition \ref{p: reduced properness result}
            we will first study the maps $M_R\cap H / M_R \cap H\cap P \to \faq$ and $N_R / (N_R \cap P) \to \faq $ induced by $\HPq.$

            \begin{lemma}\label{l: properness HRPq}
                The map $\fH^R_{P,\iq}:=\HPq|_{M_R\cap H}$ induces a smooth map $\bar \fH^R_{P,\iq}: (M_R\cap H)/ (M_R \cap H \cap P) \to \faq$ which is proper and has
                image equal to the cone $\Gamma_{\faq}(\gS^R_-),$ where
                $$
                \gS^R_- = \{\ga \in
                \gS(P) \cap
                \faqd\setmid
                \fg_\ga \not\subseteq
                \ker(\gs\Cartan - I)\}.
                $$
                In particular, the image is contained in the cone $\Gamma_{\faq}(\gS(P,\gs\Cartan)).$
            \end{lemma}

            \begin{proof}
             We start by noting that  $(M_R, M_R\cap H)$ is a reductive symmetric pair of the Harish-Chandra class, which is invariant under the Cartan involution $\Cartan.$ Furthermore, $\starfaR:= \fm_R \cap \fa$ is a maximal abelian subspace of $\fm_R\cap \fp$ (contained in $\faq$) and $M_R \cap P$ is a minimal parabolic subgroup of $M_R$ containing ${}^*A_R:= \exp{\starfaR}.$ Accordingly, by restriction the Iwasawa projection map $\HPq: H \to \faq$ induces the similar projection map $\HRPq: M_R\cap H  \to \faq$ which is the analogue of $\HPq$ defined relative to the data $M_R, M_R \cap K, P\cap M_R, H\cap M_R,$ in place of $G, K, P, H.$

             The ${}^{*}\fa_R$-roots in $N_P \cap M_R$ are precisely the restrictions of the roots from $\gS(P) \cap \faqd.$
               From this we see that  the minimal parabolic subgroup $P\cap M_R$ of $M_R$ is $\gs\Cartan$-stable.
               Hence, in view of \cite[Theorem 1.1, Lemma 3.3]{ban1986},
               the map $\bar\fH^R_{P,\iq}$ is proper and has image equal to the cone $\Gamma_{\faq}(\gS^R_-)$ given above. The final assertion now follows from the observation that $\gS(P) \cap \faqd \subseteq \gS(P, \gs \Cartan).$
            \end{proof}

            The following lemma is well known. For completeness of the exposition, we provide the proof.

            \begin{lemma}\label{l: Gindikin-Karpelevic}
                The Iwasawa map $\fH_P|_{\bar N_P}: \bar N_P \to \fa$ is proper. If $Q \in \cP(A),$ then
                $$\fH_P(N_Q \cap \bar N_P) =\Gamma_\fa(\gS(P)\cap \gS(\bar Q)).$$
            \end{lemma}
            \begin{proof}
                For the first assertion, let      $(\bar n_j)$ be sequence in $\bar N_P$ such that $\fH_P(\bar n_j)$ converges. Then $\bar n_j = k_j a_j n_j,$ with $k_j \in K,$ $a_j = \exp \fH_P(\bar n_j)$ and $n_j \in N_P.$ By passing to a converging subsequence, we may arrange that in addition the sequence  $(k_j)$ converges in $K.$ It follows that $ \bar n_j n_j^{-1}= k_j a_j $ converges in $G$. By   \cite[Lemma 39]{HC1958},
                the sequence      $(\bar n_j)$ converges.

                For the second assertion, we may assume $\gS(\bar Q) \cap \gS(P) \neq \emptyset$ and use the idea due to      S. Gindikin and F. Karpelevic \cite{GindKarp62}, to decompose $N_Q \cap \bar N_P$ by using a $P$-simple root in $\gS(\bar Q) \cap \gS(P).$ Let $\ga$ be such a root. Let $\fn_\ga =\fg_\ga + \fg_{2\ga}$ and $N_\ga = \exp \fn_\ga.$ Put $Q' = s_\ga Q s_\ga $ . Then, with the notation of Subsection \ref{H-mod-HP-decomposition},
                $$
                \gS_\circ(\bar Q) \cap \gS_\circ(P) =\{\ga\} \sqcup (\gS_\circ(\bar Q') \cap \gS_\circ(P)),
                $$ so that
                $$ N_Q \cap \bar N_P = \bar N_\ga  (N_{Q'} \cap \bar N_P) \simeq \bar N_\ga \times  (N_{Q'} \cap \bar N_P), $$
                in view of Lemma \ref{l: KDV}.
                Let $\bar n \in N_Q \cap \bar N_P.$ Then according to the above decomposition we may write $\bar n = \bar n_\ga \bar n',$ where $\bar n_\ga \in \bar N_\ga$ and $\bar n' \in N_{Q'} \cap N_P.$  Let $\fg(\ga)$ be the semisimple subalgebra generated by $\fn_\ga$ and $\bar \fn_\ga,$  and let $G(\ga)$ be the corresponding analytic subgroup of $G.$ By the Iwasawa decomposition of $G(\ga)$ for the minimal parabolic subgroup $P \cap G(\ga)$ we may write $\bar n_\ga = k_\ga a_\ga n_\ga$ with $k_\ga \in G(\ga) \cap K,$ $a_\ga \in \exp (\R H_\ga)$ and $n_\ga \in N_\ga.$ From application of Lemma
                \ref{l: KDV} we find that
                $$N_{Q'} \cap \bar N_P \simeq  N_{Q'}/ (N_{Q'} \cap N_P)$$
                and
                we see that there exists a diffeomorphism $\tau_{n_\ga}$ from
                $N_{Q'} \cap \bar N_P$ onto itself, such that
                $$ n_\ga \bar n' \in \tau_{n_\ga}(\bar n') N_P, \quad \mbox{\rm for all}\;\;\bar n' \in N_{Q'} \cap \bar N_P.$$
                This implies that
                $$\fH_P(\bar n_\ga n') =  \fH_P(a_\ga \tau_{n_\ga} (\bar n') a_\ga^{-1}) + \log a_\ga,$$
                and we see that
                $$\fH_P(\bar N_\ga(N_Q' \cap \bar N_P)) = \fH_P(N_Q' \cap \bar N_P) +\fH_P(\bar N_\ga).$$
                Now $\fH_P(\bar N_\ga)$ equals the image of $\bar N_\ga$ under the Iwasawa projection $\fH_\ga$ for the split rank 1 group $G(\ga)$ and the minimal parabolic subgroup $P \cap G(\ga).$ By     \cite[Thm.\ IX.3.8]{helgason1978}, which is based on ${\rm SU}(2,1)$-reduction, we see that $\fH_\ga(\bar N_\ga) = \R_{\geq 0} H_\ga.$ It follows that
                $$\fH_P(\bar N_\ga(N_Q' \cap \bar N_P)) = \fH_P(N_Q' \cap \bar N_P) + \R_{\geq 0}H_\ga.$$
                The proof is completed by induction on the number of elements in $\gS_\circ(\bar Q) \cap \gS_\circ(P).$
            \end{proof}

            The following lemma is      the second ingredient for the proof of Proposition \ref{p: reduced properness result}.

            \begin{lemma}\label{l: image HPq on NR}
                The Iwasawa map $\HPq|_{N_R} : N_R \to \faq$ factors through a proper map $N_R/N_R \cap N_P \to \faq$ with image equal to the cone \begin{equation}\label{e: cone faq P bar R}
                    \Gamma_{\faq}(\gS(P) \cap \gS(\bar R)).
                \end{equation}
                In particular, the image is contained in the cone $\Gamma_{\faq}(\gS(P, \gs \Cartan)).$
            \end{lemma}

            \begin{proof}
                We denote the induced map by $\fH.$ It follows by application of Lemma \ref{l: KDV} that the multiplication map $(N_R \cap \bar N_P) \times (N_R \cap N_P) \to N_R$ is a diffeomorphism. Let $\nu: N_R \cap \bar N_P \to N_R/{N_R\cap N_P}$ denote the induced diffeomorphism. Then $\fH \after \nu$ equals $\prq \after \fH_{P,R},$ where $\fH_{P,R}$ denotes the restriction of $\HP$ to $N_R \cap \bar N_P.$ This restriction is proper with image $\Gamma_{\fa} ( \gS(P) \cap \gS(\bar R)),$      by Lemma \ref{l: Gindikin-Karpelevic} above.
                In particular, the image is contained in the cone $\Gamma_{\fa} (\gS(P, \gs\Cartan)),$ by Lemma \ref{l: roots in gS P gs Cartan}. In view of Corollary \ref{c: cones and properness} (a) it now follows that $\fH \after \nu = \prq\after \fH_{P,R}$ is proper with image equal to (\ref{e: cone faq P bar R}). This implies the result.
            \end{proof}

            We proceed with a      final lemma needed for the proof of Proposition \ref{p: reduced properness result}.

            \begin{lemma}\label{l: first step properness bar psi}
                Let $\bar \psi$ be as in (\ref{e: bar psi}) and      let
                $$
                \bar \pr_1: (M_R\cap H)\times_{M_R \cap H \cap P} N_R / (N_R \cap P) \to (M_R\cap H)/(M_R \cap H \cap P)
                $$
                denote the map induced by projection onto the first component.

                Let $C \subseteq \faq$ be a compact set. Then      the set $\bar\pr_1(\bar \psi^{-1}(C))$      is relatively compact  in $(M_R \cap H)/(M_R \cap H\cap P).$
            \end{lemma}

            \begin{proof}
                Let $m \mapsto [m]$ denote the canonical projection $M_R \cap H \to (M_R \cap H) /(M_R \cap H \cap P).$ Let $(m_j)$ and $(n_j)$ be sequences in $M\cap H$ and $N_R,$ respectively, such that $\HP(m_j n_j) \in C$ for all $j.$ Then it suffices to show that the sequence $([m_j])$ in $(M_R \cap H) /(M_R \cap H \cap P)$ has a converging subsequence.

            \def\iq{{\rm q}}

                In accordance with the Iwasawa decomposition
                 $M_R = (M_R\cap K )(M_R\cap A) (M_R\cap N_P),$ we may decompose $m_j = k_j a_j \nu_j.$
                     Since $\fah \subseteq \fa_R = {\rm center}(\fl_R) \cap \fp,$ we have
                $\fm_R \cap \fa = \fa_R^\perp\cap \fa \subseteq \faq,$ so that
                $\log a_j = \HRPq(m_j).$

                The element $t_j = a_j \nu_j$ belongs to $M_R,$ hence $n_j' := t_j n_j t_j^{-1} \in N_R,$ for all $j.$ From $ m_j n_j = k_j n_j' a_j \nu_j $ it follows that
                $$\HPq(m_j n_j) = \HPq(k_j n_j') + \log a_j = \HPq(n_j') + \HRPq(m_j).$$
                We now note that both $\HPq(n_j')$ and $\HPq(m_j)$ belong to $\Gamma_{\faq}(P, \gs\Cartan)$ by Lemmas \ref{l: image HPq on NR} and
                 \ref{l: properness HRPq}. By application of Corollary \ref{c: cones and properness} we infer that the sequence $\HPq(m_j)$ is contained in a relatively compact subset of $\faq.$ By application of Lemma \ref{l: properness HRPq} it now follows that $([m_j])$ is contained in a relatively compact subset of $(M_R \cap H)/(M_R \cap H \cap P),$ hence contains a convergent subsequence.
            \end{proof}

            {\em Completion of the proof of Proposition \ref{p: reduced properness result}.}
            Let $C$ be a compact subset of $\faq$ and let $(m_j)$ be a sequence in $M_R \cap H$ and $(n_j)$ a sequence in $N_R$ such that $\bar\psi([(m_j, n_j)]) \in C$ for all $j.$ Then it suffices to show that the sequence of points
                $$[(m_j, n_j)]\in (M_R\cap H)\times_{M_R \cap H \cap P}N_R / (N_R \cap N_P)$$
                has a converging subsequence.

                In view of Lemma \ref{l: first step properness bar psi} we may pass to a subsequence of indices and assume that the sequence $([m_j])$ in $D:= (M_R\cap H)/(M_R \cap H \cap P)$ converges. Since the canonical projection $M_R \cap H \to D$ determines a principal fiber bundle, we may invoke a local trivialization to obtain a converging sequence $(\bp m_j)$ in $M_R\cap H$ such that $\bp m_j \in m_j (M_R\cap H \cap P)$ for all $j.$ Let $p_j \in M_R \cap H \cap P$ be such that $m_j = \bp m_j p_j$ for all $j.$ Then
                $$[(m_j, n_j)] = [(\bp m_j, \bp n_j)],$$
                with $\bp n_j = p_j n_j p_j^{-1}\in N_R.$

                Replacing the original sequence of points $(m_j, n_j)$ in this fashion if necessary, we may as well assume that the original sequence $(m_j)$ converges in $M_R\cap H.$ Let $m \in M_R \cap H$ be the limit of this sequence. As in the proof of Lemma \ref{l: first step properness bar psi} we may decompose $m_j = k_j a_j \nu_j$ and $m = k a \nu$ in accordance with the Iwasawa decomposition
                 $M_R = (M_R \cap K) (M_R \cap A) (M_R \cap N_P).$ Then $k_j \to k,$ $a_j \to a$ and $\nu_j \to \nu,$ for $j \to \infty.$ Put $t_j = a_j \nu_j$ and $n_j' = t_j n_j t_j^{-1}.$ As in the proof of Lemma \ref{l: first step properness bar psi} it follows that
                $$\bar \psi([m_j, n_j]) = \log a_j + \HPq(n_j').$$
                Since $(a_j)$ converges, it follows that the sequence $\HPq(n_j')$ is contained in a compact subset $C' \subseteq \faq.$ By Lemma \ref{l: image HPq on NR} it follows that the sequence $([n_j'])$ in $N_R/N_R \cap N_P$ is contained in a compact subset. Passing to a suitable subsequence of indices we may as well assume that the sequence
                $([n_j'])$ converges to a point $[n],$ for some $n \in N_R.$ It follows that
                $$[n_j] = [t_j^{-1}n_j' t_j] = t_j^{-1}\cdot[n_j' ] \to t^{-1}\cdot [n] = [t^{-1}n t],\quad (j \to \infty),$$ where $t = a \nu.$
                We conclude that the sequence $[(m_j, n_j)]$ converges with limit equal to
                 $[(m, t^{-1}n t)].$
            \qed

            We finish this section with a number of results that will be needed in Section \ref{s: section-induction}.

\begin{corollary}\label{c: properness with set C}
                Let $\cA$ be a compact subset of $A_\iq.$ Then
                \begin{enumerate}
                \itema  $ \HPq(ah) \in \HPq(\cA K) + \HPq(h), ${\ } for all $(a,h) \in \cA \times H;$
                \itemb the map $(a,h) \mapsto \HPq(ah)$
                induces a proper map $\cA \times H/H\cap P  \rightarrow \faq.$
                \end{enumerate}
            \end{corollary}

            \begin{proof}
            We first prove (a). Let $a \in \cA$ and $h \in H.$ We may decompose $h = k b n$ with
            $k \in K,$ $b \in A$ and $n \in N_P.$  Furthermore, $ak = k' a' n'$ with
            $k' \in K,$ $n' \in N_P$ and $\log a' \in \HP(\cA K).$
            Now
                $$
                ah =  a k b n = k' a' n' b n = k' a' b n''
                $$
                with  $n'' =b^{-1} n' b n\in N_P.$ It follows that
                $$
                \HPq(ah) = \prq (\log a' + \log b) \in  \HPq(\cA K) + \HPq(h).
                $$
                This establishes (a).

                Since $\HPq(\cA K)$ is compact, (b) follows from combining (a) with
                Proposition \ref{p: reduced properness result}.
              \end{proof}

            \begin{lemma}\label{l: inclusion HPq H}
                Let $P\in \cP(A).$ Then
                $\HPq(H)\subseteq \Gamma_{\faq}(\gS(P, \gs\Cartan)).$
            \end{lemma}
            \begin{proof}
               By   (\ref{e: H decomposition}) we have
                $$H = (H\cap K) (H\cap N_R) (H \cap L_R) \subseteq K N_R (H \cap L_R).$$
                Fix $h \in H,$ then we may write $h = k n_R h_L$ with $k \in K,$ $n_R \in N_R$ and $h_L \in (H \cap L_R).$ The group $P\cap L_R$ is a minimal parabolic subgroup of $L_R,$ containing $A.$ In accordance with the associated Iwasawa decomposition for $L_R,$ we may write $h_L = k_L a_L n_L$ with $k_L \in K\cap L_R,$ $a_L \in A$ and $n_L \in N_P\cap L_R.$ Since $L_R$ normalizes $N_R,$ it follows that
                $$h = k n_R k_L a_L n_L \in K n_R' a_L n_L$$
                with $n_R' \in N_R.$ We now observe that $n_R' \in K b N_P$ with $b = \exp \HP(n_R').$ Thus,
                $h \in K b a_L N_P.$
                It follows that
                \begin{equation}\label{e: inclusion HPh}
                    \HPq(h) = \prq(\log b + \log a_L) \in \HPq(N_R) + \HPq(H \cap L_R).
                \end{equation}
                Since $L_R \cap H = (M_R\cap H)(A\cap H),$ we have $\fH_{P,{\rm q}}(H\cap L_R)=\fH_{P,{\rm q}}(H\cap M_R).$ The result now follows from (\ref{e: inclusion HPh})
                by applying
                      Lemmas \ref{l: properness HRPq} and \ref{l: image HPq on NR}.
            \end{proof}

           \begin{lemma}\label{l: included cones}
	    	Let $V$ be a finite dimensional real vector space (or more generally a real locally convex Hausdorff space), $\Gamma_1$ a convex cone in $V$, $\Gamma_2$ a closed convex cones in $V$ and $B\subseteq V$ a bounded subset. If {\ }$\Gamma_1\subseteq B+\Gamma_2$ then $\Gamma_1\subseteq \Gamma_2$.
	    \end{lemma}
	    \begin{proof}
	    	Let $\gamma\in\Gamma_1$. Then for any positive integer $n \geq1$
                  we have that $n\gamma \in \Gamma_1\subseteq B+\Gamma_2$,
                  hence
                  $$ \gamma = b_n/n + \gamma_n,$$
                  with $b_n \in B$ and $\gamma_n \in \Gamma_2.$ As $B$ is bounded,
                  $b_n/n \to 0$ and we conclude that $\gamma_n \to \gamma,$ for $n \to \infty.$
                  Since  $\Gamma_2$ is closed,
                  it follows that $\gamma \in \Gamma_2.$
	    \end{proof}

            \begin{corollary}\label{c: half space result}
                Let $P\in \cP(A).$ Then for each $a \in \Aq,$ the set $\HPq(aH)$ does not contain any line of $\faq.$
            \end{corollary}
            \begin{proof}
From Corollary \ref{c: properness with set C} (a) combined with Lemma \ref{l: inclusion HPq H}
we see that
\begin{equation}
\label{e: inclusion HPq aH}
\HPq(aH ) \subseteq  \HPq(aK) + \Gamma_{\faq}(\gS(P, \gs\theta)).
\end{equation}
Arguing by contradiction, assume that $\HPq(aH)$ contains a line of the form
$Z + \R Y,$ with $Y,\,Z \in \faq,$ $Y \neq 0.$
                 Then $\R Y \subseteq (-Z) + \HPq(aK) + \Gamma(\Sigma(P, \gs,\Cartan)),$
                 and by Lemma \ref{l: included cones} we conclude that $\R Y \in \Gamma(\Sigma(P, \gs \Cartan)).$
This implies that
                 $$
                 Y \in \Gamma(\gS(P, \gs\Cartan)) \cap - \Gamma(\gS(P, \gs\Cartan)) = \{0\},
                 $$
                 contradiction.
                 \end{proof}

\section{Critical points of components of the Iwasawa map}\label{s: the set critical points}

            In this section we assume that $P \in \cP(A)$ is a fixed minimal parabolic subgroup and that $a$ is a fixed element of $\Aq.$ We will investigate the critical sets of vector components of the map $h \mapsto \HPq(ah),$ $H \to \faq.$ For this, let $X \in \faq$, and consider the function $F_{a,X}:  H \to  \R$ defined by
            \begin{equation}\label{e:defi set FaX}
                F_{a,X}(h)=\inp{ X}{\HP(ah)}=\inp{X}{\HPq(ah)}=B(X,\HPq(ah)).
            \end{equation}
            The second equality is valid because $\fah$ and $\faq$ are perpendicular with respect to the inner product $\inp{\dotvar}{\dotvar}$, while the third holds because $\HPq(ah)\in\faq\subset\fp$. We start with a result on derivatives of the function
            \begin{equation}\label{e: function F_X from G}
                F_{X}: G \to \R,\quad  g\mapsto\inp{X}{\HP(g)}.
            \end{equation}
            In order to formulate it, we need a bit of additional notation. If $F \in C^\infty(G)$ and $U \in \fg,$ we define:
            $$F(g;U) = R_UF(g) := \left. \frac{d}{dt}\right|_{t = 0} F(g \exp ( tU)). $$

            The following result and its proof can be found in      \cite[Cor.\ 5.2]{duiskolkvar1983}. See also \cite[Cor.\ 4.2]{ban1986}.

            \begin{lemma}\label{l: derivative of FX}
                Let $g\in G$ and $U\in\fg$. Then
                $$F_X(g;U) =B(\Ad(\tau(g))U,X) = B(U,\Ad(\nu(g)^{-1})X),$$
                where we have used the decompositions $g=k(g)\tau(g)$ and $\tau(g)=a(g)\nu(g)$, according to the Iwasawa decomposition $G=KAN_P$.
            \end{lemma}
     We define the set of regular elements in $\Aq$ by $\Aqreg: = \exp(\faqreg),$ see (\ref{e: aq regular}).
If $X \in \faq$ we denote by $G_X$ the centralizer of $X$ in $G$ and put
\begin{equation}
\label{e: centralizer in NP}
N_{P,X} := N_P \cap G_X.
\end{equation}

            \begin{lemma}\label{l: characterization critical point FaX}
                Let $a \in \Aq$ and let $X\in\faq.$ The point $h\in H$ is a critical point for the function $F_{a,X}$ if and only if $ah=kbn$ for certain $k\in K$, $b\in A$ and $n\in N_{P,X}(N_P\cap H)$.
            \end{lemma}

            \begin{proof}
                Let $h\in H.$ Then $h$ is a critical point for the function $F_{a,X}$ if and only if
                \begin{equation}\label{e: stationary condition F}
                    \forall U \in \fh : \;\; 0 = F_{a,X}(h;U)= B(U,\Ad(\nu(ah)^{-1})X).
                \end{equation}
                Since $\fh$ and $\fq$ are perpendicular
                with respect to $B,$ see text above (\ref{e: inner product from B}), the condition (\ref{e: stationary condition F}) is equivalent to the assertion that $\Ad(\nu(ah)^{-1})X\in\fq.$ Write $n = \nu(ah)$ and decompose $n = n_+ n_H $ according to the decomposition $N_P = N_{P,+} (N_P \cap H)$ of      Proposition \ref{p: N=N_+(N_P cap H) decomposition}.
                Since $\Ad(n_H)$ normalizes $\fq,$ the above condition is equivalent to $\Ad(n_+)^{-1} X \in \fq.$ Now apply the lemma below to see that the latter is equivalent to $n_+ \in N_{P,+} \cap N_{P,X}.$ It follows that (\ref{e: stationary condition F}) is equivalent to $n \in N_{P,X} (N_P \cap H).$
            \end{proof}

            \begin{lemma}
                 Let $n \in N_{P,+}$ {\rm (}cf. Prop.\ \ref{p: N=N_+(N_P cap H) decomposition}{ \rm )} and $X\in \faq.$ Then
                $$\Ad(n) X \in \fq \iff \Ad(n)X = X.$$
            \end{lemma}

            \begin{proof}
                The implication `$\Leftarrow$' is obvious. Thus, assume that $\Ad(n)X \in \fq.$ We may write $n=\exp(U)$, where      $U\in\fn_{P,+}$. Then by nilpotence of      $\fn_{P,+},$
                $$
                \Ad(n)X=e^{\ad(U)}X
                \in X + \fn_{P,+}
                $$
                                By assumption, $\Ad(n)X -X \in \fq.$ Since obviously $\gs(\fn_{P,+}) \cap \fn_{P,+} = 0,$ it follows that
                $\fn_{P,+}  \cap \fq = 0$ and we infer that $\Ad(n) X = X.$
            \end{proof}

            \medbreak
            Given $X \in \faq$ we agree to denote by $\cC_{a,X}$ the set of critical points for the function $F_{a,X}$. The remainder of this section will be dedicated to proving the following description of this set in case $a$ is regular. We   recall the definitions of the Weyl groups $\WKaq$ and $\WKHaq$ from (\ref{e: Weyl groups}) and Definition \ref{d: WKHaq}.
           \begin{remark}\rm
                In the following we will use the notation
                $$
                a^w: = w^{-1}\cdot a
                $$
                for $a \in \Aq$ and $w \in \WKaq.$ This notation has the advantage that
                $(a^v)^w = a^{vw}$ and $(a^w)^\beta = a^{w\beta},$
                for $v,w \in \WKaq$ and $\beta \in \gS(\fg,\faq).$ In particular, $\Ad(a^w) = a^{w\beta} I$ on $\fg_\beta.$

                We will use  similar notation for $a \in A$ and $w \in \WKa.$
            \end{remark}

            \begin{lemma}
            \label{l: set of critical points}
                Let   $a \in \Aqreg$ and $X \in \faq.$  Then
                \begin{equation}\label{e: set of critical points}
                    \cC_{a,X}=\bigcup_{w\in \WKHaq}w H_X(N_P\cap H).
                \end{equation}
            \end{lemma}

            \begin{proof}
                Let $x_w$ be a representative of $w$ in $N_{K\cap H}(\faq)$,      let $h\in H_X$ and $n_P \in N_P\cap H$. Then, with notation as in Lemma \ref{l: derivative of FX},
                $$\nu(a x_w hn_P)=\nu(x_w^{-1}a x_w h n_P)=\nu(a^w h n_P)=\nu(a^w h)n_P.$$
                The element $a^w h$ belongs to $G_X,$ and according to
                 \cite[Eqn.\ (2.6)]{duiskolkvar1983},
                \[G_X\simeq K_XAN_{P,X}.\]
                Thus, $\nu(a^w h)\in N_{P,X}$ and it follows that
                $\nu(ax_w hn_P)\in N_{P,X}(N_P\cap H)$. This proves that the set on the right-hand side of (\ref{e: set of critical points}) is included in the set on the left-hand side. It remains to prove the converse inclusion.

                Let $h\in\cC_{a,X}$. Then by Lemma \ref{l: characterization critical point FaX} we may write $ah=kbn_Xn_H$ with $k \in K$, $b \in A,$ $n_X \in N_{P,X}$ and $n_H \in N_P \cap H.$ From this we see that $k^{-1}ahn_H^{-1}=bn_X\in G_X$. The element $h':=hn_H^{-1}$,  belongs to $H$. In view of the Cartan decomposition $ H = (K\cap H)\times\exp(\fp\cap\fh)$, we may write $h'=h_1h_2$, where $h_1 \in K\cap H$ and $h_2 \in \exp(\fp\cap\fh)$. Then
                \begin{equation}\label{e: deco in GX}
                    k^{-1}ah_1h_2=k^{-1}h_1(h_1^{-1}ah_1)h_2\in G_X.
                \end{equation}

                By \cite{Most55},
                 the group $G$ decomposes as
                $$G \simeq K \times \exp(\fp \cap \fq)\times\exp(\fp \cap \fh).$$
                According to  \cite[Thm.\ 5]{Most55},
                 $G_X$ has a similar decomposition
                $$ G_X \simeq K_X\times\exp(\fp \cap \fq_X)\times\exp(\fp \cap \fh_X).$$
                By the uniqueness properties of the latter decomposition it follows from (\ref{e: deco in GX}) that $k^{-1}h_1\in K_X$, $h_1^{-1}ah_1\in \exp(\fp \cap \fq_X )$ and $h_2\in\exp(\fp\cap\fh_X)$.

                We note that $\gs\Cartan$ fixes $X$ hence leaves the centralizer $G_X$ invariant. The fixed point group $G_{X,+}$ of this involution in $G_X$ admits the Cartan decomposition
                $$G_{X,+} \simeq (K\cap H_X) \times \exp(\fp \cap \fq_X).$$
                Obviously $\faq$ is a maximal abelian subspace of $\fp \cap \fq_X.$ Hence, every element of the latter space is conjugate to an element of $\faq$ under the group $(K\cap H_X)^\circ.$ We infer that there exists an element $l \in (K \cap H_X)^\circ$ such that
                \begin{equation}\label{e: conjugate in Aq}
                    l^{-1}h_1^{-1}ah_1l \in \Aq.
                \end{equation}
                Since $a$ was assumed to be regular for  $\Sigma(\fg,\faq)$, it follows that $a$ is regular for $\Sigma(\fg_+,\faq)$ as well. Hence, (\ref{e: conjugate in Aq}) implies that the element $h_1l \in K \cap H$ normalizes $\faq.$ It follows that $h_1\in N_{K\cap H} (\faq) (K\cap H_X)$. Then, $$h' = h_1h_2 \in N_{K\cap H}(\faq)(K\cap H_X)\exp(\fp\cap\fh_X)=N_{K\cap H}(\faq)H_X$$
                and we conclude that $hn_H^{-1} \in  N_{K\cap H}(\faq)H_X$. This finally implies that
                $$h\in N_{K\cap H}(\faq)H_X(N_P\cap H),$$
                which concludes the proof.
            \end{proof}

        \section{Properties of the set of critical points}\label{s: properties of the set of critical points}
            As in the previous section, we assume that $P \in \cP(A)$ and that $a$ is a regular point in $\Aq.$ In the previous section we defined the function $F_{a,X}: H \to \R,$ for $X \in \faq,$  by (\ref{e:defi set FaX}) and we determined its set of critical points $\cC_{a,X},$ see (\ref{e: set of critical points}). The purpose of the present section is to study this set in more detail.

            We start with the following lemma.
            \begin{lemma}
            \label{l: gf induces injective immersion}
                The map $\gf: H_X \times (N_P \cap H) \to H$ given by $(h,n) \mapsto hn$ induces an injective immersion
                $$\bar\gf: H_X \times_{\scriptscriptstyle N_P \cap H_X} (N_P \cap H) \to H$$
                with image $H_X(N_P\cap H).$
            \end{lemma}

            \begin{proof}
                The group $H_X \times (N_P\cap H)$ has a natural left action on $H$ given by the formula: $(h,n)\cdot x = h x n^{-1}.$ The set $H_X (N_P \cap H)$ is the orbit for this action through the identity element $e$ of $H.$ Let $F$ be the stabilizer of $e$ for this action.  Then it follows that the map $(h,n) \mapsto (h,n)\cdot e = h n^{-1}$ factors through an injective immersion $(H_X \times (N_P\cap H))/F \to H$ with image $H_X (N_P \cap H).$ The stabilizer $F$ consists of the elements $(h,h)$ with $h \in H_X \cap N_P.$ To complete the proof of the lemma, we note that the map $(h, n) \mapsto (h, n^{-1})$ induces a diffeomorphism $H_X\times_{N_P \cap H_X} (N_P \cap H) \to (H_X\times (N_P \cap H))/F.$
            \end{proof}

            \begin{lemma}\label{l: cCaX}
                Let $X \in \faq.$ Then the set $\cC_{a,X}$ is closed in $H.$ Moreover,
                the following holds.
                  \begin{enumerate}[{\rm (a)}]
                  \item If {\ }$\fh_X + (\fn_P \cap \fh) = \fh$ {\ } then {\ } $\cC_{a,X} =  H.$
                  \item If {\ } $\fh_X + (\fn_P \cap \fh) \subsetneq \fh$ {\ } then {\ } $\cC_{a,X}$
                  is a finite union of lower dimensional injectively immersed submanifolds.
                  \end{enumerate}
            \end{lemma}

\begin{proof}      Since $\cC_{a,X}$ is the set of critical points of the smooth function $F_{a,X},$ it is closed.

From Lemma \ref{l: set of critical points} combined
with Lemma \ref{l: gf induces injective immersion} it follows that $\cC_{a,X}$
is a finite union of injectively immersed submanifolds of dimension
$d_X: = \dim (\fh_X + ( \fn_P \cap \fh)).$ From this, (b) is immediate.

For (a) we assume the hypothesis to be fulfilled, or equivalently, that $d_X = \dim(H).$ Then $\cC_{a,X}$ is open in $H.$ Since this set is also closed in $H$, and contains
$H_X (N_P \cap H),$ it follows that $\cC_{a,X} \supseteq H^\circ.$
From      Lemma \ref{l: set of critical points} it follows that $\cC_{a,X}$ is left $N_{K\cap H}(\faq)$-invariant, so that $\cC_{a,X} \supseteq N_{K\cap H}(\faq) H^\circ.$
Since $H$ is essentially connected, the latter set equals $H,$ see
(\ref{e: essentially connected}).
\end{proof}

 \begin{lemma}\label{l: conditions on X}
                Let $X \in \faq.$ Then the following assertions are equivalent:
                \begin{enumerate}[{\rm (a)}]
                    \item $\fh=  \fh_X+ (\fn_P\cap\fh);$
                    \item $\forall \ga \in \gS(\fg, \fa): \ga(X) = 0.$
                 \end{enumerate}
                 \end{lemma}

            \begin{proof}
First assume (b). Then $\fg_X = \fg$ and (a) follows.
We will prove the converse implication by contraposition. Thus, assume that (b) does not hold.
Then there exists a root $\gb \in \gS(\fg, \fa)$ such that $\gb(X) \neq 0.$
By changing sign if necessary, we may in addition arrange that $\gb \in \gS(P).$

Given a subset   $\cO \subseteq \gS(\fg, \fa)\cup \{0\},$ we agree  to write
\begin{equation}
\label{e: defi fg cO}
\fg_\cO = \oplus_{\ga \in \cO}\; \fg_\ga.
\end{equation}
In particular, we see that $\fn_P = \fg_{\gS(P)}.$
 We also agree to write
 $\cO^\gs: = \cO \cap \gs(\cO).$  Then using
$\gs (\fg_\ga) = \fg_{\gs \ga}$ we readily see that
                 \begin{equation}\label{e: intersection fg omega with fh}
                 \fg_{\cO} \cap \fh = (\fg_{\cO^\gs})^\gs = \oplus_{\omega \in \cO^\gs/\{1, \gs\}} \;\; (\fg_\omega)^\gs ;
                 \end{equation}
here $\cO^\gs/\{1, \gs\}$ denotes the set of orbits for the action on $\cO^\gs$ of the subgroup
$\{1, \gs\}$ of $\Aut(\fg).$ If we apply (\ref{e: intersection fg omega with fh}) to the set $\cO_X:=
\{\ga \in \gS(\fg, \fa) \setmid \ga(X) = 0\} \cup \{0\},$ we find
$$
\fh_X = \oplus_{\omega \in \cO_X/\{1, \gs\}}\;\; (\fg_\omega)^\gs.
$$
We note that $\gS(P)^\gs = \gS(P, \gs),$ so that
$$
\fn_P \cap \fh =\fg_{\gS(P,\gs)}\cap \fh.
$$

We now consider the set $\cO_\gb: = \{\gb, \gs\gb, -\gb, -\gs \gb\}.$ Since
$\cO_X \cap \cO_\gb = \emptyset,$ it follows from the above that
                 \begin{equation}
                 \label{e: key inclusion tangent space}
                ( \fh_X + (\fn_P\cap \fh))\cap \fg_{\cO_\gb}=
                \fn_P  \cap \fh \cap \fg_{\cO_\gb}
                = (\fg_{\gS(P,\gs) \cap \cO_\gb})^\gs.
                \end{equation}
                On the other hand,
                $$
                \fh \cap \fg_{\cO_\gb}= (\fg_{\cO_\gb})^\gs.
                $$
                From   $\gb(X) \neq 0$ it follows that $\gb \notin \fahd.$ If $\gb\in \faqd$ then
                $\gS(P, \gs) \cap \cO_\gb = \emptyset$ and if $\gb \notin \faqd$ then
                $\gS(P,\gs) \cap \cO_{\gb} \subseteq\{\gb , \gs \gb\}.$
                In any case,
                $\gS(P,\gs)\cap \cO_\gb$ is a proper $\gs$-invariant subset of
                $\cO_\gb.$ By application of (\ref{e: intersection fg omega with fh}) it now follows  that
                $$
                (\fg_{\gS(P,\gs) \cap \cO})^\gs\subsetneq (\fg_\cO)^\gs.
                $$
                Using (\ref{e: key inclusion tangent space})   we infer that (a) is not valid.
                \end{proof}

We agree to write
\begin{equation}
\label{e: defi S}
S: = \faq \setminus \cap_{\ga \in \gS(\fg, \faq)} \ker \ga.
\end{equation}

\begin{remark}
\label{r: span and S}
If $\gS(\fg, \faq)$ spans $\faq$ then it follows that $S = \faq\setminus\{0\}.$
\end{remark}

\begin{corollary}
\label{c: characterization S}
$
                S=\{X\in\fa_\iq\setmid\cC_{a,X} \subsetneq H \}. $
\end{corollary}

\begin{proof}
Let $X \in \faq.$
In the situation of Lemma \ref{l: cCaX} (b) the set $\cC_{a,X}$ is a countable union
of lower dimensional submanifolds, hence nowhere dense by the Baire category theorem. Thus, by application of Lemmas \ref{l: cCaX} and \ref{l: conditions on X}
it follows that      $\cC_{a,X} \subsetneq H \iff X \in S.$
\end{proof}

 For
 each $Z \in \faq,$ let $\gS(Z)$ denote the collection of roots in      $\gS(\fg, \fa_\iq)$ vanishing on $Z.$ We define the equivalence relation $\sim$ on $\faq$ by
                $$X \sim Y \iff \gS(X) = \gS(Y).$$
                Then clearly, $\sim$ has finitely many equivalence classes in $\faq$ and
                $$X \sim Y \iff  G_X = G_Y.$$
The class of $0$ is given by    $[0] = \cap_{\ga \in \gS(\fg, \faq)} \ker \ga$ and
$S$ is the union of the remaining finitely many equivalence classes for $\sim$. Furthermore, the set $\cC_{a, X}$ depends on $X \in S$ through the centralizer $G_X,$ hence through  the equivalence class $[X]$ for $\sim.$ Accordingly, we will also write $\cC_{a, [X]}$ for this set.

We define
 \begin{equation}\label{e: def-critical-set}
                \cC_a: =\cup_{X\in S}\; \cC_{a,X}.
            \end{equation}

\begin{lemma}{\ }
\label{l: property cCa new}
\begin{enumerate}[\rm (a)]
\item
There exists a finite subset $S_0 \subseteq  S$ such that (\ref{e: def-critical-set}) is valid for the union over  $S_0$ in place of $S.$
 \item
 The set $\cC_a$ is closed and a finite union of
lower dimensional injectively immersed submanifolds of $H.$
\item The set $\cC_a$ is nowhere dense in $H.$
\end{enumerate}
            \end{lemma}

            \begin{proof}
By the discussion preceding the lemma, $\cC_a$ is the union of the sets
$\cC_{a,[X]},$ for $[X]\in S/\sim.$ Since the latter set is finite,
assertion (a) follows with
$S_0$ a complete set of representatives for $S/\sim.$
Assertion (b) now follows by application of
Corollary \ref{c: characterization S} and Lemma \ref{l: cCaX}. Assertion (c)  follows from (b) by application of the Baire category theorem.
            \end{proof}

The following result  illustrates the importance of the set $\cC_a.$

\begin{prop}
\label{p: submersiveness Fa}
The set $H \setminus \cC_a$ is open and dense in $H.$
Assume that $\gS(\fg, \faq)$ spans $\faqd.$
Then the map $F_a: h \mapsto \HPq(ah),$ $H \to \faq$ is submersive at all points of $H \setminus \cC_a.$
\end{prop}

\begin{proof}
The first assertion is a consequence of Lemma \ref{l: property cCa new}.

Let $h_0 \in H \setminus \cC_a.$ Then for every $X \in S$ the point $h_0$ is not critical for the function $F_{a,X}.$ As $S = \faq\setminus \{0\},$ see Remark \ref{r: span and S}, it follows that $F_a: h \mapsto \HPq(ah)$ is submersive at $h_0.$
\end{proof}

            \begin{lemma}\label{l: HP of aH and aC}
                Let $P \in \cP(A)$ and    $a \in \Aqreg.$ Then the following assertions are valid.
                \begin{enumerate}
                    \itema The sets $\HPq(aH)$ and $\HPq(a\cC_a)$ are closed in $\faq.$
                    \itemb If $\,\gS(\fg, \faq)$ spans $\faqd$ then the set $\HPq(aH)\setminus\HPq(a\cC_a)$  is open and closed
in  $\faq\setminus\HPq(a\cC_a).$
                \end{enumerate}
            \end{lemma}
            \begin{proof}
                For $\cA \subseteq A$ compact, the map $\cA \times H/(H\cap P)\to\faq, (b, [h]) \mapsto \HPq(bh)$ is proper, hence closed; see Corollary \ref{c: properness with set C}. In particular, it follows that $\HPq(aH)$  is closed in $\faq$.

                It follows from Lemma \ref{l: property cCa new}
that $\cC_a$ is closed in $H.$ Moreover, $\cC_a$ is a countable union of lower dimensional submanifolds of $H.$ Thus, by the Baire property, $\cC_a$ has empty interior in $H.$ In particular, it is a proper subset of $H.$

                Furthermore, the set $\cC_a$  is right $H\cap P$-invariant, hence has closed image in $H/H\cap P.$ It follows that $\HPq(a\cC_a)$ is closed in $\faq.$ This establishes (a).

By Proposition \ref{p: submersiveness Fa} the map $F_a: h \mapsto \HPq(ah)$ is submersive at the points of $H \setminus \cC_a.$ Hence $\HPq(a(H \setminus \cC_a))$ is open in $\faq.$ It follows that
                \begin{equation}\label{e: complement of H aC}
                    \HPq(aH) \setminus \HPq(a\cC_a) = \HPq(a(H \setminus \cC_a))\setminus \HPq(a\cC_a)
                \end{equation}
                is open in $\faq$ hence in $\faq\setminus\HPq(a\cC_a).$ Finally, since $\HPq(aH)$ is closed, the  first set in (\ref{e: complement of H aC}) is closed in $\faq\setminus \HPq(a\cC_a).$ We conclude that the set (\ref{e: complement of H aC}) is both open and closed in $\faq \setminus \HPq(a\cC_a).$
            \end{proof}
        \begin{lemma}\label{l: at-least-one-connected-component}
            Assume that $\gS(\fg, \faq)$ spans $\faqd.$ Then $\fH_{P,{\rm q}}(aH)\setminus \fH_{P,{\rm q}}(a\cC_a)\neq\emptyset.$
        \end{lemma}
        \begin{proof}
            Under the assumption that $\gS(\fg, \faq)$ spans $\faqd$, the map $\fH_{P,{\rm q}}:a H\to\faq$ is submersive except at points of $\cC_a$. The set $H\setminus\cC_a$ is open and non-empty. Thus, $\fH_{P,{\rm q}}(a(H\setminus \cC_a))$ is open and non-empty. By Sard's Theorem, $\fH_{P,{\rm q}}(a\cC_a)$ has measure zero. This implies that
            \[\fH_{P,{\rm q}}(a(H\setminus\cC_a))\setminus\fH_{P,{\rm q}}(a\cC_a)\neq\emptyset,\]
            and hence
            \[\fH_{P,{\rm q}}(aH)\setminus \fH_{P,{\rm q}}(a\cC_a)\neq\emptyset.\]
        \end{proof}
        \begin{remark}
            The lemma can readily be extended to the case that $\gS(\fg, \faq)$ does not span $\faqd$.
        \end{remark}

    \section{The computation of Hessians}\label{s: computation of hessians}
        We retain the assumption that $P \in \cP(A).$     Furthermore, we assume that $a \in \Aqreg$ and $X \in \faq.$
In this section we will compute the Hessian of the function $F_{a,X}:H\to\R, $ defined in (\ref{e:defi set FaX}), at   all points of  its critical locus $\cC_{a,X}.$

        Given $U \in \fh,$ we denote by $R_U$ the associated left-invariant vectorfield on $H$ defined by
        $$
        R_U(h) = dl_h(e)U =
        \frac{\partial}{\partial t} (h \exp tU)|_{t=0}, \qquad (h \in H).
        $$
        The associated derivation on $C^\infty(H)$ is denoted by the same symbol.

        If $f: H \to \R$ is a $C^2$-function with critical point at $h,$ then its Hessian at $h$ is the symmetric bilinear form $H(f)(h)= H(f)_h$ on $T_hH$ given by
        $${\rm H}(f)_h(R_U(h), R_V(h)) := R_U R_V f(h) =\partial_s \partial_t f(h \exp sU \exp tV)|_{s= t = 0},$$
        for $U,V \in \fh.$
        \begin{lemma}\label{l: formula for Hessian}
            Let $a\in \Aq,$ $X\in\faq$ and $h \in H.$ Then for all $U,V\in\fh$ we have:
            $$ R_U R_VF_{a,X}(h)=B(U,L_{a,X,h}(V))=-\inp{U}{\theta L_{a,X,h}(V)},$$
            where $L_{a,X,h}:\fh\to\fh$ is the linear map given by
            \begin{equation}
\label{e: formula L}
L_{a,X,h}(V)=-\Ad(h^{-1})\circ\pi_{\fh}\circ \Ad(a^{-1})\circ \Ad(k_a(h))\circ \ad(X)\circ E_{\fk}\circ \Ad(\tau(ah))V.
\end{equation}
            Here $\pi_{\fh}:\fg \to\fh$ denotes the projection according to the decomposition $\fg=\fh \oplus \fq$ and $E_{\fk}:\fg\to\fk$ is the projection associated with the Iwasawa decomposition $\fg=\fk\oplus \fa\oplus \fn_P$. The notation $k_a(h)$ is used to express the $K$-part of the element $ah$ with respect to the Iwasawa decomposition $G = K A N_P.$ Finally, $\tau(ah)$ denotes the $(AN_P)$-part of $ah$ with respect to the same Iwasawa decomposition.
        \end{lemma}
        \begin{proof}
            By \cite[Lemma 5.1]{ban1986}, see also \cite{duiskolkvar1983}, we obtain that for $x\in G$ and $U,V\in\fg,$
            \begin{equation*}
                R_U R_VF_{X}(x)=B([\Ad(\tau)U,\,E_{\fk}\circ\Ad(\tau)V],X),
            \end{equation*}
            where $F_X$ is the function defined in (\ref{e: function F_X from G}) and where $\tau:=\tau(x)$.  Therefore,
            \begin{eqnarray*}
                \lefteqn{R_U R_VF_{X}(x)=-B(\Ad(\tau)U,\,\ad X\circ E_{\fk}\circ\Ad(\tau)V)}\\
                &=& -B(U,\,\Ad(\tau)^{-1}\circ \ad X\circ E_{\fk}\circ\Ad(\tau)V).
            \end{eqnarray*}
            We can restrict now to the case where $x=ah$ and $U,V\in\fh$. Since $F_{a,X}(h)=F_X(ah)$, we obtain
            \begin{equation}\label{e:hessian-intermed} R_U R_VF_{a,X}(h)=R_U R_VF_X(ah)=B(U,\,-\pi_{\fh}\circ\Ad(\tau)^{-1}\circ \ad X\circ E_{\fk}\circ\Ad(\tau)V).\end{equation}
            Since  $ah=k_a(h)\tau(ah)$, it follows that  $\tau^{-1} = \tau(ah)^{-1}=h^{-1}a^{-1}k_a(h)$ and by applying $\Ad$ to this equality we obtain
            $$
            \Ad(\tau)^{-1}=\Ad(h^{-1})\Ad(a^{-1})\Ad(k_a(h)).
            $$
            We complete the proof by substituting this equality in (\ref{e:hessian-intermed}) and observing that $\pi_\fh$ commutes with $\Ad(h^{-1}).$
  \end{proof}

    \section{The transversal signature of the Hessian}\label{s: transversal signature of the hessian}
In this section we fix $P \in \cP(A),$ $a \in \Aqreg$ and
$X \in \faq.$ We will study the behavior of the Hessian $H(F_{a,X})_h$ of the function $F_{a,X}: H \to \R$ defined in
(\ref{e:defi set FaX}) at each point $h$ of its critical set $\cC_{a,X}.$ This Hessian is a symmetric bilinear form on $T_h H.$ Its kernel at $h$ is by definition equal to the following linear subspace of $T_h H,$
        $$\ker (H(F_{a,X})(h)): = \{ V \in T_h H \setmid  H(F_{a,X})(h)(V,\dotvar) = 0\}.$$
        By symmetry, the Hessian induces a non-degenerate symmetric bilinear form $\bar H(F_{a,X})(h)$ on the quotient space $T_h H / \ker (H(F_{a,X})(h)).$ For each $w \in \WKHaq$ we select a representative $x_w \in N_{K\cap H}(\faq).$ The set
$$
\cC_{a,X,w}:= x_w H_X (H\cap N_P)
$$
is an injectively immersed submanifold of $H,$
see Lemma \ref{l: gf induces injective immersion}.
In particular this set has a well-defined tangent space at each of
its points. We will show that the Hessian of $F_{a,X}$ is transversally non-degenerate along $\cC_{a,X,w}.$

\begin{lemma}\label{l: Dana s nemesis lemma}
            Let $w \in \WKHaq.$ Then at each point $\bar{h} \in \cC_{a,X,w}$ the kernel of the Hessian $H(F_{a,X})(\bar{h})$ equals the tangent space $T_{\bar{h}}\cC_{a,X,w}.$
\end{lemma}

        The proof of this lemma will make use of Lemma \ref{l: kernel-of-L} below.
In that lemma, $L_{a,X,h}\in \End(\fh)$
is defined as in (\ref{e: formula L}).
Let $\bar{k}_a:= \pi\circ k_a:H\to K/M$, where $k_a: H \to K$ is defined as in Lemma \ref{l: formula for Hessian}  and where $\pi$ denotes the canonical projection $K \to K/M.$

        \begin{lemma}\label{l: kernel-of-L}
             Let $h\in H_X^{\circ}$
and $V\in \fh$. Then the  following statements are equivalent.
             \begin{enumerate}
                \itema $V\in\ker L_{a,X,h}$,
                \itemb ${\rm d} (l_{{k}_a(h)^{-1}}\circ \bar{k}_{a}\circ l_h)(e)(V)\in\mathfrak{k}_X/\fm,$
                \itemc $V\in\fh_X+(\fh\cap\fn_P)$.
             \end{enumerate}
        \end{lemma}
        \begin{proof}
            First, we prove that (a) $\Longrightarrow$ (b). Assume (a) holds. In view of (\ref{e: formula L}) this is equivalent to
            \begin{equation}\label{e: belonging to fq}\Ad(a^{-1})\circ\Ad(k_a(h))\circ \ad(X)\circ E_{\mathfrak{k}}\circ \Ad(\tau(ah))V\in\mathfrak{q}.\end{equation}
            Observe that     $\Ad(k_a(h))\circ \ad(X)\circ E_{\mathfrak{k}}\circ \Ad(\tau(ah))V\in \mathfrak{p}$.   In view of    \cite[Lemma 5.7]{ban1986} we   see that (\ref{e: belonging to fq}) implies that
            \begin{equation}\label{e: belonging to aq}\Ad(k_a(h))\circ \ad(X)\circ E_{\mathfrak{k}}\circ \Ad(\tau(ah))V\in\fa_\iq.\end{equation}
            Since $h\in H_X$ and $G_X=K_XAN_{P,X}$,   see (\ref{e: centralizer in NP}),  it follows that $k_a(h)$ centralizes $X$. Thus, $\Ad(k_a(h))$ and $\ad(X)$ commute. Now $\Ad(k_a(h))\circ E_{\mathfrak{k}}\circ \Ad(\tau(ah))V$ is an element in $\mathfrak{k}$, which decomposes as
            \[\mathfrak{k}=\mathfrak{k}_X+\bigoplus_{\substack {\alpha\in\Sigma(P)\\ \alpha(X)\neq 0}}(I+\theta)\fg_{\alpha}.\]

  Furthermore, by (\ref{e: belonging to aq}), we know that $\ad(X)$ maps this element to an element of $\fa_\iq$. This implies that
            \begin{equation*}\Ad(k_a(h))\circ E_{\mathfrak{k}}\circ \Ad(\tau(ah))V\in\mathfrak{k}_X.\end{equation*}
           Since
 $k_a(h)\in K_X$, we obtain that
            \begin{equation}
            \label{e: belonging to kX}E_{\mathfrak{k}}\circ \Ad(\tau(ah))V\in\mathfrak{k}_X.
            \end{equation}
            By the use of   \cite[Lemma 5.2]{ban1986}, we may rewrite
            \[E_{\mathfrak{k}}\circ \Ad(\tau(ah))={\rm d} l_{k_a(h)}(e)^{-1}\circ\d{k}_{a}(h)\circ\d l_h(e)={\rm d}(l_{k_a(h)^{-1}}\circ k_a\circ l_h)(e).\]
            Hence, (\ref{e: belonging to kX})   implies
            \begin{equation}\label{e: differential notation}{\rm d} (l_{k_a(h)^{-1}}\circ {k}_{a}\circ l_h)(e)(V)\in\mathfrak{k}_X.\end{equation}
            Observe that $\d\pi(e):\fk_X\to\fk_X/\fm$ is given by the canonical projection and that the maps $\pi$ and $l_{k_a(h)^{-1}}$ commute. Hence, equation (\ref{e: differential notation}) is equivalent to
            \begin{equation}\label{e: differential notation modulo m}{\rm d} (l_{{k}_a(h)^{-1}}\circ \bar{k}_{a}\circ l_h)(e)(V)\in\mathfrak{k}_X/\fm\end{equation}
and (b) follows.

            Next, we prove that (b) $\Longrightarrow$ (c). Assume (b) and denote by $\varphi$ the diffeomorphism $\varphi:K/M\to G/P$ arising from the Iwasawa decomposition $G=KAN_P$.
            The inclusion $H\hookrightarrow G$ induces the map $\psi:H\to G/P$. It is easy to check that the diagram given below commutes.
            \begin{equation}\label{cd: non-deg-lemma-simple}\minCDarrowwidth55pt\begin{CD}
                H @>\psi>> G/P\\
                @V\bar{k}_aVV @VVl_{a}V\\
                K/M @>\varphi>> G/P
            \end{CD}\end{equation}
            The map $\psi$ commutes with the left multiplication by an element $h\in H$, viewed either as the map $l_h:H\to H$ or as the map $l_h:G/P\to G/P$. On the other hand, the diffeomorphism $\varphi$ introduced above, commutes with the left multiplication $l_k:K/M\to K/M$, where $k \in K$. Hence, the commutative diagram (\ref{cd: non-deg-lemma-simple}) gives rise to the following commutative diagram. We use the notation $k:=k_a(h)$.
            \begin{equation}\label{cd: group-level}\minCDarrowwidth55pt\begin{CD}
                H @>\psi>> G/P\\
                @Vl_{k}^{-1}\circ\bar{k}_a\circ l_hVV @VVl_{k^{-1}ah}V\\
                K/M @>\varphi>> G/P
            \end{CD}\end{equation}
            Note that under each of the four maps in diagram (\ref{cd: group-level}),
the origin of the domain is mapped to the origin of the codomain.     Taking derivatives at the origins we obtain the commutative diagram given below.
            \begin{equation}\label{cd: algebra-level}\minCDarrowwidth55pt\begin{CD}
                \fh @>\psi_{\star}>> \fg/\underline{\fp}\\
                @VTVV @VV\d(l_{k^{-1}ah})(eP)V\\
                \fk/\fm @>\varphi_{\star}>> \fg/\underline{\fp}
            \end{CD}\end{equation}
            Here
            \[\underline{\fp}=\fm\oplus\fa\oplus\fn_P\]
denotes the Lie algebra of $P$ and $T$ denotes the map $\d(l_{k}^{-1}\circ\bar{k}_a\circ l_h)(e):\fh\to\fk/\fm$. Furthermore,      $\gf_* = d\gf(eM)$ and $\psi_* = d\psi(e).$

            Observe that $k^{-1}ah=\tau:=\tau(ah)$. Since $h$ belongs to $H_X$, it follows that $\tau$ and $\tau^{-1}$ belong to $AN_{P,X}\subseteq P$. This in turn implies that      $\Ad(\tau^{-1})$ is a bijection from $\fg_X$ to $\fg_X$      which normalizes $\underline{\fp}$. Let $\overline{\Ad(\tau)}:\fg/\underline{\fp}\to\fg/\underline{\fp}$ be the map induced by $\Ad(\tau):\fg\to\fg$. Then
            \[\d(l_{k^{-1}ah})(eP)=\overline{\Ad(\tau)}.\]
            We use the commutativity of diagram (\ref{cd: algebra-level})
 to compute the pre-image of $\fk_X/\fm$ under the map $T:$
            \begin{align*}
                T^{-1}(\fk_X/\fm)&=\psi_{\star}^{-1}\circ\overline{\Ad(\tau^{-1})}\circ\varphi_{\star}(\fk_X/\fm)\\
                &= \psi_{\star}^{-1}(\overline{\Ad(\tau^{-1})}(\fk_X+\underline{\fp}))\\
                &= \psi_{\star}^{-1}(\overline{\Ad(\tau^{-1})}(\fg_X+\underline{\fp}))\\
                &= \psi_{\star}^{-1}((\Ad(\tau^{-1})\fg_X)+\underline{\fp})\\
                &= \psi_{\star}^{-1}(\fg_X+\underline{\fp})\\
                &=\{U\in\fh\setmid U+\underline{\fp}\in\fg_X+\underline{\fp}\}\\
                &=\fh_X+(\fh\cap\underline{\fp}).
            \end{align*}

            Since $\fh\cap\underline{\fp}=(\fm\oplus\fa)\cap\fh\oplus(\fn_P\cap\fh)$, see Subsection \ref{ss: langlands-decomp-subsec}, and $(\fm\oplus\fa)\cap\fh\subseteq \fh_X$, we obtain that $\fh_X+(\fh\cap\underline{\fp})=\fh_X+(\fh\cap\fn_P)$. Thus if (b) holds, then $T(V) \in \fk_X/\fm$ and we infer that $V \in \fh_X + (\fh \cap {\underline \fp})$ hence (c).

Finally,      the implication (c) $\Longrightarrow$ (a) is easy.
        \end{proof}
        \emph{Proof of Lemma \ref{l: Dana s nemesis lemma}.\ } Recall that $H$ is essentially connected. By \cite[Prop.~2.3]{ban1986}, the centralizer  $H_X$      is essentially connected as well (relative to $G_X$).

        Assume first that $\bar{h}=h\in H_X^{\circ}$. Then, by Lemma \ref{l: kernel-of-L} above, we have that
        \[\ker L_{a,X,h}=\fh_X+(\fn_P\cap\fh).\]
        Since     $dl_h(e)$ is a linear isomorphism     $\fg \to T_h G$, mapping $T_e[H_X(N_P\cap H)]$ onto $T_h[H_X(N_P\cap H)]$, we obtain that
        \[\ker H(F_{a,X})(h) =d l_h(e)(\ker L_{a,X,h})=T_h[H_X(N_P\cap H)],\]
        which establishes the assertion for $\bar{h}= h \in H_X^{\circ}$.

        Let now $\bar{h}=hn$, with $n\in N_P\cap H$. Then the right-multiplication
$r_n: H\to H$ is a diffeomorphism and $F_{a,X}\circ r_n=F_{a,X}$,
so that
        \[
\ker H(F_{a,X})(hn)=dr_n(h)[\ker H(F_{a,X})(h)]=d r_n(h) T_h[H_X(N_P\cap H)].
\]
As the latter space equals $T_{hn}[H_X(N_P\cap H)]$ this proves the assertion for $\bar{h}\in H_X^{\circ}(N_P\cap H)$.

        Finally, we discuss the general case $\bar{h}\in wH_X(N_P\cap H)$.
        Since $H$, respectively $H_X$, is essentially connected, see (\ref{e: essentially connected}),
          we may write  $\bar{h}=x_whn$, where $h\in H_X^{\circ}$, $n\in N_P\cap H$ and $x_w$ is a representative of $w$ in $N_{K\cap H}(\faq)$ chosen accordingly. Since $x_w$ normalizes $\Aq,$
        \[
F_{a,X}\circ l_{x_w}=F_{w^{-1} a,X}.
\]
Furthermore, from $a \in \Aqreg$ it follows that
$w^{-1} a \in\Aqreg$. Since $l_{x_w}$ is a diffeomorphism
from $H$ to itself, it follows that  $ dl_{x_w}(hn) $ is a linear
isomorphism from  $T_{hn}H$  onto $T_{\bar{h}}H$
and that
    \begin{eqnarray}
\ker H(F_{a,X})(\bar h) &=&
\ker H(F_{a,X})(x_whn)\nonumber\\
& =& d l_{x_w}(hn) [\ker H(F_{x_w^{-1}a,X})(hn)]\nonumber\\
&=& d l_{x_w}(hn)T_{hn}[H_X(N_P\cap H)]\nonumber\\
&=& T_{\bar h} [x_w H_X (N_P \cap H)] \nonumber\\
&=& T_{\bar{h}}\cC_{a,X,w}.
\label{e: calculation kernel Hessian}
\end{eqnarray}
\qed
\medbreak
        We will now determine the set of critical points where the Hessian is transversally positive definite. For the description of our next result we define the following subsets of $\gS(P).$ If $\ga \in \gS(\fg, \fa) \cap \faqd,$ then the associated root space $\fg_\ga$ is $\gs\Cartan$-invariant.   Hence, for such a root $\ga,$
        $$\fg_\ga = \fg_{\ga,+} \oplus \fg_{\ga, -}\;,$$
        where
        $$
        \fg_{\ga, \pm} = \{U \in \fg_\ga \setmid \gs \Cartan U = \pm U\}.$$
        Accordingly, we      define
        \[\Sigma(\fg,\faq)_\pm:=\{\alpha\in\Sigma(\fg,\faq)\setmid \fg_{\alpha,\pm}\neq 0\}.\]
        In order to formulate the first main result of this section, we need to specify particular subsets of $\gS(P).$
        \begin{defi}{\ }\label{d: gs P plus minus}
            \begin{enumerate}
                \itema $\gS(P)_+ :=\{\ga\in \gS(P)\setmid \ga \in \faqd \;\implies  \; \fg_{\ga,+} \neq 0\}.$
                \itemb $\gS(P)_-:=\{\ga \in \gS(P,\gs\Cartan) \setmid \ga \in \faqd \; {\implies} \; \fg_{\ga,-}  \neq 0\}.$
            \end{enumerate}
        \end{defi}
        \noindent
        Note   that (b) in this definition is consistent with (\ref{e: defi of Sigma P minus}).

        \begin{prop}\label{p: definite Hessian}
            Let $w \in \WKHaq.$ Then the Hessian   $H(F_{a,X})(x_w)$ is positive definite transversally to $\cC_{a,X,w}$ if and only if the following two conditions are fulfilled
            \begin{enumerate}
                \itema $\forall \ga \in \gS(P)_+:\;\;\ga(X)\ga(w^{-1}(\log a)) \leq 0$;
                \itemb $\forall \ga \in \gS(P)_-: \;\;\ga(X) \geq 0$.
            \end{enumerate}
        \end{prop}

        \begin{remark}
        For the geometric meaning of these conditions we refer to
        Lemma \ref{l: local-min-global-min},      towards the end of this section.         \end{remark}

        \proof
            We will prove the proposition in a number of steps. As a first step, let $l_w:= l_{x_w}$ denote left multiplication by $x_w$ on $H.$ Then the tangent space of $\cC_{a,X,w}$ at $x_w$ is the image of $\fh_X + (\fh \cap \fn_P)$ under the tangent map $dl_w(e): \fh \to T_{x_w}H.$ We will denote by $\rmH_w$ the pull-back of the Hessian   $H(F_{a,X})(x_w)$ under $dl_w(e).$ Then, in view of
            (\ref{e: calculation kernel Hessian}),
            \begin{equation}\label{e: kernel Hw}
            \ker H_w = \fh_X + (\fn_P \cap \fh)
            \end{equation}
            and the following conditions are equivalent:
            \begin{enumerate}
                \itema the Hessian   $H(F_{a,w})(x_w)$ is positive definite transversally to $\cC_{a,X,w};$
                \itemb the bilinear form $\rmH_w$ is positive definite transversally to $\fh_X + (\fh \cap \fn_P).$
            \end{enumerate}

            Accordingly, we will concentrate on deriving necessary and sufficient conditions for (b) to be valid.
            \begin{lemma}\label{l: Hw and Lw}
                The bilinear form $\rmH_w$ on $\fh$ is given by
                $$\rmH_w(U,V) = \inp{U}{L_w V},\qquad (U,V \in \fh),$$
                where $L_w: \fh \to \fh$ is the linear map given by
                $$L_w = - \pi_\fh \after \ad(X) \after \Ad(a^w) \after E_{\fk}\after \Ad(a^w).$$
            \end{lemma}

            \begin{proof}
                Let $U, V \in \fh.$ Then in view of Lemma \ref{l: formula for Hessian}
                we have
                $$
                \rmH_w(U,V) = R_U R_V F_{a,X}(x_w) = B(U, L_{a,X, h} V) =- \inp{U}{\Cartan L_{a,X,h} V}$$
                with $h = x_w$ and $L_{a,X, h}$ defined as in Lemma \ref{l: formula for Hessian}. Now $a h = a x_w = x_w a^w$ and we see that $\tau = \tau(ah) = a^w$ and $k_a(h)=x_w$. Hence,
                \begin{eqnarray*}
                    -\Cartan \after L_{a,X,h}(V)&=&\Cartan \circ\Ad(x_w^{-1})\circ\pi_{\fh}\circ \Ad(a^{-1})\circ \Ad(x_w)\circ \ad(X)\circ E_{\fk}\circ \Ad(a^w)V\\
                    &=&\Cartan \circ\pi_{\fh}\circ\Ad(x_w^{-1})\circ \Ad(a^{-1})\circ \Ad(x_w)\circ \ad(X)\circ E_{\fk}\circ \Ad(a^w)V\\
                    &=&\Cartan \circ\pi_{\fh}\circ\Ad(a^w)^{-1}\circ \ad(X)\circ E_{\fk}\circ \Ad(a^w)V\\
                    &=& -\pi_{\fh}\circ \Ad(a^w)\circ \ad(X)\circ E_{\fk}\circ \Ad(a^w)V.
                \end{eqnarray*}
                The result now follows since $\Ad(a^w)$ and $\ad(X)$ commute.
            \end{proof}

            In the sequel it will be useful to consider the finite subgroup
            $$F = \{1, \gs, \Cartan, \gs \Cartan\} \subseteq \Aut(\fg).$$
            The natural left action of $F$ on $\fg$ leaves $\fa$ invariant, and induces natural left actions on $\fad$ and on $\gS(\fg, \fa).$ Accordingly, if  $\tau \in F$ and $\ga \in \gS(\fg, \fa),$ then
            $$\tau (\fg_\ga) = \fg_{\tau \ga}$$
            If $\cO$ is an orbit for the $F$-action on $\gS(\fg, \fa),$ we write, in accordance with (\ref{e: defi fg cO}),
            $$\fg_\cO = \bigoplus_{\ga \in \cO} \fg_\ga.$$
            Then obviously,
            \begin{equation}\label{e: deco fg in orbits}
                \fg = \fg_0 \oplus \bigoplus_{\cO \in \gS(\fg,\fa)/F} \fg_\cO,
            \end{equation}
            with mutually orthogonal summands. Each of the summands is $F$-invariant, hence $\gs$-invariant. In particular, if we write $\fh_0 = \fh\cap \fg_0$ and $\fh_\cO = \fh \cap \fg_\cO,$ then
            \begin{equation}\label{e: deco fh in F orbits}
                \fh = \fh_0 \oplus \bigoplus_{\cO \in \gS(\fg, \fa)/F} \fh_\cO,
            \end{equation}
            with $F$-stable orthogonal summands.

            \begin{lemma}{\ }
                \begin{enumerate}
                    \itema The decomposition {\rm (\ref{e: deco fh in F orbits})} is orthogonal for $\inp{\dotvar}{\dotvar}.$
                    \itemb The decomposition {\rm (\ref{e: deco fh in F orbits})}  is preserved by $L_w.$
                    \itemc The decomposition {\rm (\ref{e: deco fh in F orbits})}
                              is orthogonal for $\rmH_w.$
                \end{enumerate}
            \end{lemma}
            \begin{proof}
                The validity of (a) follows immediately from the fact that relative to the given inner product, the root spaces are mutually orthogonal, as well as orthogonal to $\fg_0.$

                For (b) we note that the decomposition (\ref{e: deco fg in orbits}) is preserved by $\Ad(A),$ $\ad(\fa),$ $E_\fk$ and $\pi_\fh.$ Finally, in view of Lemma \ref{l: Hw and Lw}, the validity of (c) follows from (a) and (b).
            \end{proof}

            It follows from the above lemma that the kernel of $\rmH_w$ decomposes in accordance with
            (\ref{e: deco fh in F orbits}). Let      $\fv_{P,X} := (\ker \rmH_w)^\perp \cap \fh.$
            Then in view of
            (\ref{e: kernel Hw}) we have
            \begin{equation}
            \label{e: deco fv P X}
            \fv_{P,X} = \fh_X^\perp \cap (\fh\cap \fn_P)^\perp \cap \fh =
            \bigoplus_{\cO\in\gS(\fg, \fa)/F} \fv_\cO,
            \end{equation}
            \vspace{-5mm}

            \noindent  with 
             \begin{equation}
            \label{e: defi fv cO}
            \fv_\cO:= \fh_X^\perp \cap (\fh\cap \fn_P)^\perp \cap \fh_\cO
            \end{equation}
            From these definitions it follows that $\rmH_w$ is non-degenerate on each of the spaces $\fv_\cO.$ Moreover, $\rmH_w$ is positive definite if and only if the restriction of $\rmH_w$ to $\fv_\cO$ is positive definite for every $\cO \in \gS(\fg, \fa)/F.$ This in turn is equivalent to the condition that the symmetric map $L_w: \fh \to \fh$ has a positive definite restriction to each of the spaces $\fv_\cO$ (if $\fv_\cO$ is zero, we      agree that the latter is automatic). We will now systematically discuss the types of orbits $\cO$ for which $\fv_\cO$ is non-trivial.

            First of all, we note that $\ga \in \cO \implies -\ga = \Cartan \ga \in \cO.$ Therefore, we see that $\cO \cap \gS(P) \neq \emptyset$ for all $\cO \in \gS(\fg, \fa)/F.$ Let $\sim $ denote the equivalence relation on $\gS(P)$ defined by
            $$\ga \sim \gb \iff F\ga = F \gb,$$
            then the map $\ga \mapsto F\ga$ induces a bijection from $\gS(P)/\!\sim\,$ onto $\gS(\fg, \fa)/F.$ The following lemma summarizes all possibilities for the spaces $\fv_{\cO},$ as $\cO \in \gS(\fg, \fa)/F.$

            \begin{lemma}\label{l: orbit cases}
                Let $\ga \in \gS(P),$ and put $\cO = F\ga.$
                \begin{enumerate}
                    \itema If $\ga(X) = 0$ then $\fv_\cO = 0.$
                    \itemb If $\ga(X) \neq 0$      then we are in one of the following two cases {\rm (b.1)}
                     and {\rm (b.2)}.
                    \begin{enumerate}
                        \item[{\rm (b.1)}] $\ga \in \gS(P,\gs);$ in this case $\fv_\cO = \{ V + \gs(V)\setmid  V \in \fg_{-\ga}\}.$
                        \item[{\rm (b.2)}] $\ga \in \gS(P, \gs \Cartan);$ in this case $\fv_\cO = \fh_\cO.$
                    \end{enumerate}
                \end{enumerate}
            \end{lemma}

            \begin{proof}
                (a) If $\ga(X) = 0$ then $\fh_\cO \subseteq \fg_X,$ so that $\fv_\cO = \{0\}.$

                (b) Assume that $\ga(X) \neq 0.$ Then it follows that $\ga \notin \fahd,$ so that $\ga \neq \gs \ga.$
                By Lemma \ref{l: gS(P) as disjoint union}      we are in one of the cases (b.1) and (b.2).

                We first discuss case (b.1). Then $\gs \ga \in \gS(P)$ so that $\gs\ga \neq - \ga$ and $\cO = F \ga$ consists of the four distinct elements $\ga, \Cartan \ga = -\ga , \gs \ga$ and $\gs\Cartan \ga = - \gs \ga.$ We see that $\fh_\cO$ consists of sums of elements of the form $U + \gs(U)$ and $V + \gs(V)$ with $U \in \fg_\ga$ and $V \in \fg_{-\ga}.$ The elements $U +\gs(U)$ belong to $\fh \cap \fn_P,$ whereas the elements $V + \gs(V)$ belong to $\fh_X^\perp \cap (\fh \cap \fn_P)^\perp.$   In view of (\ref{e: defi fv cO}) this implies the assertion of (b.1).

                Next, we discuss case (b.2). Then  $\cO \cap \gS(P) = \{\ga, -\gs \ga\}$ so that $\fn_P \cap \fh= 0.$ Since obviously $\fh_\cO \perp \fh_X,$ we infer the assertion of (b.2).
            \end{proof}

            We will now proceed by explicitly calculating the restrictions $L_w|_{\fv_\cO}$  for all these cases. The following lemma will be instrumental in our calculations.

            \begin{lemma}\label{l: computation Tw}
                Let $T_w: \fg \to \fg$ be defined by
                $$T_w = \ad(X) \after \Ad(a^w) \after  E_\fk \after \Ad(a^w).$$
                Let $\gb \in \gS(\fg,\fa)$ and $U_\gb \in \fg_\gb.$
                \begin{enumerate}
                    \itema If $\gb \in \gS(P)$ then $T_w(U_\gb) = 0.$
                    \itemb If $\gb \in  - \gS(P)$ then $T_w (U_\gb) = \gb(X) \, (a^{2w\gb}U_\gb - \Cartan U_\gb).$
                \end{enumerate}
            \end{lemma}

            \begin{proof}
                Assume $\gb \in \gS(P).$ Then $\fg_\gb \subseteq \fn_P \subseteq \ker E_\fk.$ Since $\Ad(a^w)$ preserves $\fg_\gb,$ (a) follows.

                For (b), assume that $\gb \in - \gS(P).$ Then $U_\gb $ equals $U_\gb + \Cartan U_\gb $ modulo $\fn_P,$ so that $E_\fk(U_\gb) = U_\gb + \Cartan U_\gb.$ Hence,
                \begin{eqnarray*}
                    T_w (U_\gb) & = & \ad (X) \after \Ad(a^w) [a^{w\gb} (U_\gb + \Cartan U_\gb)]\\
                    &=& \ad(X) (a^{2w\gb} U_\gb + \Cartan U_\gb)\\
                    &=& \gb(X) (a^{2w\gb} U_\gb - \Cartan U_\gb).
                \end{eqnarray*}
            \end{proof}

            In our calculations of $L_w|_{\fv_\cO},$ we will distinguish between the cases described in Lemma \ref{l: orbit cases}. Case (a) is trivial.

            \begin{lemma}[Case b.1]{\ }\label{l: case b.1}
                Let      $\cO = F \ga$ with $\ga \in \gS(P,\gs)$ and $\ga(X) \neq 0.$ Then
                $$L_w|_{\fv_\cO} = \frac{\ga(X)}2 (a^{-2w\ga} - a^{2w\ga}) I.$$
                In particular, this restriction is positive definite if and only if $\ga(X) \ga(w^{-1}\log a) < 0.$
            \end{lemma}
            \begin{proof}
                Let $V \in \fg_{-\ga}$ and put $Z:= V + \gs(V).$ Since $-\ga, -\gs \ga \in - \gS(P),$ it follows from Lemma \ref{l: computation Tw} that \begin{eqnarray*}
                    T_w (Z)&=&-\ga(X)(a^{-2w\ga}V - \Cartan V)-\ga(\gs X) (a^{2w\ga}\gs V - \Cartan \gs V)\\
                    &=&\ga(X)[- a^{-2w\ga}V + a^{2w\ga} \gs V + \Cartan V - \Cartan \gs V]
                \end{eqnarray*}
                so that
                $$L_w(Z)=  - \pi_\fh \after T_w (Z) =\frac{\ga(X)}2 ( a^{-2w\ga} - a^{2w\ga})Z.$$
                It follows that $L_w$ restricts to multiplication by a scalar on $\fv_\cO.$ The sign of this scalar equals the sign of $- \ga(X) \ga(w^{-1}\log a).$ The result follows.
            \end{proof}
            We now turn to the calculation of $L_w|_{\fv_\cO}$ in case (b.2), where $\cO = F\ga,$ with $\ga \in \gS(P,\gs \Cartan)$      and $\ga(X)\neq 0.$
            There are two possibilities between which we will distinguish:
            \begin{enumerate}[xxxxxxx]
                \item[{\rm (b.2.1)}] $\ga \in \gS(P,\gs\Cartan) \setminus \faqd,$
                \item[{\rm (b.2.2)}] $\ga \in \gS(P)\cap \faqd.$
            \end{enumerate}
            In each of these cases, $\fv_\cO  = \fh_\cO$      by Lemma \ref{l: case b.1}.
            We will use the notation
            $$\fv(U) = \fh(U) = \fh \cap {\rm span}\, (F\cdot U),$$
            for $U \in \fg_\ga.$
            In case (b.2.1), the orbit $\cO = F\ga$ consists of the four distinct roots $\ga, \gs \ga, \Cartan \ga$ and $\gs \Cartan \ga$, and
            $$\fv(U) = \R (U + \gs(U)) \oplus \R (\gs\Cartan (U) + \Cartan ( U)).$$
            In case (b.2.2), $\cO = F\ga = \{\ga, - \ga\},$ and we see that
            $$\fv(U) = \R(U + \gs(U)).$$
            In all of these cases, we see that if $U_1, \ldots, U_m$ is an orthonormal basis of $\fg_\ga,$ then
            \begin{equation}\label{e: deco fv cO}
                \fv_\cO = \bigoplus_{j=1}^m \fv(U_j),
            \end{equation}
            with mutually orthogonal summands.

            \begin{lemma}[Case (b.2.1)]\label{l: case b.2.1}
                Let $\cO = F\ga,$  with $\ga \in \gS(P, \gs\Cartan)\setminus \faqd$ and $\ga(X)\neq 0.$
                Then $L_w|_{\fv_\cO}$ is positive definite if and only if $\ga(X) > 0$ and $\ga(X)\ga(w^{-1}\log a) < 0.$
            \end{lemma}
            \begin{proof}
                Fix an element $U \in \fg_\ga$ and put $Z_1 = U + \gs(U)$ and $Z_2 = \Cartan Z_1= \Cartan U + \gs \Cartan U.$ Then $T_w(U) = 0$ by
                Lemma \ref{l: computation Tw}, hence
                \begin{eqnarray*}
                    T_w(Z_1) &= &T_w(\gs U) \\
                    &=& \gs \ga (X) (a^{2w\gs\ga} \gs(U) - \Cartan \gs(U))\\
                    &=& \ga(X) (\Cartan \gs(U) - a^{-2w \ga} \gs(U)),
                \end{eqnarray*}
                from which we see that
                $$L_w(Z_1) = - \pi_\fh T_w(Z_1) = \frac{\ga(X)}2 ( a^{-2w\ga} Z_1- Z_2).$$
                Likewise,
                $$L_w(Z_2) = \frac{\ga(X)}{2} ( a^{-2w\ga} Z_2- Z_1).$$
                It follows that $L_w$ preserves the subspace $\fv(U)$ of $\fv_\cO$ spanned by the orthogonal vectors $Z_1, Z_2$ and that the restriction $L_w|_{\fv(U)}$ has the following matrix with respect to this basis:
                $${\rm mat}(L_w|_{\fv(U)}) =\frac{\ga(X)}{2}\left(\begin{array}
                                                                    {cc}a^{-2w \ga}  & - 1\\
                                                                    - 1 & a^{- 2w\ga}
                                                                    \end{array}\right)$$
                This matrix is positive definite if and only if both its trace and determinant are positive. This is equivalent to
                $$\ga(X) > 0 \quad {\rm and} \quad \ga(X) ( a^{- 4w\ga} - 1) > 0.$$
                It follows that $L_w$ is positive definite on the subspace $\fv(U)$ if and only if the inequalities
                $ \ga(X) > 0$ and $\ga(X)\ga(w^{-1} \log a) < 0$ are valid.

                Let $U_1, \ldots, U_m$ be an orthonormal basis for $\fg_\ga.$ Then by (\ref{e: deco fv cO}) we see that map $L_w$ is positive definite if and only if all restrictions $L_w|_{\fv(U_j)}$ are positive definite. This is true if and only if $\ga(X) > 0$ and $\ga(X)\ga(w^{-1} \log a ) < 0.$
            \end{proof}

            \begin{lemma}[Case (b.2.2)]\label{l: case b.2.2}
                Let  $\cO = F\ga$ with $\ga \in \gS(P)\cap \faqd$ and $\ga(X) \neq 0.$ Then $L_w|_{\fv_\cO}$ is positive definite if and only if the following two conditions are fulfilled.
                \begin{enumerate}
                    \itema $\ga \in \gS(P)_+\cap \faqd\;  \implies \; \ga(X)\ga(w^{-1}\log a)) < 0.$
                    \itemb $\ga \in \gS(P)_-\cap \faqd\; \implies \; \ga(X) > 0.$
                \end{enumerate}
            \end{lemma}
            \begin{proof}
                We recall that $\fv_\cO = \fh_\cO$ in this case and write $\fv_{\cO, +} = \fv_\cO \cap \fk$ and $\fv_{\cO, -} = \fv_\cO \cap \fp.$ Then
                $$\fv_\cO = \fv_{\cO,+} \oplus \fv_{\cO, -},$$
                with orthogonal summands. We will show that $L_w$ preserves this decomposition, and determine when both restrictions      $L_w|_{\fv_{\cO,\pm}}$ are positive definite.

                Let $U_{\pm} \in \fg_{\ga, \pm}$ and put $Z_\pm = U_\pm + \gs(U_\pm).$ Then
                 $Z_\pm \in \fv_{\cO, \pm},$ and every element of $\fv_{\cO, \pm}$ can be expressed in this way.

                By a straightforward computation, involving Lemma \ref{l: computation Tw}, we find
                $$L_w(Z_\pm) = \frac12 \ga(X) (a^{-2w\ga} \mp 1)Z_\pm.$$
                This shows that $L_w$ acts by a real scalar $C_{\pm}$ on $\fv_{\cO, \pm}.$ The restriction of $L_w$ to $\fv_{\cO_\pm}$ is positive definite if and only if the restrictions of $L_w$ to both subspaces $\fv_{\cO, \pm}$ are positive definite. The latter condition is equivalent to
                $$\fv_{\cO, +} \neq 0 \implies C_+ > 0 \quad{\rm and}\quad\fv_{\cO, -} \neq 0 \implies C_- > 0.$$
                The space $\fv_{\cO,\pm}$ is non-trivial if and only if $\fg_{\ga, \pm} \neq 0,$ which in turn is equivalent to $\ga \in \gS(P)_\pm\cap \faqd.$ On the other hand, the sign of $C_+$ equals that of $- \ga(X) \ga(w^{-1}\log a)$ whereas the sign of $C_-$ equals that of $\ga(X).$ From this the desired result follows.
            \end{proof}

            {\em Completion of the proof of Proposition \ref{p: definite Hessian}.\ } First assume that $\rmH_w$ is positive definite. Then $L_w$ restricts to a positive definite symmetric map on each of the spaces $\fv_\cO$ for $\cO = F\ga,$ $\ga \in \gS(P).$  First assume that $\ga \in \gS(P)_+.$ If $\ga(X) = 0,$ then
            \begin{equation}\label{e: inequality with X and log a}
                \ga(X) \ga(w^{-1}\log a) \leq 0
            \end{equation}
            holds. If $\ga(X)\neq 0,$ we are in one of the cases (b.1) or (b.2) of Lemma \ref{l: orbit cases}. In the latter case, we are either in the subcase (b.2.1) or in (b.2.2) with $\ga \in \gS(P)_+\cap \faqd.$ In all of these cases, inequality (\ref{e: inequality with X and log a}) is valid. We conclude that assertion (a) of the proposition is valid.

             For the validity of assertion (b), assume that $\ga \in \gS(P)_-.$ If $\ga(X) = 0,$ then
            \begin{equation}\label{e: inequality with X}
                \ga(X)  \geq 0.
            \end{equation}
            If $\ga(X)\neq 0,$ then we must be in case (b.2) of Lemma \ref{l: orbit cases}, since $\gS(P)_- \cap \gS(P, \gs) = \emptyset.$ We are either in subcase (b.2.1) or in subcase (b.2.2) with $\ga \in \gS(P)_+\cap \faqd.$ In both subcases, (\ref{e: inequality with X}) holds. This establishes condition (b) of the proposition, and the implication in one direction.

            For the converse implication, assume that conditions (a) and (b) of the proposition hold. Let $\ga\in \gS(P)$ and put $\cO = F\ga.$ Then it suffices to show that $\rmH_w$ is positive definite on $\fv_\cO.$

            If $\ga(X) = 0,$ then $\fv_\cO = 0$ by Lemma \ref{l: orbit cases} and it follows that $\rmH_w$ is positive definite on $\fv_\cO.$ Thus, assume that $\ga(X) \neq 0.$ Then by regularity of $\log a,$ the expression $\ga(X) \ga(w^{-1} \log a)$ is different from zero. Hence if any of the inequalities (\ref{e: inequality with X and log a}) or (\ref{e: inequality with X}) holds, it holds as a strict inequality.

            In case (b.1), $\ga \in \gS(P,\gs)\subseteq \gS(P)_+$ so that
            (\ref{e: inequality with X and log a}) is valid.  Therefore, $\rmH_w|_{\fv_\cO}$ is positive definite by Lemma \ref{l: case b.1}.  In case (b.2.1), $\ga \in \gS(P,\gs\theta)\setminus\faq^*\subseteq  \gS(P)_+ \cap \gS(P)_-$ so that (\ref{e: inequality with X and log a}) and (\ref{e: inequality with X})
            are both valid. Hence, $\rmH_w|_{\fv_\cO}$ is positive definite by Lemma \ref{l: case b.2.1}.

            Finally, assume we are in case (b.2.2). Then $\ga \in \faqd,$ hence it follows
            from hypotheses (a) and (b) of the proposition that conditions (a) and (b) of
            Lemma \ref{l: case b.2.2} are fulfilled. Hence,  $\rmH_w|_{\fv_\cO}$ is positive definite.
        \qed

        \begin{corollary}\label{c: const-sign}
            Let  $w \in \WKHaq$. Then the function $F_{a,X}$ as well as the signature and rank of its Hessian are constant on the immersed
            submanifold $w H_X (N_P\cap H).$
                    \end{corollary}

        \begin{proof}
        As the group $H$ is essentially connected, $H_X=Z_{K\cap H}(\faq) H_X^\circ$.
        Let $x_w$ be a representative of $w$ in $N_{K\cap H}.$  Since $Z_{K\cap H}(\faq)$ is normal in
        $N_{K\cap H}(\faq),$ it follows that
        $$
        w H_X (N_P \cap H) = x_w Z_{K\cap H}(\faq)H_X^\circ  (N_P\cap H) = Z_{K\cap H} (\faq) x_w H_X^\circ (N_P \cap H).
        $$
The function $F_{a,X}: H \to \R$ is left $Z_{K\cap H}(\faq)$- and right $(N_P\cap H)$-invariant.
Hence, it suffices to prove the assertions for the set $x_w H_X^\circ$ of critical points.
This set is connected, so that $F_{a,X}$ is constant on it.
From Lemma $\ref{l: Dana s nemesis lemma}$ it follows that rank and signature of its Hessian remain constant along this set as well.
        \end{proof}
         As   in (\ref{e: introduction set Omega}) we define
        $$
        \Omega :=  \conv(\WKHaq \cdot \log a) + \GammaP.
        $$

        \begin{lemma}\label{l: local-min-global-min}
            Let $a\in \Aqreg$ and $X\in\faq.$ Assume that the function $F_{a,X}$ has a local minimum at the critical point $h\in \cC_{a,X}$. Then for every $U\in\Omega$
            $$\inp{X}{U} \geq  \inp{X}{\HPq(ah)}.$$
            In particular,{\ } $\Omega$ lies on one side of the hyperplane $\HPq(ah)+X^\perp$.
        \end{lemma}
        \begin{proof}
            The critical point $h$ belongs to a connected immersed submanifold of the form $x_w H_X^\circ (H\cap N_P).$ All points of this submanifold are critical for $F_{a,X},$ so that $F_{a,X}$ is constant along it. We see that
            $$F_{a,X} (h) = F_{a,X}(x_w) = \inp{X}{\HPq(x_w^{-1} a x_w)} = \inp{X}{w^{-1}\log a}.$$
            The Hessian of $F_{a,X}$ at the critical point $h$ must be positive semidefinite. It now follows from Proposition \ref{p: definite Hessian} that
            \begin{enumerate}
                \itema $\forall \ga \in \gS(P)_+:\;\;\ga(X)\ga(w^{-1}(\log a)) \leq 0$;
                \itemb $\forall \ga \in \gS(P)_-: \;\;\ga(X) \geq 0$.
            \end{enumerate}
            By (a) and      Lemma \ref{l: inequality and conv WKH} below (applied to $-X$), it follows that
            $$\inp{X}{U_1} \geq \inp{X}{w^{-1}\log a} = F_{a,X}(h),$$
            for all $U_1 \in \conv(\WKHaq \cdot w^{-1}\log a).$ From (b) it follows that
            $\inp{X}{H_\ga} = \inp{H_\ga}{H_\ga} \ga(X)/2 \geq 0$ for all $\alpha\in\gS(P)_-$, so that
            $$\inp{X}{U_2} \geq 0 \qquad (\forall \;U_2 \in \GammaP).$$
            Since every element $U \in \Omega$ may be decomposed as $U = U_1 + U_2$
            with    $U_1$ and $U_2$ as above, the assertion follows.
        \end{proof}

        \begin{remark}
            It can be readily shown that the converse implication also holds. Indeed if for every $U\in\Omega$
            $$\inp{X}{U} \geq  \inp{X}{w^{-1}(\log a))},$$
            then the two conditions of Proposition \ref{p: definite Hessian} hold.
        \end{remark}

        \begin{lemma}\label{l: second char of gS P plus}
            The set $\gS(P)_+$ consists of all roots $\ga \in \gS(P)$ with $\ga \in \fahd$ or $\ga|_{\faq}\in \gS(\fg, \faq)_+.$
        \end{lemma}
        \begin{proof}
            In view of Definition \ref{d: gs P plus minus} it suffices to show that for $\ga \in \gS(\fg,\fa) \setminus (\fahd \cup \faqd)$ we have
            $\ga|_{\faq} \in \gS(\fg, \faq)_+.$ Assume $\ga \notin \fahd \cup \faqd.$ Then
            $\ga$ and  $\gs\Cartan \ga$ are distinct roots that restrict to the same root
            $ \bar \ga $ of $\gS(\fg, \faq).$ Thus, the sum $\fg_\ga + \gs\Cartan \fg_\ga$ is direct and contained in $\fg_{\bar \ga}$ and we see that $\fg_{\bar \ga, +} \neq 0.$
             \end{proof}
        \begin{lemma}
        \label{l: inequality and conv WKH}
            Let $P \in \cP(A).$ Let $X,Y \in \faq$ and assume that $\ga(X) \ga(Y) \geq 0$ for all $\ga \in \gS(P)_+.$ Then
            $$\inp{X}{U} \leq \inp{X}{Y},\qquad \mbox{\rm for all}\quad  U \in \conv(\WKHaq \cdot Y).$$
        \end{lemma}
        \begin{proof}
            In view of   Lemma \ref{l: second char of gS P plus}, the hypothesis is equivalent to
            $$\ga(X)\, \ga(Y) \geq 0$$
            for all roots $\ga \in \gS(\fg, \faq)_+.$ We may now fix a Weyl chamber $\faq^+$ for the root system $\gS(\fg, \faq)_+$ such that $X$ and $Y$ belong to the closure of $\faq^+.$ Then it is well known that $\inp{X}{wY} \leq \inp{X}{Y}$ for all $w$ in the reflection group $W(\gS(\fg, \faq)_+)$ generated by $\gS(\fg, \faq)_+.$ Since this reflection group is equal to $\WKHaq,$ by Proposition 2.2 in \cite{ban1986}, the result follows.
        \end{proof}

  \section{Reduction by a limit argument}\label{s: limit argument}
    Before turning to the proof of our main theorem, Theorem \ref{t: conv thm new}, we will first prove a lemma that reduces the validity of the theorem to  its validity under the additional assumption that the element $a$ be regular in $\Aq.$ We assume that $P \in \cP(A)$ and recall the definition of the closed convex polyhedral cone $\Gamma(P)$ given in
    Definition \ref{d: Gamma P}.

    \begin{lemma}
    \label{l: reduction lemma}
    Assume that the assertion
    \begin{equation}
    \label{e: main statement}
    \pr_\iq\after \HP ( aH ) = \conv(\WKHaq \cdot \log a) + \Gamma(P)
    \end{equation}
    is valid for all $a \in \Aqreg.$ Then assertion (\ref{e: main statement}) holds for all $a \in \Aq.$
    \end{lemma}

        \begin{proof}
        Assume the assertion is valid for all $a \in \Aqreg,$ and let $a \in \Aq$ be an arbitrary
        fixed element. Fix a sequence $(a_j)_{j\geq 1}$ in $\Aqreg$ with limit $a.$
        We will establish the equality (\ref{e: main statement}) for $a.$

        First we will show that the set on the left-hand side of the equality is contained in the set on the right-hand side. For this, assume that
           $h\in H$.  By the validity of (\ref{e: main statement})
            for $a_j$ in place of $a,$  there exist, for each $j \geq 1,$ elements  $\gl_{w,j}\in [0,1]$ with $\sum_{w \in \WKHaq } \gl_{w,j} = 1$ and elements $\gamma_j \in \Gamma(P)$ such that 
            $$
            \HPq(a_j h) = \sum_{w \in\WKHaq} \; \gl_{w,j} w(\log a_j) + \gamma_j.
            $$
             By passing to a subsequence of indices we may arrange that the sequence $(\gl_{w,j})_j$ converges with limit $\gl_w \in [0,1]$ for each $w \in \WKHaq.$ It follows that the sequence $(\gamma_j)$ must have a limit $\gamma \in \faq$ such that
            $$\HPq(ah) = \lim_{j \to \infty} \HPq(a_j h) = \sum_{w \in \WKHaq} \gl_w w(\log a) + \gamma.$$
            By taking the limit we see that $\sum_w \gl_w = 1$ and since $\Gamma(P)$ is closed, $\gamma \in \Gamma(P).$ Hence, $\HPq(ah) \in \conv(\WKHaq \cdot \log a) + \Gamma(P),$ and we      obtain the desired first inclusion.


            For the converse inclusion, assume that $Y \in \conv(\WKHaq \dotvar \log a)+\Gamma(P)$. Then there exist $\gamma \in \Gamma(P)$ and $\gl_w \in [0,1]$ with $\sum_{w \in \WKHaq} \gl_w =1$ such that
            $$
            Y = \sum_{w \in \WKHaq} \gl_w   w(\log a)  + \gamma.
            $$
           Put
            $$
            Y_j = \sum_{w \in\WKHaq} \gl_w w(\log a_j) + \gamma.$$
            Then for every $j$ there exists $h_j \in H$ such that $\HPq(a_j h_j) = Y_j.$
             The sequence $(Y_j)$ is convergent, hence contained in a compact set of $\faq.$ Likewise, the sequence $(a_j)$ is contained in a compact subset $\cA \subseteq \Aq.$ By Corollary \ref{c: properness with set C} (b) there exists a compact subset $\cK$ of $H/H\cap P$ such that $h_j (H\cap P) \in \cK$ for all $j.$ By passing to a subsequence we may arrange that $h_j(H\cap P)$ converges in $H/H\cap P$. By continuity of the induced map
             $\barHPq: H/H\cap P \to \faq ,$ see (\ref{e: induced smooth map on H mod P}), it now follows that
            $$
            Y = \lim_{j\to\infty} Y_j = \lim_{j\to\infty} \HPq(a_j h_j) = \HPq(a h)\in \HPq(aH).
            $$
        \end{proof}

    \section{Proof of the main theorem}\label{s: section-induction}
    In this section we will prove our main result. For $P \in \cP(A)$ we recall the definition of the closed convex polyhedral cone $\Gamma(P)$ given in Definition \ref{d: Gamma P}.

 \begin{thm}\label{t: conv thm new}
            Let       $P$ be a minimal parabolic subgroup of $G$ containing $A$      and let $a \in \Aq.$ Then
            \begin{equation}\label{e: convexity equality}
                \pr_\iq\after \HP ( aH )=\fH_{P,\iq}(aH) = \conv(\WKHaq \cdot \log a)+\GammaP.
            \end{equation}
    \end{thm}

 The    proof of our main theorem proceeds by induction, for whose induction step the following lemma is a key ingredient.

        If $X \in \faq,$ we denote by $G_X$ the centralizer of $X$ in $G.$
        This group belongs to the Harish-Chandra class and is $\gs$-stable. Moreover,   by
        \cite[Prop.~2.3]{ban1986}, the centralizer $H_X: = H \cap G_X$ is an essentially connected open subgroup of $(G_X)^{\gs}. $
        From
        \[
        P\cap G_X=(Z_K(\fa)AN_P)\cap (K_XAN_{P,X})=Z_K(\fa)AN_{P,X},
        \]
        see (\ref{e: centralizer in NP}) for notation,
        we see that  $P_X:= P \cap G_X$ is a minimal parabolic subgroup of $G_X.$

        We agree to write $\GammaPX$ for the cone in $\faq$ spanned by      $\prq H_\ga,$ for $\ga \in \gS(P)_-$ with $\ga(X) = 0.$  Furthermore,    for a given $a \in \Aq,$ we define $\Omega_{a, X} = \Omega_{X}$ by
        \begin{equation}
        \label{e: defi Omega X new}
          \Omega_X :=  \bigcup_{w\in \WKHaq} \Omega_{X,w}, \qquad {\rm where}\qquad\qquad\qquad
        \end{equation}
        \begin{equation}\label{e: defi Omega X w new}
        \!  \!  \!  \!  \!  \!  \!  \!  \!  \!  \!  \! \!  \!  \!  \!
         \Omega_{X,w} :=  \conv(\WKHXaq \cdot w^{-1} \log a)+\Gamma(P_X).
         \end{equation}

We write  $\Omega := \Omega_0$ and note that this set equals $ \conv(\WKH\cdot \log a ) + \Gamma(P)$ hence contains $\Omega_X$ for every $X \in \faq.$

         \begin{remark}\label{r: on Omega X}
         It is clear from the definition that the set $\Omega_{X,w},$ for $w \in \WKHaq,$
         is a closed convex polyhedral set, contained in the affine subset $w^{-1}\log a + {\rm span}\, \{H_\ga\! \setmid \!\! \ga\in \gS(\fg_X, \faq)\}$ of $\faq.$ In particular,
         $$
         \Omega_{X,w} \subseteq  w^{-1}\log a + X^\perp.
         $$
         \end{remark}

          \begin{lemma}\label{l: HPq of a cC X}
            Let $X \in S,$ $a \in \Aqreg$ and let $\cC_{a,X}\subseteq H$ be the set of critical points of the function $F_{a,X}: H \to \R;$ cf. Lemma \ref{l: set of critical points} and {\rm (\ref{e: defi S})}. If    the analogue of the assertion of   Theorem \ref{t: conv thm new}  holds for the data $\,G_X,$ $H_X,$ $K_X$ and $P_X$ in place of $\,G,H,K$ and $P$ then
            \begin{equation}\label{e: image of critical set}
                \HPq(a\cC_{a,X})= \Omega_X.
            \end{equation}

        \end{lemma}
        \begin{proof}
        Using the characterization of $\cC_{a,X}$ given in Lemma \ref{l: set of critical points}, we obtain
                        \begin{eqnarray}
                \HPq(a \cC_{a,X}) & =&\bigcup_{w\in \WKHaq }\HPq(a w H_X(N_P\cap H))\nonumber\\
                &=&\bigcup_{w\in \WKHaq} \HPq (a^w H_X),\label{e: H of CX as union}
            \end{eqnarray}
            where  $a^w= w^{-1} a w$ is regular in $\Aq,$ for each $w \in \WKHaq.$

              By the compatibility of the Iwasawa decompositions for the two groups $G$ and $G_X$ we see that the restriction of $\HPq: G \to \faq$ to $G_X$ equals the similar projection $G_X \to \faq$ associated with $P_X;$ we denote the latter by $\fH_{P_X, \iq}.$ Hence,
            $$\HPq(a^w H_X)=  \fH_{P_X, \iq}(a^w H_X).$$
            In view of the hypothesis that the convexity theorem holds for the data $G_X, H_X, P_X,$  we infer that
            $$
            \HPq(a^wH_X)=\conv(\WKHXaq \cdot\log a^w)+\Gamma(P_X) = \Omega_{X,w}.
            $$
            In view of (\ref{e: H of CX as union}) and (\ref{e: defi Omega X new})
            we now obtain (\ref{e: image of critical set}).
        \end{proof}

        \medbreak\noindent
        {\em Proof of Theorem \ref{t: conv thm new}.\ }
            The proof relies on an inductive procedure, with induction over the rank of the root system $\Sigma(\fg,\faq)$. The legitimacy of this procedure has been discussed at length in  \cite[Sect. 2]{ban1986}.

            We start the induction with $\rk\, \Sigma(\fg,\fa_\iq)=0$.
            In this case,
            \begin{equation}
            \label{e: root  vanishes on faq}
            \forall \ga \in \gS(\fg, \fa):\quad \ga|_{\faq} =0.
            \end{equation}
            This implies that
            $\faq$ is central in $\fg.$ As $G$ is of the Harish-Chandra class, $\Ad(G) \subset
            {\rm Int}(\fg_\C)$ so that $G$ centralizes $\faq.$ Hence $\Aq$ is central in $G.$
            Furthermore, (\ref{e: root vanishes on faq}) also implies that every root
            $\ga \in \gS(\fg,\fa)$ is fixed by $\gs,$ so that $\fg_\ga$ is $\gs$-invariant. This implies that
            the Iwasawa decomposition
            $G = KA N_P$ is $\gs$-stable, so that
            $
            H = (H\cap K)(H\cap A)(H\cap N_P).
            $
            We conclude that
            \begin{equation}
            \label{e: left side in rank zero case}
            \HPq(aH)= \HPq(H a) = \HPq(H\cap A) + \log a = \log a.
            \end{equation}
            On the other hand, it follows from (\ref{e: root vanishes on faq}) that
            $\gS(P)_- = \emptyset,$ so that $\Gamma(P) = \{0\}.$ Furthermore, since $G$ centralizes $\faq,$ we see that $\WKH = \{e\},$ so that
            \begin{equation}
            \label{e: right side in rank zero case}
            \conv(\WKH\cdot \log a) +\Gamma(P) = \log a.
            \end{equation}
            From (\ref{e: left side in rank zero case}) and (\ref{e: right side in rank zero case})
            we see that the equality (\ref{e: convexity equality}) holds in case $\rk\, \Sigma(\fg,\fa_\iq)=0.$
            
            Now assume that $m$ is a positive integer, that $\rk\,\Sigma(\fg,\faq) = m$ and that the assertion of the theorem has already been established for the case that $\rk\,\Sigma(\fg,\faq) < m$.

            By    Lemma \ref{l: reduction lemma} it suffices to prove the validity of (\ref{e: convexity equality}) under the assumption that $a \in \Aqreg.$ We will first do so under the additional assumption that $\Sigma(\fg, \faq)$ spans $\faqd.$ In the end, the general case will be reduced to this.

            Our assumption that $\Sigma(\fg, \faq)$ is spanning guarantees that for each non-zero $X\in \faq$ not all roots of $\Sigma(\fg,\faq)$ vanish on $X$. Therefore, the rank of $\Sigma(\fg_X,\faq)$ is strictly smaller than $m = \rk \Sigma(\fg,\faq).$ By the induction hypothesis,  the convexity theorem holds for $(G_X, H_X, K_X,  P_X).$ Hence, by Lemma \ref{l: HPq of a cC X} we have that
            \begin{equation}
            \label{e: Omega X by induction hypo}
            \HPq(a\cC_{a,X}) = \Omega_X.
            \end{equation}
            By
            Remark \ref{r: on Omega X} the complement $\faq\setminus  \Omega_X$ is open and dense in $\faq.$

            Let $S_0 \subseteq S$ be a finite subset as in Lemma \ref{l: property cCa new}. Then it follows by application of Lemma \ref{l: HPq of a cC X} that
            \begin{equation}
            \label{e: image of a cCa under HPq}
            \HPq(a \cC_a) = \cup_{X \in S_0} \Omega_X.
            \end{equation}
            In particular, the complement of this set in $\faq$  is dense.
            Moreover, it follows from  (\ref{e: image of a cCa under HPq})  and the text below
            (\ref{e: defi Omega X w new}) that
            \begin{equation}
            \label{e: image critical set in Omega}
            \HPq(a\cC_a) \subseteq  \Omega = \conv(\WKHaq \cdot \log a) + \Gamma(P).
            \end{equation}
            From Lemma \ref{l: HP of aH and aC} we see that $\HPq(aH)$ and $\HPq(a\cC_a)$ are closed subsets of $\faq$ and that $\HPq(aH)\setminus\HPq(a\cC_a)$ is an open and closed subset of the (open and dense) subset $\faq\setminus \HPq(a\cC_a),$ hence a union of connected components of the latter set. Lemma \ref{l: at-least-one-connected-component} ensures that at least one connected component of $\faq\setminus\HPq(a\cC_a)$ must belong to $\HPq(aH)\setminus\HPq(a\cC_a)$.

            From      (\ref{e: image critical set in Omega}) it follows that
            $$
            \faq\setminus\Omega  \subseteq \faq\setminus \HPq(a\cC_a).
            $$
            Now $\faq\setminus\Omega$ is connected hence must be contained in a connected component $\Lambda$ of $\faq\setminus\HPq(a\cC_a)$.

            There are two possibilities:
            \begin{enumerate}
                \itema $\Lambda\subseteq  \HPq(aH)\setminus \HPq(a\cC_a)$;
                \itemb $\Lambda\cap (  \HPq(aH)\setminus \HPq(a\cC_a)    ) = \emptyset$.
            \end{enumerate}
            From its definition, one sees that $\Omega$ is strictly contained in a half-space, which implies that $\faq\setminus\Omega$, and therefore $\Lambda$, must contain a line
            of $\faq.$ From Corollary \ref{c: half space result} we know that $\HPq(aH)$ does not contain  such a line,  so that we may exclude case (a) above. From (b) it follows that
            $$
            (\faq\setminus\Omega)\cap  \HPq(aH)\setminus \HPq(a\cC_a) =\emptyset,
            $$
              which implies that $\HPq(aH)\setminus \HPq(a\cC_a)\subseteq  \Omega$. Combining   this with (\ref{e: image critical set in Omega})
         we conclude that
            \begin{equation}\label{e: first inclusion convexity}
                \HPq(aH)  \subseteq   \Omega.
            \end{equation}
            We now turn to the proof of the converse inclusion.

            In the above we concluded that the set $\HPq(aH)\setminus \HPq(a\cC_a)$ is open and closed as a subset of $\faq\setminus \HPq(a\cC_a)$. In view of (\ref{e: first inclusion convexity}) the set is also open and closed as a subset of $\Omega\setminus \HPq(a\cC_a)$. Thus, $\HPq(aH)\setminus \HPq(a\cC_a)$ is a union of connected components of $\Omega\setminus \HPq(a\cC_a)$. We will establish the converse of (\ref{e: first inclusion convexity}) by showing that all connected components of $\Omega\setminus \HPq(a\cC_a)$ are contained in $\HPq(aH)$.

           Again by the use of Lemma \ref{l: at-least-one-connected-component} we infer that at least one connected component $\gL_1$ of $\Omega\setminus \HPq(a\cC_a)$ is contained in $\HPq(a\cC_a).$  Arguing by contradiction, assume this were not the case for all components.  Then there
            exists a second connected component $\gL_2$ of $\Omega\setminus \HPq(a\cC_a)=\Omega\setminus\cup_{X\in S_0}\Omega_X$ such that
            \begin{equation}
            \label{e: intersection gLtwo empty}
            \gL_2 \cap \HPq(aH) = \emptyset.
            \end{equation} In view of Remark \ref{r: on Omega X},
             we may apply Lemma \ref{l: convex-cones} below to the set $\Omega$ and the finite collection of subsets $\Omega_{X,w}$, where  $X\in S_0$ and $w\in \WKHaq$, and obtain a line segment with the properties of Lemma \ref{l: convex-cones}, connecting $\gL_1$ and $\gL_2$. By following    intersections along this line segment, we see that we may assume that the connected components $\gL_1$ and $\gL_2$ exist with the additional property that they are adjacent, i.e., there exists a codimension 1 subset $\Omega_{X,w} \subseteq \Omega$ together with a point $Y \in \Omega_{X,w}$ and a positive number $\epsilon > 0$ such that $B(Y; \epsilon) \setminus \Omega_{X,w}$ consists of two connected components $\gL_1'$ and $\gL_2'$ such that $\gL_j' \subseteq \gL_j$ for $j =1,2$. In particular, this implies that $\gL_1'$ and $\gL_2'$ are on different sides of the hyperplane ${\rm aff}(\Omega_{X,w})=Y+X^\perp$. We may replace $X$ by $-X$ if necessary, to arrange that $Y + tX \in \Lambda_1$ for $t \downarrow 0.$ Then
             \begin{equation}
             \label{e: estimate on component one}
             \inp{X}{\dotvar} \geq \inp{X}{Y} \quad {\rm on} \quad {\rm cl}\,(\gL_1').
             \end{equation}
            By (\ref{e: Omega X by induction hypo})
            there exists a point $h \in \cC_{a,X}$ such that $\HPq(ah) = Y.$ For a sufficiently small neighborhood $U$ of $h$ in $H$ we have $\HPq(aU) \subseteq  B(Y;\epsilon).$
            Combined with (\ref{e: intersection gLtwo empty}) this implies $\HPq(aU) \subseteq {\rm cl} (\gL_1').$
In view of (\ref{e: estimate on component one}) we now infer that
$
F_{a,X} \geq \inp{X}{Y} = F_{a,X}(h)$ on $U.$ Hence, $F_{a,X}$ has a local minimum at $h.$ By what we established in Lemma \ref{l: local-min-global-min} this implies that $\Omega$ should be on one side of the hyperplane $Y + X^\perp,$ contradicting the observation that $\gL_1'$ and $\gL_2'$ are non-empty open subsets on different sides of this hyperplane, but both contained in $\Omega.$

            In view of this contradiction we conclude that all components of $\Omega \setminus \HPq(a\cC_a)$ are contained in $\HPq(aH).$

            This finishes the proof in case $\gS(\fg, \faq)$ has rank $m$ and spans $\faqd.$ We finally consider the case with $\rk\, \gS(\fg, \faq) = m$ in general.

            Let $\fc$ be the intersection of the root hyperplanes $\ker \ga \subseteq  \faq$ for $\ga \in \gS(\fg, \faq).$ Then $\fc$ is contained in $\faq$ and central in $\fg.$ Since $G$ is of the Harish-Chandra class, $\Ad(G)$ is contained in ${\rm Int}\,(\fg_\C),$  hence centralizes $\fc.$ Therefore, the subgroup $C: = \exp(\fc)$ is central in $G.$

Let $\bp \fp$ be the orthocomplement of $\fc$ in $\fp.$ Then $\bp \fg = \fk \oplus \bp \fp$ is an ideal of $\fg$ which is complementary to $\fc.$

            By the Cartan decomposition and the fact that $\fc$ is central, it follows that the map $K \times \bp \fp \times \fc \to G,$ $(k, X, Z) \mapsto k \exp X \exp Z$ is a diffeomorphism onto. It readily follows that $\bp G = K \exp \bp \fp$ is a group of the Harish-Chandra class, with the indicated Cartan decomposition for the Cartan involution $\bp \Cartan = \Cartan|_{\bp G}.$ The restricted map $\bp \gs := \gs|_{\bp G}$ is an involution of $\bp G$ which commutes with $\bp \Cartan.$ The group $\bp H := H$ is an open subgroup of $(\bp G)^{\bp \gs},$ which is essentially connected. Furthermore, $\bp \faq:= \bp \fp \cap \faq$ is maximal abelian in $\bp \fp \cap \fq$ and $\bp \fa = \bp \fp \cap \fa$ is maximal abelian in $\bp \fp.$ The root system $\gS(\bp \fg, \bp \faq)$ consists of the restrictions of the roots from $\gS(\fg, \faq),$ hence spans the dual of $\bp \faq.$

            The group $\bpP= \bpG \cap P$ is a minimal parabolic subgroup of $\bpG$ containing $\bpA.$ We note that $\bp P  = M \bp \! A N_P.$

            We note that $\Aqreg \simeq  \bpAqreg \times C.$
            Let $a \in \Aqreg.$ Then we may write $a =\bp a \cdot c,$ with $\bp a \in \bpAqreg$ and $c \in C.$ By the convexity theorem for $\bp G$ and since $c$ is central in $G,$ it now follows that
            \begin{eqnarray*}
                \HPq(a H) & =&  \HPq(\bp a H c) \\
                &=& \fH_{\bpP, \iq}(\bp a H) + \log c \\
                &=& \conv (\WKHaq \cdot \log \bp a) + \Gamma(\bp P) + \log c\\
                &=& \conv (\WKHaq \cdot \log a ) + \Gamma(P).
            \end{eqnarray*}
             \qed

          We  recall that the relative interior of a convex subset $S$ of a finite dimensional real
        linear space is defined to be the interior of $S$ in its affine span ${\rm aff }(S).$

        \begin{lemma}\label{l: convex-cones}
            Let $V$ be a finite dimensional real linear space and $C\subseteq V$ a closed convex polyhedral subset with non-empty interior. Let $C_i$ {\rm ($i\in \{1, \ldots, n\}$)} be closed convex polyhedral subsets of $C$, of positive codimension. Then the following statements are true.
            \begin{enumerate}
                \itema The complement $C':= C\setminus\cup_{i=1}^nC_i$ is dense in $C$.
                \itemb Let $A$ and $B$ be open subsets of $V$ contained in $C'.$ Then for each $a \in A$ there exists $b \in B$ such that for each $i$ with $C_i\cap [a,b] \neq \emptyset$ the following assertions are valid,
                \begin{enumerate}
                    \item[\rm (1)] $\codim (C_i) = 1;$
                    \item[\rm (2)] $[a,b] \cap C_i$      consists of a single point $p$ which belongs to the  relative interior of $C_i$.
                         Furthermore, if $p\in C_j$ for some  $1 \leq j \leq n,$ then ${\rm aff}(C_j)={\rm aff}(C_i)$.
                \end{enumerate}
            \end{enumerate}
        \end{lemma}

        \begin{proof}
        Standard, and left to the reader.
        \end{proof}

    \appendix
        \section{Proof of Lemma \ref{l: KDV}}\label{a: proof of KDV}

            Finally, we prove Lemma \ref{l: KDV}.

We begin by showing that the result holds for $G$ a complex semi-simple Lie group, connected with trivial center. That proof will be based on the following general lemma, inspired by  \cite[Prop.~1]{chevalley1958}.

              Let $\fh$ be a complex abelian Lie algebra and let $\mathcal{N}$ be the class of complex finite dimensional nilpotent Lie algebras    $\fn$, equipped with a representation of $\fh$ by derivations, such that the following conditions are fulfilled
                \begin{enumerate}
                \itema the representation of $\fh$ in $\fn$ is semi-simple;
                \vspace{-6pt}
                \itemb  all weight spaces of $\fh$ in $\fn$ have complex dimension one.
                \end{enumerate}
                If $\fn$ belongs to the class $\cN,$ we write $\gL(\fn)$ for the set of $\fh$-weights in $\fn.$ If $\gl \in \gL(\fn),$ then the associated weight space is denoted by $\fn_\gl.$

                \begin{lemma}\label{l: weight-spaces-deco}
                Let      $\fn \in \cN$ and let $N$ be the connected, simply-connected Lie group with Lie algebra $\fn$. Let $\gl_1 , \ldots, \gl_m$ be the distinct weights of $\,\fh$ in $\fn.$ Then the map
                $$
                \psi:  (X_1 ,\ldots, X_m )\mapsto \exp X_{1}\cdot\cdots\cdot\exp X_{m}
                 $$
                defines a diffeomorphism
                $$                  \fn_{\lambda_1}\times\ldots\times\fn_{\lambda_m}\;\;{\buildrel\simeq \over \longrightarrow}\;\; N.
                $$
            \end{lemma}
            \begin{proof}
                     We will use induction on ${\rm dim}_{\C}(\fn)$. If $\dim_\C \fn =1$ then $\fn$ is abelian and the result holds trivially.

                Next,      assume that $m > 1$ and assume that the result has been established for $\fn$ with $\dim_\C \fn < m.$  Assume that $\fn \in \cN$ has dimension $m.$

                Denote by $\fn_1$ the center of $\fn$, which is non-trivial. If $\fn_1 = \fn$ then $\fn$ is abelian and the result is trivially true. Thus, we may as well assume that
                $0 \subsetneq \fn_1 \subsetneq \fn.$ In particular, this implies that both $\fn_1$ and $\fn/\fn_1$ have dimensions at most $m -1.$ Put $l := \dim \fn_1.$

                The ideal $\fn_1$ is stable under the action of $\fh$ and it is readily verified that $\fn_1$ and $\fn/\fn_2$ with the natural $\fh$-representations belong to $\cN.$ Furthermore, since all weight spaces are $1$-dimensional, we see that
                $$
                \gL(\fn) = \gL(\fn_1) \sqcup \gL(\fn/\fn_1).
                $$
                We will first prove that $\psi$ is a diffeomorphism under the assumption that
                the $\fh$-weights in $\fn$ are numbered in such a way that
                $$
                \gL(\fn_{1}) = \{\gl_{1} , \ldots, \gl_l \}  \quad{\rm and}\;\; \gL(\fn/\fn_1) =
                \{\gl_{l+1}, \ldots, \gl_{m}\}.
                $$
                Since $N$ is simply-connected, the map $\exp:\fn\to N$ is a diffeomorphism; hence, $N_1:= \exp(\fn_1)$ is the connected  subgroup of $N$ with Lie algebra $\fn_1$.
                In particular, $N_1$ is simply connected as well. Since $\fn_1$ is an ideal, $N/N_1$ has a unique structure of Lie group for which the natural map $N \to N/N_1$ is a Lie group homomorphism.  We now observe that $N\to N/N_1$ is a principal fiber bundle with fiber $N_1$. By standard homotopy theory we have a natural exact sequence                 $$
               \pi_1(N)\to\pi_1(N/N_1)\to \pi_0(N_1).
               $$
               Since $N$ is simply-connected, and $N_1$ connected, we conclude that $N/N_1$ is the simply connected group with Lie algebra $\fn/ \fn_1.$

                By the induction hypothesis, the maps
                \begin{eqnarray*}
                     & \psi_{\fn_1} &:\fn_{\lambda_1}\times\ldots\times\fn_{\lambda_l}\to N_1\\
                     & \psi_{\fn/\fn_1} &:(\fn/\fn_1)_{\lambda_{l+1}}\times\ldots\times(\fn/\fn_1)_{\lambda_m}\to N/N_1
                \end{eqnarray*}
                are diffeomorphisms.
                For every $j\in\{l+1,\ldots,m\}$ the canonical projection $\fn \to \fn/\fn_1$ induces the isomorphisms of weight spaces $\fn_{\lambda_j}\to(\fn/\fn_1)_{\lambda_j}$.
                 Let $\bar \psi: \fn_{\gl_{l+1}} \times \ldots \times \fn_{\gl_m} \to N/N_1$ be defined
                by
                $
                \bar \psi (X_{l+1}, \ldots, X_m) = \exp X_{l+1}
                \cdot \ldots \cdot \exp X_m \cdot N_1.
                $
                Then the following diagram commutes:
                \begin{equation*}\minCDarrowwidth55pt\begin{CD}
                \fn_{\lambda_{l+1}}\times\ldots\times\fn_{\lambda_m} @>\bar{\psi}>> N/N_1\\
                @V\simeq VV @|\\
                (\fn/\fn_1)_{\lambda_{l+1}}\times\ldots\times(\fn/\fn_1)_{\lambda_m} @>\psi_{\fn/\fn_1}>> N/N_1
            \end{CD}\end{equation*}
            From this we infer that $\bar \psi$ is a diffeomorphism.
            We now obtain that the map $\tilde{\psi}:\fn_{\lambda_{l+1}}\times\ldots\times\fn_{\lambda_m}\times N_1\to N$, \[(X_{l+1},\ldots,X_m,n_1)\mapsto (\exp X_{l+1}\cdot\ldots\cdot\exp X_m)n_1,\] is a diffeomorphism onto $N.$ Since
                \[\psi(X_1,\ldots,X_l,X_{l+1},\ldots,X_m)=
                \tilde{\psi}(X_{l+1},\ldots,X_m,\psi_{\fn_1}(X_1,\ldots,X_l))\]
                it follows that $\psi$ is a diffeomorphism as well. Clearly, the above proof works
                for every enumeration of the weights in $\gL(\fn /\fn_1).$ Since the weight spaces
                $(\fn_1)_\gl$ for $\gl \in \gL(\fn_1)$ are all central in $\fn,$ we conclude that
                the result holds for any enumeration of  the weights in $\gL(\fn).$
            \end{proof}

            \begin{corollary}\label{c: complex-KDV}
                Let $G$ be a connected complex semi-simple Lie group and $\fn_B$ the nilpotent radical of a Borel subalgebra $\nb$ of $\fg$. Let $\fh$ be a Cartan subalgebra contained in $\nb$. Let $\fn_1,\ldots,\fn_k$ be linearly independent subalgebras of $\fn_B$, each of which is a direct sum of $\fh$-root spaces, and assume that their direct sum $\fn:=\fn_1\oplus\ldots\oplus\fn_k$ is again a subalgebra. Put $N:=\exp\fn$ and  $N_j:= \exp(\fn_j),$ for $1\leq j \leq k.$

                Then the multiplication map
                \[\mu:N_1\times\ldots\times N_k\to N\] is a diffeomorphism.
            \end{corollary}
            \begin{proof}
                This is an immediate consequence of Lemma \ref{l: weight-spaces-deco}.
            \end{proof}
\medbreak\noindent

            \emph{Proof of Lemma \ref{l: KDV}.\ }   We assume that  $G$ is a real reductive Lie group of the Harish-Chandra class.      Define
            \[
            \fg_1:=[\fg,\fg],
            \]
            the semi-simple part of the Lie algebra of $G.$ Let
           $G_1$ be the analytic subgroup of $G$ with Lie algebra $\fg_1.$ Since the nilpotent radical $N_P$ of $P$ is completely contained in $G_1$, we may assume from the start that $G=G_1$, i.e. $G$ is connected semi-simple with finite center.

            Since $\Ad$ is a finite covering homomorphism from $G$ onto $\Aut(\fg)^\circ,$
            mapping $N$ diffeomorphically onto $\Ad(N),$ whereas $\Aut(\fg)^\circ$ is a connected real
            form of ${\rm Int}(\fg_\C),$ we
              may assume that $G$ is a connected real form of a connected complex semi-simple Lie group   $G_{\C}$  with trivial center. Let $\tau$ be the conjugation on $G_{\C}$, such that
            \[G=(G_{\C}^{\tau})^\circ.\]
            Let $\fg_{\C}$ denote the Lie algebra of $G_{\C}$, then $\fg_{\C}=\fg\oplus i\fg.$
            Note that the complexification $\fn_{P\C}$ of $\fn_P$ equals $\fn_P\oplus i\fn_P$ and that
            $$
            N_P = (N_{P\C})^{\tau}.
            $$

            Take a Cartan subalgebra of $\fg_{\C}$, containing $\fa_{\C}=\fa\oplus i\fa$. It is of the form
            \[\fh_{\C}=\ft_{\C}\oplus\fa_{\C},\]
            where $\ft$ is a maximal abelian subspace      of $\fm:=Z_{\fk}(\fa)$.
            Since $\ft$      centralizes $\fa,$ all $\fa$-root spaces are invariant under $\ad (\ft).$
            This      implies that the subalgebras $\fn_{j\C}:=\fn_j\oplus i\fn_j$ ($j\in\{1,\ldots,k\}$) of $\fn_{P\C}$ are direct sums of $\fh_{\C}$-root spaces. Furthermore, their direct sum  equals $\fn_{\C} =\fn\oplus i\fn,$ hence is a subalgebra. Finally, there exists a Borel subalgebra containing $\fh_\C + \fn_\C.$
            By Corollary \ref{c: complex-KDV}, the multiplication map
            \[
            \mu_{\C}:N_{1\C}\times\ldots\times N_{k\C}\to N_{\C}
            \]
            is a diffeomorphism. It readily follows that $\mu_\C$ restricts to a bijection
           from $( N_{1\C})^\tau \times \cdots \times (N_{k\C})^\tau $ onto $( N_\C)^\tau.
           $
            Since
            \[(N_{\C})^{\tau}=N\quad\textrm{and} \quad (N_{j\C})^{\tau}=N_j\;\; \textrm{ for all }1 \leq j \leq k,
            \]
            it follows that $\mu$ is a bijective embedding from
            $N_1 \times \cdots \times N_k$ onto $N,$ hence a diffeomorphism.

        \section{The case of the group}\label{a: group case}

            Every  semisimple Lie group $G$ can be viewed  as a semi-simple symmetric space
            for the group $G\times G.$ In this section we investigate what our convexity theorem means for this particular example. An independent proof for this case is presented in \cite[Section 3.2.2]{balib14}.

            More generally, let $G$ be real reductive group of the Harish-Chandra class, $\theta$
            a Cartan involution, $K:={G}^{\theta}$ the associated maximal compact subgroup
            and $\fg = \fk \oplus \fp$ the associated Cartan decomposition as in Section \ref{s: introduction}. Let $\fa$  be a maximal abelian subspace of $\fp,$ $A = \exp \fa$ and $\gS(\fg, \fa)$ the associated root system.

            Let
            $$
            G': =G\times G.
            $$
            Then $\Cartan':= \Cartan \times \Cartan$ is a Cartan decomposition of $G'$ with
            associated maximal compact subgroup $K' := K\times K.$  The involution
            $$
            \gs': G' \to G',\;\; (x,y) \mapsto (y,x),
            $$
            commutes with $\Cartan'.$ Its fixed point group $H'$ equals the diagonal in $G \times G$ and  is essentially connected in $G',$ see \cite[Example 2.3.7]{balib14}.

              The associated space $\fp' \cap \fq'$ equals $\{(X,-X)   \setmid X \in \fp\}$ and has
            \[
            \faq':=\{(X,-X) \setmid X\in\fa \}
            \]
            as a maximal abelian subspace. Its root system is given by
            \[
            \Sigma(\fg',\faq')=\Sigma(\fg,\fa)\times\{0\}\cup\{0\}\times\Sigma(\fg,\fa).
            \]
            Finally, $\faq'$ is contained in the maximal abelian subspace $\fa' := \fa \times \fa$ of $\fp'.$ We put $A' := \exp (\fa') = A \times A$. Note that
            the projection map $\pr_q: \fa' \to \faq'$ is given by
            \begin{equation}
            \label{e: pr q in group case}
            \prq(U,V)=(\half (U-V),\half (V-U)).
            \end{equation}

            Let $P$ and $Q$ be minimal parabolic subgroups of $G$ containing $A$, i.e. $P,Q\in\cP(A)$. Then $P\times Q$ is a minimal parabolic subgroup of $G'$ containing $A'$. Moreover, any minimal parabolic subgroup of $G'$ containing $A'$ is of this form.
            The positive system of $\fa'$-roots associated with $P\times Q$ is given by
           \[
            \Sigma(P\times Q):=\Sigma(P)\times\{0\}\cup\{0\}\times\Sigma(Q),
            \]
            where $\Sigma(P)$ and $\Sigma(Q)$ are positive systems for $\Sigma(\fg,\fa)$ corresponding to the minimal parabolic subgroups $P$ and $Q$.

            In the present setting, our   main result, Theorem \ref{t: conv thm new}, tells us that for $a\in A_{q}'$ we have
            \[
            \prq\circ\fH_{P\times Q}(aH')={\rm conv}(W_{K'\cap H'} \cdot \log a)+\Gamma (P\times Q).
            \]

          In order to     determine  the cone $\Gamma(P\times Q)$, we      need to determine      the set
             $
             \gS(P\times Q , \gs'\Cartan')
             $
             of roots $\gamma\in\Sigma(P\times Q)$ for which $\sigma'\theta'\gamma\in\Sigma(P\times Q)$. Let $\gamma=(\alpha,0)$ be such a root. Then $\alpha\in\Sigma(P)$ and $\sigma'\theta'\gamma=(0,-\alpha)$ must be an element of $\{0\}\times\Sigma(Q)$ so that $\ga  \in \gS(P) \cap \gS(\bar Q).$  Likewise,      if $(0, \gb)$ belongs to this set then
             $\gb \in \gS(Q) \cap \gS(\bar P).$ We thus see that
             \[
            \Sigma(P\times Q,\sigma\theta')=(\Sigma(P)\cap\Sigma(\bar{Q}))
            \times\{0\}\cup\{0\}\times(\Sigma(\bar{P})\cap\Sigma(Q)).\]
            Notice that there are no roots $\gamma\in\Sigma(P\times Q)$ for which $\sigma'\theta'\gamma=\gamma.$
            Thus, $\gS(P \times Q)_- = \gS(P\times Q, \gs'\Cartan')$ and we conclude that
\[
\Gamma(P\times Q)=\Gamma_{\faq'} (\gS(P\times Q, \gs'\Cartan')) =
\sum_{\gamma\in\Sigma(P\times Q,\sigma'\theta')}\R_{\geq 0}\,\prq H'_{\gamma}.
\]
            If $\gamma$ is of the form $(\alpha,0)$ then $H'_{\gamma}=(H_{\alpha},0)$ and
            if $\gamma = (0,\ga),$ then $H'_{\gamma}=(0,H_{\alpha})$.  In view of (\ref{e: pr q in group case}) we now obtain
            \begin{eqnarray*}
\lefteqn{\Gamma(P \times Q) = }\\
                 &= & \sum_{\alpha\in\Sigma(P)\cap\Sigma(\bar{Q})}\R_{\geq 0}( {\half H_{\alpha}},-{\half H_{\alpha}})+
\sum_{\alpha\in\Sigma(\bar{P})\cap\Sigma(Q)}\R_{\geq 0}(-{\half H_{\alpha}},{\half H_{\alpha}})\\
                &=& \sum_{\alpha\in\Sigma(P)\cap\Sigma(\bar{Q})}\R_{\geq 0}
({ H_{\alpha}},- {H_{\alpha}})+
\sum_{\alpha\in\Sigma(\bar{P})\cap\Sigma(Q)}\R_{\geq 0}({H_{-\alpha}},-
{H_{-\alpha}})\\
                &=& \sum_{\alpha\in\Sigma(P)\cap\Sigma(\bar{Q})}\R_{\geq 0}({H_{\alpha}},-{H_{\alpha}})+
\sum_{\alpha\in\Sigma(P)\cap\Sigma(\bar{Q})}\R_{\geq 0}({H_{\alpha}},-{H_{\alpha}})\\
                &=&  \sum_{\alpha\in\Sigma(P)\cap\Sigma(\bar{Q})}\R_{\geq 0}(H_{\alpha},-H_{\alpha}).
 \end{eqnarray*}
                 We will identify $\faq'$ with $\fa$ via the map $(X,-X)\mapsto X$. Thus,
            \[\Gamma_{\faq'}(\Sigma(P\times Q))=\Gamma_{\fa}(\Sigma(P)\cap\Sigma(\bar{Q})).
\]

For $Q=P$, the resulting cone is the zero one, and we retrieve the non-linear convexity theorem of Kostant \cite{kostant1973} for the group $G.$  At the other extreme, for $Q = \bar P$, the resulting cone is maximal, and we retrieve the convexity theorem of 
 \cite{ban1986} for the pair $(G',H').$

            Taking $a = e$ we obtain, with the same identification      $\faq' \simeq \fa$,
            \begin{equation}\label{e: a=e-cone}\prq\circ\fH_{P\times Q}(H')=\Gamma_{\fa}(\Sigma(P)\cap\Sigma(\bar{Q}))
\end{equation}
            On the other hand,
\begin{eqnarray*}
\prq\circ\fH_{P\times Q}(H')  &=&
\prq\circ\fH_{P\times Q}({\rm diag}(G\times G))\\
            &=& \prq(\{(\fH_{P}(g),\fH_{Q}(g))\setmid g\in G\})\\
            &=&  \prq(\{(\fH_{P}(kan_p),\fH_{Q}(kan_p))\setmid k\in K,a\in A, n_p\in N_{P}\})\\
            &=& \prq(\{(\log a,\fH_{Q}(an_pa^{-1})+\log a )\setmid a\in A, n_p\in N_{P}\})\\
            &=& \prq(\{(\log a,\fH_{Q}(n_p)+\log a)\setmid a\in A, n_p\in N_{P}\})\\
            &=& \{(-\half \fH_{Q}(n_p),\half \fH_{Q}(n_p))\setmid n_p\in N_{P}\}.
\end{eqnarray*}

      Using the same identification $\faq'\simeq \fa$ as above, we conclude that
 \begin{eqnarray*}
            \prq\circ\fH_{P\times Q}(H')&=& -\frac{1}{2}\fH_{Q}(N_{P})\\
            &=& -\half \fH_{Q}((N_{P}\cap\bar{N}_{Q})(N_{P}\cap N_{Q}))\\
            &=& -\half \fH_{Q}(N_{P}\cap\bar{N}_{Q}).
            \end{eqnarray*}
            Thus, by equation (\ref{e: a=e-cone}), we obtain that
            \[-\frac{1}{2}\fH_{Q}(N_{P}\cap\bar{N}_{Q})=\Gamma_{\fa}(\Sigma(P)\cap\Sigma(\bar{Q})),\]
            which is equivalent to
            \[
\fH_{Q}(N_{P}\cap\bar{N}_{Q})=\Gamma_{\fa}(\Sigma(\bar{P})\cap\Sigma(Q)).\]
            Thus we retrieve the identity of Lemma \ref{l: Gindikin-Karpelevic}, which, of course,
            was used in the proof of our main theorem.

\bibliographystyle{amsplain}
\providecommand{\bysame}{\leavevmode\hbox to3em{\hrulefill}\thinspace}
\providecommand{\MR}{\relax\ifhmode\unskip\space\fi MR }
\providecommand{\MRhref}[2]{%
  \href{http://www.ams.org/mathscinet-getitem?mr=#1}{#2}
}
\providecommand{\href}[2]{#2}

\def\adritem#1{\hbox{\small #1}}
\def\distance{\hbox{\hspace{3.5cm}}}
\def\apetail{@}
\def\addBan{\vbox{
\adritem{E.~P.~van den Ban}
\adritem{Mathematical Institute}
\adritem{Utrecht University}
\adritem{PO Box 80 010}
\adritem{3508 TA Utrecht}
\adritem{The Netherlands}
\adritem{E-mail: E.P.vandenBan{\apetail}uu.nl}
}
}
\def\addBalibanu{\vbox{
\adritem{D. B\u alibanu}
\adritem{Mathematical Institute}
\adritem{Utrecht University}
\adritem{PO Box 80 010}
\adritem{3508 TA Utrecht}
\adritem{The Netherlands}
\adritem{E-mail: danabalibanu{\apetail}gmail.com}
}
}
\mbox{}
\vfill
\hbox{\vbox{\addBalibanu}\vbox{\distance}\vbox{\addBan}}

\end{document}